\documentclass[11pt,a4paper]{article}

\usepackage[margin=1.0in]{geometry}
\usepackage[bookmarks=false]{hyperref}
\usepackage{amsmath}
\usepackage{amsfonts}
\usepackage{algorithm}
\usepackage{algorithmic}
\usepackage{amsthm}
\usepackage{graphicx}
\usepackage{appendix}
\usepackage{subcaption}

\newtheorem{theorem} {Theorem}
\newtheorem{lemma} {Lemma}
\newtheorem{definition} {Definition}
\newtheorem{corollary} {Corollary}

\newtheorem{assumption} {Assumption}

\newtheorem{remark} {Remark}

\def\x{{\mathbf{x}}}

\def\u{{\mathbf{u}}}
\def\n{{\mathbf{n}}}

\def\z{{\mathbf{z}}}

\def\y{{\mathbf{y}}}

\def\X{{\mathbf{X}}}
\def\Y{{\mathbf{Y}}}

\def\A{{\mathbf{A}}}
\def\D{{\mathbf{D}}}
\def\M{{\mathbf{M}}}

\def\I{{\mathbf{I}}}
\def\B{{\mathbf{B}}}

\def\V{{\mathbf{V}}}
\def\Z{{\mathbf{Z}}}
\def\W{{\mathbf{W}}}
\def\U{{\mathbf{U}}}

\DeclareMathOperator*{\argmin}{arg\,min}

\newcommand{\mS}{\mathcal{S}}
\newcommand{\mX}{\mathcal{X}}

\newcommand{\mbS}{\mathbb{S}}

\newcommand{\E}{\mathbb{E}}

\newcommand{\trace}{\textrm{Tr}}
\newcommand{\rank}{\textrm{rank}}

\newcommand{\reals}{\mathbb{R}}
\newcommand{\var}{\textbf{Var}}
\newcommand{\diag}{\textrm{diag}}
\newcommand{\breg}{B}

\title{On the Efficient Implementation of the \\ Matrix Exponentiated Gradient Algorithm \\for  Low-Rank Matrix Optimization}
\author{Dan Garber \\ {\small Technion - Israel Institute of Technology}\\ {\small \texttt{dangar@technion.ac.il}}
\and
Atara Kaplan \\  {\small Technion - Israel Institute of Technology} \\ {\small  \texttt{ataragold@campus.technion.ac.il}}
}

\date{}

\begin{document} 

\maketitle

\begin{abstract}
Convex optimization over the spectrahedron, i.e., the set of all real $n\times n$ positive semidefinite matrices with unit trace,  has important applications in machine learning, signal processing and statistics, mainly as a convex relaxation for optimization  problems with low-rank matrices. It is also one of the most prominent examples in the theory of first-order methods for convex optimization in which non-Euclidean methods can be significantly preferable to their Euclidean counterparts. In particular, the desirable choice is the Matrix Exponentiated Gradient (MEG) method which is based on the  Bregman distance induced by the (negative) von Neumann entropy. Unfortunately, implementing MEG requires  a full SVD computation on each iteration, which is not scalable to high-dimensional problems.

In this work we propose an efficient implementations of MEG, both with deterministic and stochastic gradients, which are tailored  for optimization with low-rank matrices, and only use a single low-rank SVD computation on each iteration. We also provide efficiently-computable certificates for the correct convergence of our methods. Mainly, we prove that under a strict complementarity condition, the suggested methods converge from a ``warm-start" initialization with similar rates to their full-SVD-based counterparts. 
Finally, we bring empirical experiments which both support our theoretical findings and demonstrate the practical appeal of our methods.

%Mirror-Descent-type methods based on the von Neumann entropy that only use low-rank SVD computations. We prove that under a strict complementarity assumption, which is related to the spectral gap in the gradient vector at optimal points, our proposed method converges from a ''warm start'' initialization with similar rates to the standard method. Our results apply both to deterministic smooth minimisation and to stochastic smooth minimization.

%While the Mirror-Decent algorithm can be used to reduce the dependence on the dimension in the convergence rate, these problems can still be challenging when solved in large-scale as they require full-rank eigen-decomposition computations in each iteration. 
%In this paper we consider a simple low-rank variant of the traditional Mirror-Decent algorithm. We present a sufficient condition that guarantees the algorithm will converge to the optimal solution, which is merely dependent on the first few eigenvalues of the iterates. In addition, we show that by using a warm-start initialization for the algorithm, this convergence condition is indeed met for smooth convex minimization, without any significant penalty to the convergence rate. Moreover, we show that for stochastic convex minimization, the condition is met with high probability. 

\end{abstract}

\section{Introduction}

In this paper we consider the following optimization problem:
\begin{align} \label{eq:Model}
\min_{\X\in\mathcal{S}_n}{f(\X)},
\end{align}
where $\mathcal{S}_n=\{\X\in\mathbb{S}^n\ \vert\ \trace(\X)=1,\ \X\succeq0\}$ is the spectrahedron in $\mathbb{S}^n$ --- the space
of $n\times n$ real symmetric matrices, and $f:\reals^{n\times n}\rightarrow\reals$ is convex. For the sake of simplicity, throughout this paper we assume the gradient vector of $f(\cdot)$ is a symmetric matrix over $\mbS^n$, i.e., $\nabla{}f(\X)\in\mbS^n$ for all $\X\in\mbS^n$\footnote{If $\nabla{}f(\X)$ is not a symmetric matrix then, throughout this paper, it could always be replaced with the symmetric matrix $(\nabla{}f(\X)+\nabla{}f(\X)^{\top})/2$. In particular, since we consider optimization over symmetric matrices, this transformation does not change the standard matrix inner-product $\langle{\Y,\nabla{}f(\X)}\rangle$ for any $\X,\Y\in\mbS^n$.}. 
We also assume $f(\cdot)$ is $\beta$-smooth over $\mS_n$ in the following typical sense:
$\Vert \nabla f(\X)-\nabla f(\Y)\Vert_2\le \beta\Vert \X-\Y\Vert_*$ for all $\X,\Y\in\mS_n$, where $\Vert{\cdot}\Vert_2$ is the spectral norm for matrices  (largest singular value) and $\Vert{\cdot}\Vert_*$ is the nuclear norm for matrices (sum of singular values).

Problem \eqref{eq:Model} is used to model numerous deterministic and stochastic low-rank matrix recovery problems. These problems have many important modern applications in machine learning, statistics, signal processing, which include, among others, the well known tasks of matrix completion \cite{MatrixCompletion1,MatrixCompletion2,MatrixCompletion3,MatrixCompletion4}, robust PCA \cite{robustPCA1,robustPCA2,robustPCA3,robustPCA4,robustPCA5}, and phase retrieval \cite{phaseRetrieval1,phaseRetrieval2,phaseRetrieval3}.

We consider Problem \eqref{eq:Model} both in the deterministic first-order setting where the full gradient of $f(\cdot)$ is available, and in the stochastic first-order setting where only unbiased estimates for the gradient vector are available. In both settings, in terms of gradient complexity, the methods  of choice for solving Problem \eqref{eq:Model} in large scale are the non-Euclidean proximal gradient methods that are based on Bregman distances w.r.t. the (negative) von Neumann entropy (entropy w.r.t. the eigenvalues) \cite{beckMD,sebastienMD,beckOptimizationBook, BregmanMatrices}.  This is because for Problem \eqref{eq:Model},  these methods are almost dimension independent \cite{sebastienMD}. In particular, the complexities of Euclidean methods  depend on the Lipschitz parameter of the gradient in Euclidean norm when using full gradients, or the Euclidean norm of the stochastic gradients and their variance in Euclidean norm when using stochastic gradients. On the other hand, proximal gradient methods based on the von Neumann entropy, measure these quantities w.r.t. to the spectral norm. This can often lead to convergence rates with significantly improved dependence on the dimension --- a crucial consideration in high dimensional settings. For example, in the stochastic setting this can improve a factor of $\sqrt{n}$ in the rate of Euclidean stochastic gradient descent, to only $\sqrt{\log{n}}$ for stochastic mirror descent with Bregman distances w.r.t. the von Neumann entropy, or a factor of $n$ in the rate of Euclidean projected gradient descent for smooth minimization, to only $\log{n}$ for proximal gradient with Bregman distancess w.r.t. the von Neumann entropy \cite{sebastienMD}.

%\textcolor{red}{As a demonstration of the significant effect the type of norm can have on the smoothness parameter, consider for example the function
%$f(\X)=\frac{1}{2}\Vert \mA(\X)-\mathbf{b}\Vert_2^2$,
%where $\mathcal{A}(\X)=\left(\begin{array}{c}
%\trace(\A_{1}^{\top}\X),\ 
%\trace(\A_{2}^{\top}\X),\ 
%\ldots,\ 
%\trace(\A_{m}^{\top}\X) 
%\end{array}\right)^{\top}$, for $\A_{1},...,\A_{m}\in\mathbb{R}^{n\times n}$. This function is $\sum_{i=1}^{m}\Vert\A\Vert_2^2$-smooth with respect to the nuclear norm, because it satisfies that
%\begin{align*}
%\Vert \nabla f(\X)-\nabla f(\Y)\Vert_2 & =\left\Vert\sum_{i=1}^{m} (\trace(\A_i^{\top}\X)\A_i^{\top}-\trace(\A_i^{\top}\Y)\A_i^{\top})\right\Vert_2 \le \sum_{i=1}^{m}\Vert\A_i\Vert_2 \cdot \vert\trace(\A_i^{\top}(\X-\Y))\vert 
%\\ & \le \sum_{i=1}^{m}\Vert\A_i\Vert_2^2\Vert \X-\Y\Vert_*,
%\end{align*}
%whereas, it is $\sum_{i=1}^{m}\Vert\A\Vert_F^2$-smooth with respect to the Frobenius norm, because it satisfies that
%\begin{align*}
%\Vert \nabla f(\X)-\nabla f(\Y)\Vert_F & =\left\Vert\sum_{i=1}^{m} (\trace(\A_i^{\top}\X)\A_i^{\top}-\trace(\A_i^{\top}\Y)\A_i^{\top})\right\Vert_F \le \sum_{i=1}^{m}\Vert\A_i\Vert_F \cdot \vert\trace(\A_i^{\top}(\X-\Y))\vert 
%\\ & \le \sum_{i=1}^{m}\Vert\A_i\Vert_F^2\Vert \X-\Y\Vert_F.
%\end{align*}}

The family of proximal gradient methods with the von Neumann entropy for Problem \eqref{eq:Model} is also referred to in the literature as the Matrix Exponentiated Gradient (MEG) method (see for instance \cite{BregmanMatrices}), which we also adopt in this work. See also Algorithm \ref{alg:exactMEG} in the sequel for a full description of the method (both with deterministic and stochastic gradients).

Despite the improved gradient complexity of the MEG method over its Euclidean counterparts, a clear caveat is the high computational cost of executing each iteration, since each iteration requires in worst case a full-rank SVD computation of an $n\times n$ matrix  (to compute matrix logarithm and matrix exponential) which amounts to $O(n^3)$ runtime and is clearly not practical for high dimensional problems\footnote{this runtime could in principle be improved to $O(n^{\omega})$, where $\omega$ is the best exponent for fast matrix multiplication algorithms (currently $\omega < 2.37286$ \cite{alman2021refined}), however, to the best of our knowledge, such algorithms are not considered practical and are not in common use.}. 

While the high computational cost of high-rank SVD computations is also present in Euclidean proximal gradient methods for Problem \eqref{eq:Model}, in the recent works \cite{garberLowRankSmooth, garberStochasticLowRank} it was established that under a natural strict complementarity condition (see discussion in the sequel) and at least in a certain ball around an optimal solution, Euclidean proximal gradient methods (both deterministic and stochastic) which only apply low-rank SVD computations (with rank that matches that of the optimal solution), provably converge to an optimal solution with rates similar to their full-SVD-based counterparts, resulting in computationally efficient methods, at least from a ``warm-start" initialization. This was also supported by empirical evidence brought in   \cite{garberLowRankSmooth, garberStochasticLowRank} which demonstrated the correct convergence of these methods using only low-rank computations in practice.

Thus, given that von Neumann entropy-based gradient methods are superior (in terms of gradient complexity) to their Euclidean counterparts for Problem \eqref{eq:Model}, and given the progress in \cite{garberLowRankSmooth, garberStochasticLowRank} on SVD-efficient Euclidean gradient methods, it is natural to ask whether similar results could also be obtained for the matrix exponentiated gradient method, and by that, to obtain the best of both worlds.

In this paper we provide a positive resolution to this question. We propose a variant of the MEG method which utilizes only low-rank SVD computations, and prove that under the strict complementary condition, it converges from a ``warm-start" initialization to an optimal solution, with similar rates to its full-SVD-based counterparts ($O(1/t)$ for deterministic gradients and $O(1/\sqrt{t})$ for stochastic gradients). 

At a high-level, our low-rank MEG variant works by replacing all lower eigenvalues in the matrix obtained from a standard MEG update, with some small value (the same for all lower eigenvalues), which in turn requires to maintain and decompose only low-rank matrices. This modification of the MEG update arises from an intuition that in the proximity of a low-rank optimal solution, the lower eigenvalues of the iterates are expected to decay rapidly towards zero (while remaining positive throughout though). Our main technical contribution is in showing, via rather involved analysis, that such a modification, under a strict complementarity condition and in certain proximity of the optimal solution, indeed results in sufficiently-bounded approximation errors which yields convergence rates similar to those of the original method. 
Importantly, the techniques used for Euclidean gradient methods in \cite{garberLowRankSmooth, garberStochasticLowRank} are not easily extendable to our setting, since the MEG update has a very different structure than its Euclidean counterpart. In particular, while the Euclidean projected gradient update thresholds lower eigenvalues to zero and thus naturally promotes low-rank iterates, the MEG update almost always produces full rank matrices, which  requires both to modify the algorithm itself and a substantially more involved analysis (a more detailed discussion is brought in Section \ref{sec:challenges}).
  
%This is obtained by a careful and rather involved analysis of the MEG updates when all lower eigenvalues are thresholded  

%\textcolor{red}{Our low-rank MEG step replaces all low eigenvalues of the iterates with some small fixed value. The motivation for this step arises from the fact that the optimal solution is low-rank, and hence, the small eigenvalues of the iterates are expected to decay rapidly towards zero while remaining positive throughout.} 
Aside from our local convergence guarantees  we also suggest a recipe for computing certificates which can validate the correct convergence of our proposed low-rank MEG variant throughout the run. Finally, we present numerical simulations which support both our theory and overall methodology of incorporating low-rank updates to MEG and computing certificates for convergence.

As discussed, our provable local convergence guarantees hold under a strict complementarity condition which is also central to the analyses in \cite{garberLowRankSmooth, garberStochasticLowRank}, and also in the recent related work \cite{spectralFrankWolfe}. This condition is given in the following Assumption \ref{ass:strictcomp}. While, similarly to \cite{garberLowRankSmooth, garberStochasticLowRank}, we state this condition in terms of a certain eigen-gap in the gradient vector at optimal solutions which is natural for our analysis, in \cite{spectralFrankWolfe} it was showed that it is indeed equivalent to strict complementarity for Problem \eqref{eq:Model}.\footnote{This means that there exists a corresponding optimal solution to the dual problem $\Z^*\succeq 0$, such that $\rank(\X^*) + \rank(\Z^*)=n$.}

\begin{assumption}[strict complementarity]\label{ass:strictcomp}
An optimal solution $\X^*$  of rank $r^*$ for Problem \eqref{eq:Model} satisfies the strict complementarity assumption with parameter $\delta >0$ if $\lambda_{n-r^*}(\nabla f(\X^*))-\lambda_{n}(\nabla f(\X^*)) \geq \delta$.
\end{assumption}

Importantly, a positive eigen-gap between $\lambda_{n-r}(\nabla{}f(\X^*))$ and $\lambda_n(\nabla{}f(\X^*))$ for some optimal solution $\X^*$, as considered in Assumption \ref{ass:strictcomp}, is a sufficient condition (though not necessary)  for  $\X^*$ to have rank at most $r$. This is captured by the following lemma, which in particular implies that the rank of any optimal solution cannot exceed the algebraic multiplicity of the smallest eigenvalue in the corresponding gradient vector.

\begin{lemma}[Lemma 7 in \cite{garberLowRankSmooth}] \label{lemma:opt-grad}
Let $\X^*\in\mathcal{X^*}$ be any optimal solution, such that $\rank(\X^*)=r^*$, and write its eigen-decomposition as $\X^*=\sum_{i=1}^{r^*}{\lambda_i \mathbf{v}_i \mathbf{v}_i^{\top}}$. Then, the gradient vector $\nabla f(\X^*)$ admits an eigen-decomposition such that the set of vectors $\{\mathbf{v}_i\}_{i=1}^{r^*}$ is a set of top eigen-vectors of $(-\nabla f(\X^*))$ which corresponds to the eigenvalue $\lambda_1(-\nabla f(\X^*))=-\lambda_n(\nabla f(\X^*))$. 
\end{lemma}

A motivation for the plausibility of Assumption \ref{ass:strictcomp} is that it implies the robustness of the low-rank of optimal solutions to Problem \eqref{eq:Model}, to small perturbations in the problem's parameters. See detailed discussions in \cite{garberLowRankSmooth, garberStochasticLowRank, spectralFrankWolfe}. This assumption is also empirically supported by experiments conducted in \cite{garberLowRankSmooth, garberStochasticLowRank, spectralFrankWolfe} and also in this current work, see Section \ref{section:experiments}. We note that strict complementarity has also played an instrumental role in several recent works which used it to prove linear convergence rates for several first-order methods, without strong convexity \cite{zhou2017unified, drusvyatskiy2018error, garberRankOne, spectralFrankWolfe}.

The strict complementarity assumption is central to our theoretical analysis. It revolves around showing that our low-rank MEG method approximates sufficiently well  the steps of its exact full-rank counterpart. Hence, we can derive nearly matching convergence rates using only low-rank computations. Since, as evident from Assumption \ref{ass:strictcomp}, strict complementarity is a local property of a certain optimal solution, we naturally also require a ``warm-start'' initialization assumption so that our method starts in a region in which we can exploit this property. While we do not have a formal argument regarding the necessity of either assumption, in Section \ref{section:experiments} we bring empirical evidence that both demonstrates the good performance of our method when Assumption \ref{ass:strictcomp} indeed seems to hold, and that when strict complementarity does not hold, our method may exhibit poor convergence.

Problem \eqref{eq:Model} (and other close variants of) has received significant attention within the optimization, statistics, and machine learning communities in recent years, with many attempts to provide efficient algorithms for large scale instances under a variety of assumptions. Our interest in this work is to study efficient implementations of classical proximal gradient methods based on the von Neumann entropy since, at least in terms of gradient complexity, these are considered the methods of choice among first-order methods for optimization over the spectrahedron, as discussed above. We refer the interested reader to \cite{garberLowRankSmooth, garberStochasticLowRank, spectralFrankWolfe} for comprehensive discussions about alternative methods/approaches and more related work. 

We note that two recent works \cite{AllenLi2017-FTCL, Carmon2019ARS} have also considered certain approximations of MEG-type updates, which rely on a very different approach of applying low-rank randomized sketching techniques. These works provide efficient MEG variants for \textit{online linear optimization} over the spectrahedron (also known as Matrix Multiplicative Weights), which can in turn be used to solve certain semidefintie programs with affine constraints (SDP) \cite{Carmon2019ARS}. While these SDP algorithms do not rely on additional assumptions, such as our Assumption \ref{ass:strictcomp} or a ``warm-start'' initialization, they are not applicable to general smooth convex objectives as we consider in this work, and they generally do not produce low-rank solutions even when a low-rank solution exits (they typically produce a solution with rank that scales with $1/\epsilon^2$, where $\epsilon$ is the target accuracy, see \cite{Carmon2019ARS}). Additional related works on fast SDP solvers include nearly-linear runtime algorithms for solving positive SDPs \cite{peng2012faster, allen2016using}.
 Another recent relevant work is \cite{scalableSDP}, which combines matrix sketching  techniques with a primal-dual method to solve certain semidefinite programs using a computationally-efficient low-rank representation with minimal storage requirements.

%We note that two recent works \cite{AllenLi2017-FTCL, Carmon2019ARS} have also considered certain approximations of MEG-type updates, but for the very different problem of \textit{online linear optimization} over the spectrahderon. These works use a very different approach than ours and rely on advanced low-rank randomized sketching techniques. Interestingly, while their results do not rely on additional assumptions, such as our Assumption \ref{ass:strictcomp}, and do not require  ``warm-start'' initialization, they still require to store in memory potentially high-rank matrices, as opposed to our method, which (aside of gradient computations) only requires to store (factorized) low-rank matrices. Another recent relevant work is \cite{scalableSDP}, which combines matrix sketching  techniques with a primal-dual method to solve certain semidefinite programs using a computationally-efficient low-rank representation with minimal storage requirements.

The rest of this paper is organized as follows.
\begin{itemize}
\item
In Section \ref{sec:perm} we provide preliminaries including a fully detailed review of the matrix exponentiated gradient method. Importantly, in this section we also discuss the fundamental challenges of extending the results in \cite{garberLowRankSmooth, garberStochasticLowRank} which consider a Euclidean setup, to our von Neumann entropy-based Bregman distance setup.

\item
In Section \ref{sec:lowrankMEG} we present in full detail our low-rank MEG variant and give basic convergence results. These convergence results depend on how well the low-rank sequences approximates the steps of the standard MEG method. In this section we also present an efficient way to compute certificates that can verify that these approximations are well-behaved and by that certify the correct convergence of the method.

\item
In Section \ref{sec:main_tools} we present our main technical tool towards proving our main results: that in a certain ball around an optimal low-rank solution, under the strict complementarity assumption, our low-rank MEG variant indeed approximates sufficiently well the steps of the standard MEG method.

\item
In Section \ref{section:smoothCase} we prove our first main result:  local convergence of the low-rank MEG method with \textit{deterministic} updates with rate $O(1/t)$, under strict complementarity.

\item
In Section \ref{section:stochasticCase} we prove our second main result:  local convergence of the low-rank MEG method with \textit{stochastic} updates with rate $O(1/\sqrt{t})$, under strict complementarity.

\item
In Section \ref{section:experiments} we present numerical experiments in support of our theoretical findings. These experiments support the plausibility of Assumption \ref{ass:strictcomp}, demonstrate the efficient convergence of our low-rank MEG method under Assumption \ref{ass:strictcomp}, and finally, suggest that Assumption \ref{ass:strictcomp} might be necessary for the efficient and reliable convergence of our low-rank MEG method.
\end{itemize}

\section{Preliminaries}\label{sec:perm}

\subsection{Notation}

Throughout this work we use the following notations. 
For real matrices we let $\Vert\cdot\Vert_2$ denote the spectral norm (i.e., the largest singular value), we let $\Vert\cdot\Vert_*$ denote the nuclear norm (i.e., the sum of all singular values),
and we let we let $\Vert\cdot\Vert_F$ denote the Frobenius norm.
For any $\X,\Y\in\mathbb{S}^{n}$, we denote the standard matrix inner product as $\langle\X,\Y\rangle=\X\bullet\Y=\trace(\X\Y)$.
For a real symmetric matrix $\X\in\mathbb{S}^{n}$, we let
$\lambda_i(\X)$ denote its $i$th largest eigenvalue. 
We denote $\mathbb{S}^n_{++}$ to be the set of all symmetric $n\times n$ positive-definite matrices.
For a matrix $\X\in\mathbb{S}^n_{++}$ with an eigen-decomposition $\X=\V \mathrm{diag}(\lambda_1,\ldots,\lambda_n)\V^{\top}$ we define the matrix logarithm as $\log(\X)=\V \mathrm{diag}(\log(\lambda_1),\ldots,\log(\lambda_n))\V^{\top}$ and the matrix exponential as $\exp(\X)=\V \mathrm{diag}(\exp(\lambda_1),\ldots,\exp(\lambda_n))\V^{\top}$. 

\subsection{Bregman distances, the von Neumann entropy and the Matrix Exponentiated Gradient algorithm}

\begin{definition}[Bregman distance]
Let $\omega$ be a real-valued and proper function over a nonempty, closed and convex subset of the parameter domain that is continuously differentiable and $\alpha$-strongly convex w.r.t. some norm, for some $\alpha > 0$.
Then the Bregman distance is defined by: $B_{\omega}(\x,\y)=\omega(\x)-\omega(\y)-\left<\x-\y,\nabla\omega(\y)\right>$.
\end{definition}

For the problem under consideration in this paper, Problem \eqref{eq:Model}, where the parameter domain is the set of symmetric positive definite matrices, the standard strongly convex function that is considered is based on the von Neumann entropy, and is given by 
\begin{align*}% \label{def:vonNeumannEntropy}
\omega(\X)=\trace(\X\log(\X)-\X).
\end{align*}

The Bregman distance corresponding to the von Neumann entropy, has the form of 
\begin{align*}% \label{def:bregmanNotInSpectrahedron}
B_{\omega}(\X,\Y)& = \trace(\X\log(\X)-\X\log(\Y)-\X+\Y).
\end{align*}

In this paper, we are interested in matrices in $\mathcal{S}_n$, for which the trace is equal to $1$. Therefore, for any $\X,\Y\in\mathcal{S}_n$, the Bregman distance reduces to
\begin{align} \label{def:bregmanInSpectrahedron}
B_{\omega}(\X,\Y)& = \trace(\X\log(\X)-\X\log(\Y)).
\end{align}
In order for this definition to include symmetric positive semi-definite matrices, we use the convention $0\log(0):=0$.

For simplicity, from now on we will denote the Bregman distance corresponding to the von Neumann entropy as $\breg(\X,\Y):=B_{\omega}(\X,\Y)$.

It is known that  the von Neumann entropy is $1$-strongly convex with respect to the nuclear-norm over $\mS_n$ (see for instance \cite{strongconvexity}). This implies that
\begin{align} \label{ineq:strongConvexityOfBregmanDistance}
\breg(\X,\Y)\ge\frac{1}{2}\Vert \X-\Y\Vert_*^2.
\end{align}

Another important property of Bregman distances is the following three point identity (see \cite{beckMD}):
\begin{align} \label{lemma:threePointLemma}
B_{\omega}(\X,\Y)+B_{\omega}(\Y,\Z)-B_{\omega}(\X,\Z)=\langle \nabla\omega(\Z)-\nabla\omega(\Y),\X-\Y\rangle.
\end{align}

Bregman distances are central to Mirror-Decent algorithms, where the update step for some $\Z\in\mathcal{S}_n$ can be written, in our case, as:
\begin{align} \label{eq:MDGeneralUpdate}
 \Z_{+}&=\argmin_{\X\in\mathcal{S}_n}\left\{\langle\eta\nabla f(\Z),\X\rangle+\breg(\X,\Z)\right\} \nonumber \\
 &=\argmin_{\X\in\mathcal{S}_n}\left\{\langle\eta\nabla f(\Z)-\nabla\omega(\Z),\X\rangle+\omega(\X)\right\},
\end{align}
where $\eta > 0$ is the step-size.

In the case of the von Neumann entropy, this update is equivalent to an update of the following form (see \cite{BregmanMatrices}):
\begin{align} \label{eq:MDvonNeumannUpdate}
\Z_{+}=\frac{1}{b}\exp(\log(\Z)-\eta\nabla f(\Z)),
\end{align} 
where $b:=\trace(\exp(\log(\Z)-\eta\nabla f(\Z)))$.

For stochastic optimization, the gradient is replaced with an unbiased estimator of the gradient, $\widehat{\nabla}\in\mathbb{S}^n$, resulting in the following update:
  \begin{align} \label{eq:MDvonNeumannStochasticUpdate}
\Z_{+}=\frac{1}{b}\exp(\log(\Z)-\eta\widehat{\nabla}),
\end{align} 
where $b:=\trace(\exp(\log(\Z)-\eta\widehat{\nabla}))$.

The update rules \eqref{eq:MDvonNeumannUpdate} and \eqref{eq:MDvonNeumannStochasticUpdate} are simply the Matrix Exponentiated Gradient (MEG) method and its stochastic counterpart \cite{BregmanMatrices}, which are given in complete form in Algorithm \ref{alg:exactMEG}.

\begin{algorithm}
	\caption{Matrix Exponentiated Gradient for Problem \eqref{eq:Model}}\label{alg:MEG}
	\label{alg:exactMEG}
	\begin{algorithmic}
		\STATE \textbf{Input:} $\{\eta_t\}_{t\geq 1}$ - sequence of (positive) step-sizes
		\STATE $\Z_1$ - an arbitrary point in $\mS_{n}\cap\mathbb{S}^n_{++}$
		\FOR{$t = 1,2,...$} 
            \STATE \ $ \Y_t=\exp(\log (\Z_{t})-\eta_{t}\nabla f(\Z_{t}))$ (for deterministic optimization)
            \\ \OR
            \STATE \ $ \Y_t=\exp(\log (\Z_{t})-\eta_{t}\widehat{\nabla}_t)$ (for stochastic optimization)
			\STATE $b_{t}=\trace(\Y_t)$       
			\STATE $ \Z_{t+1} = \Y_{t}/b_t$
        \ENDFOR
	\end{algorithmic}
\end{algorithm}

The following result is standard and states the convergence rates of Algorithm \ref{alg:exactMEG} with either full or stochastic gradients (see for example \cite{beckOptimizationBook,sebastienMD}).
\begin{theorem}[Convergence of Algorithm \ref{alg:exactMEG}]
When used with deterministic gradients and fixed step-size $\eta = 1/\beta$, the iterates $\{\X_t\}_{t\geq 1}$ of Algorithm \ref{alg:exactMEG} satisfy: $f(\X_t) -f^* = O\left({\frac{\beta\log{n}}{t}}\right)$ for all $t\geq 1$. Alternatively, when used with stochastic gradients and step-size $\eta_t = \frac{\sqrt{\log{n}}}{G\sqrt{t}}$, the ergodic sequence $\{\bar{\X}_t\}_{t\geq 1}$, where $\bar{\X}_t := \frac{1}{t}\sum_{i=1}^t\X_t$, satisfy: $\E[f(\bar{\X}_t) - f^*] = O\left({\frac{G\sqrt{\log{n}}}{\sqrt{t}}}\right)$ for all $t\geq 1$, where $G \geq \max_{t}\Vert{\widehat{\nabla}_t}\Vert_2$.
\end{theorem}

It is important to note that from a computational perspective, the most expensive step in Algorithm \ref{alg:exactMEG} is the computation of the matrix exponential in the update of the matrix variable $\Y_t$ on each iteration. Computing a matrix exponential requires in worst case a full-rank SVD of the input matrix.

\subsection{The challenge of low-rank MEG updates}\label{sec:challenges}

Since our work is mostly inspired by the recent works \cite{garberLowRankSmooth, garberStochasticLowRank} which developed similar local convergence results for Problem \eqref{eq:Model} using Euclidean gradient methods that only apply low-rank SVD computations, we first discuss the differences between the Euclidean and our non-Euclidean Bregman setups, and we emphasize and clarify the difficulty of extending the results \cite{garberLowRankSmooth, garberStochasticLowRank} to the von Neumann entropy-based Bregman setup.

We first recall the Euclidean projected gradient mapping for the spectrahedron $\mathcal{S}_n$, given by $\Pi_{\mathcal{S}_n}[\X-\eta\nabla{}f(\X)]$ for some $\X\in\mathcal{S}_n$, where $\Pi_{\mathcal{S}_n}[\cdot]$ denotes the Euclidean projection onto $\mathcal{S}_n$. Let $\Y=\sum_{i=1}^n\lambda_i v_i v_i^{\top}$ denote the eigen-decomposition of the matrix $\Y=\X-\eta\nabla{}f(\X)$. Its Euclidean projection onto the spectrahedron is given by
\begin{align}\label{eq:euclidProj}
\Pi_{\mathcal{S}_n}[\Y]=\sum_{i=1}^n\max\{0,\lambda_i-\lambda\} v_i v_i^{\top},
\end{align}
where $\lambda\in\reals$ is the unique scalar satisfying  $\sum_{i=1}^n\max\{0,\lambda_i-\lambda\}=1$. 

We can see that the Euclidean projection has a thresholding effect on the eigenvalues, i.e., all eigenvalues below or equal to $\lambda$ become zero after the projection operation. This unique property is central to the analysis in \cite{garberLowRankSmooth, garberStochasticLowRank} and is used to establish that in the proximity of an optimal solution which satisfies an eigen-gap assumption such as the strict complementarity condition (Assumption \ref{ass:strictcomp}), the Euclidean projected gradient mapping always results in a low-rank matrix, and thus, as evident from \eqref{eq:euclidProj}, can be computed exactly using only a low-rank SVD computation.

This unique property and consequence of the Euclidean projected gradient mapping does not hold anymore for the von Neumann entropy-based MEG mapping given in \eqref{eq:MDvonNeumannUpdate}. In particular we see that almost everywhere in $\mS_n$ (except for the points for which this mapping results in a low-rank matrix, such as low-rank optimal solutions), due to the matrix logarithm and exponent, this mapping is defined only for full-rank matrices and also results in full-rank matrices. Thus, in stark contrast to the Euclidean case, we cannot expect that in any proximity of a low-rank optimal solution that this mapping could be computed exactly via a low-rank SVD. Thus, our last hope is that the MEG mapping \eqref{eq:MDvonNeumannUpdate} could be sufficiently approximated (in a suitable sense) using only a low-rank SVD computation.

Thus, when attempting to derive results in the spirit of  \cite{garberLowRankSmooth, garberStochasticLowRank} for the MEG method, our challenge is two-folded: i. we need to suggest a way to approximate the steps of MEG, which generally requires to store and manipulate full-rank matrices, using only low-rank SVD computations, and ii. we need to be able to establish that the resulting approximation errors can indeed be properly controlled to guarantee the convergence of the new approximated updates.

\section{The Approximated MEG Method with Low-Rank Updates}\label{sec:lowrankMEG}

Our approach of replacing the exact MEG updates in Algorithm \ref{alg:MEG} with approximated low-rank updates is fairly straightforward. Instead of computing the mapping $\Y = \exp(\log\X-\eta\nabla{}f(\X))$ for some $\X\in\mS_n$ (e.g., in case of deterministic updates), we compute explicitly only its best rank-$r$ approximation $\Y^r$ for some parameter $r$, which is simply given by the top $r$ components in the eigen-decomposition of $\Y$, and then slightly perturb it in a way that sets all eigenvalues in the places $r+1$,\dots,$n$ to some small fixed value, so that we end up with a positive definite matrix. We first show that such consecutive updates could indeed be performed by computing only a rank-$r$ SVD and that the iterates of the proposed method could be stored in memory in the form of low-rank matrices.

Concretely, focusing on deterministic updates (the stochastic case follows the same reasoning with the obvious changes), we define a low-rank MEG update in the following way.
\begin{definition}[Low-rank MEG update] \label{def:lowRankMEGUpdate}
The (deterministic) low-rank MEG update to some $\X\in\mathcal{S}_n$ with step-size $\eta >0$ is given by
\begin{align} \label{eq:LowRankMDvonNeumannUpdate}
\X_+:=(1-\varepsilon)\frac{\Y^r}{a}+\frac{\varepsilon}{n-r}(\I-\V^r{\V^r}^{\top})
\end{align}
for some $\varepsilon\in[0,1]$, where $\Y:=\exp(\log(\X)-\eta\nabla f(\X))$ admits the  eigen-decomposition $\Y=\V\Lambda\V^{\top}$, and $\Y^r=\V^r\Lambda^r{\V^r}^{\top}$ is its rank-$r$ approximation, i.e.,  the top $r$ components in the eigen-decomposition, and $a:=\sum_{i=1}^{r} \lambda_i(\Y)$ is its trace.
\end{definition} 

The complete description of our low-rank matrix exponentiated gradient algorithm which is based on the low-rank updates in Definition \ref{def:lowRankMEGUpdate}, is given below as Algorithm \ref{alg:LRMD}.

We now discuss how, perhaps excluding gradient computations, the steps of Algorithm  \ref{alg:LRMD}, i.e., the update in  \eqref{eq:LowRankMDvonNeumannUpdate}, could be computed efficiently in terms of storage and runtime.

We first note that given the eigen-decomposition of $\Y^r$ in \eqref{eq:LowRankMDvonNeumannUpdate}, i.e., the matrix $\V^r\in\reals^{n\times r}$ and the diagonal matrix $\Lambda^r\in\reals^{r\times r}$ (note that only the diagonal elements are needed),  in order to store $\X_+$ in memory (in factored form), indeed the only dense matrix that needs to be stored is $\V^r$. 

Thus, it remains to discuss how $\V^r,\Lambda^r$ could be computed efficiently in terms of runtime and storage. There are various approaches for this task. In particular, efficient methods for high-dimensional matrices rely on iterative approximation algorithms (whose error could be made arbitrarily small) which are standard procedures in numerical linear algebra and are used in numerous algorithms and research papers. Accounting for the exact runtimes of such methods and for the errors incurred by the resulting approximations is beyond the scope of this paper. Below we outline one such popular and simple approach and explain its principled efficient implementation.

Since the exponentiation of a symmetric matrix only changes the eigenvalues (takes their exponent) and leaves the eigenvectors unchanged, in order to compute $\V^r,\Lambda^r$, it suffices to compute the matrices $\V^r,\Sigma^r$ which store the top $r$ components in the eigen-decomposition of the matrix $\M=\log(\X)-\eta\nabla{}f(\X)$. One approach towards computing $\V^r,\Sigma^r$ is as follows. The main step is to compute some $\W\in\reals^{n\times r}$ with orthonormal columns such that $\W\W^{\top} \approx \V^r\V^{r\top}$. Then, we can approximate $\V^r,\Sigma^r$ by computing the eigenvalues and eigenvectors of $\W\W^{\top}\M\W\W^{\top}\approx\V^r\V^{r\top}\M\V^r\V^{r\top} = \V^r\Sigma^r\V^{r\top}$ via the classical Rayleigh--Ritz method (see for instance \cite{jia2001analysis}). That is, we can apply the following steps:
\begin{enumerate}
\item
Compute $\W\in\reals^{n\times r}$ with orthonormal columns such that $\W\W^{\top} \approx \V^r\V^{r\top}$.
\item
Compute the eigenvalues and eigenvectros of $\W^{\top}\M\W\in\reals^{r\times r}$, $\{(\lambda_i,\mathbf{v}_i)\}_{i=1}^r\subset\reals\times\reals^r$.
\item
Return the the eigenvalues and eigenvectors of $\W\W^{\top}\M\W\W^{\top}$, $\{(\lambda_i,\W\mathbf{v}_i)\}_{i=1}^r$.
\end{enumerate}
Given $\W^{\top}\M\W$, step 2 requires an eigen-decomposition of a $r\times r$ matrix and thus can be done in $O(r^3)$ time. Step 3 then requires additional $O(r^2n)$ time to compute the products $\W\mathbf{v}_i, i=1,\dots,r$. It thus remains to be discussed how to compute such matrix $\W$ and  the corresponding matrix $\W^{\top}\M\W$ efficienatly.

Let us recall that $\M=\log(\X)-\eta\nabla{}f(\X)$. Now, suppose $\X$ itself is given in factored form such as in \eqref{eq:LowRankMDvonNeumannUpdate}, i.e., 
$\X=(1-\varepsilon_{-})\widetilde{\X}+\frac{\varepsilon_{-}}{n-r}(\I-\V_{-}^r{\V_{-}^r}^{\top})$,
where $\widetilde{\X}\in\mathcal{S}_n$ is such that $\rank(\widetilde{\X})=r$ and $\varepsilon_{-}\in[0,1]$, and the eigen-decomposition of $\widetilde{\X}$ is $\widetilde{\X}=\V_{-}^r\Lambda_{-}^r{\V_{-}^r}^{\top}$, where $\V_-^r\in\reals^{n\times r}, \Lambda_-^r\in\reals^{r\times r}$. Then, the eigen-decomposition of $\X$ can be written as
\begin{align}\label{eq:eigenFastUpdate} 
\X&=(1-\varepsilon_{-})\V_{-}^r\Lambda_{-}^r{\V_{-}^r}^{\top}+\frac{\varepsilon_{-}}{n-r}\left(\I-\V_{-}^r{\V_{-}^r}^{\top}\right). 
\end{align}
Thus,
\begin{align*}
\log(\X)=\V_{-}^r\log\left((1-\varepsilon_{-})\Lambda_{-}^r\right){\V_{-}^r}^{\top}+\log\left(\frac{\varepsilon_{-}}{n-r}\right)\left(\I-\V_{-}^r{\V_{-}^r}^{\top}\right). 
\end{align*}

Note that given the factorization of  the rank-$r$ initialization $\X_0$ to Algorithm \ref{alg:LRMD}, the first iterate of the Algorithm, $\X_1$ can also be written in similar form to \eqref{eq:eigenFastUpdate}.
Thus, given $\X$ in the form \eqref{eq:eigenFastUpdate} and  some $\W\in\reals^{n\times r}$ as discussed above, computing $\W^{\top}\log(\X)\W$ requires $O(r^2n)$ time. Now, if the gradient $\nabla{}f(\X)$ is a dense matrix without particular structure, then computing $\W^{\top}\M\W=\W^{\top}(\log(\X)-\eta\nabla{}f(\X))\W$ requires $O(rn^2)$ time. However, in many applications the gradient admits favorable properties (e.g., it is the sum of a low-rank matrix given in factored form plus a sparse matrix), and then computing $\W^{\top}\nabla{}f(\X)\W$ could be much more efficient, and so the overall time to compute $\W^{\top}\M\W$ explicitly can be much better than the worst case $O(rn^2)$.

We now discuss the remaining first step of computing some $\W\in\reals^{n\times r}$ with orthonormal columns such that $\W\W^{\top}\approx\V^r\V^{r\top}$, where we recall that $\V^r$ stores as columns the top $r$ eigenvectors of $\M = \log(\X) - \eta\nabla{}f(\X)$. This could be carried out via fast iterative methods such as subspace iteration \cite{saad2011numerical} (aka orthogonal iteration \cite{golub2013matrix}), or even faster Lanczos-type methods (see for instance \cite{Musco15} for a recent study of the such methods). The subspace iteration method for example, which is perhaps the simplest, starts with some orthonormal $\W\in\reals^{n\times r}$ and repeatedly apply the steps: 1. compute the product $\M\W$ and 2. compute a QR-factorization of $\M\W$ to obtain  a new matrix $\W$ with orthonormal columns, that will be used in the next iteration. Thus, each such iteration takes $O(n^2r)$ time in worst-case, where again as discussed above, this could be improved if multiplying  $\M$ with $\W$ could be performed faster than $O(n^2r)$ (in particular, computing the QR-factorization of $\M\W$ requires only $O(r^2n$) time). The number of iterations to reach $\gamma$-accuracy, for instance in the sense that $\Vert{\W\W^{\top}-\V^r\V^{r\top}}\Vert \leq \gamma$, for a given $\gamma >0$, is proportional to $1/\gamma$ in worst-case (and in particular not explicitly dependent on the dimension $n$), but can be significantly faster when the eigengap $\lambda_{r}(\M) - \lambda_{r+1}(\M)$ is sufficiently large, see for instance  \cite{Musco15, golub2013matrix}. 

Thus, to conclude this part, aside from gradient computations, Algorithm \ref{alg:LRMD} could be implemented so that it requires only $O(nr)$ memory, and (treating the approximation error $\gamma$ in the computation of $\W$ above as constant for simplicity)  runtime per iteration proportional to $r^2n$ plus the time to multiply the gradient $\nabla{}f$ with a $n\times r$ matrix, which  in worst-case amounts to $O(rn^2)$.

\begin{algorithm}
	\caption{Low Rank Matrix Exponentiated Gradient for Problem \eqref{eq:Model}}
	\label{alg:LRMD}
	\begin{algorithmic}
		\STATE \textbf{Input:} $\{\eta_t\}_{t\ge1}$ - sequence of (positive) step-sizes, $\lbrace\varepsilon_t\rbrace_{t\ge 0}\subset[0,1]$ - sequence of approximation parameters, $r\in\{1,2,\dots,n-1\}$ - SVD rank parameter, initialization matrix -  $\X_0\in\mS_n$ such that $\rank(\X_0) = r$
		\STATE \textbf{Initialization:} $\X_1=(1-\varepsilon_0)\X_0+\frac{\varepsilon_0}{n}\I$
		\FOR{$t = 1,2,...$} 
            \STATE \ $ \Y_t=\exp(\log (\X_{t})-\eta_{t}\nabla f(\X_{t}))$ (for deterministic optimization; not to be explicitly computed)
            \\ \OR
            \STATE \ $ \Y_t=\exp(\log (\X_{t})-\eta_{t}\widehat{\nabla}_{t})$ (for stochastic optimization;  not to be explicitly computed) \COMMENT{$\widehat{\nabla}_{t}$ is unbiased estimator for  $\nabla{}f(\X_t)$}
            \STATE $\Y_t^r=\V^r\Lambda^r{\V^r}^{\top}$ - rank-$r$ eigen-decomposition of $\Y_t$
			\STATE $a_{t}=\trace(\Y_t^r)$
            \STATE $\X_{t+1}=(1-\varepsilon_t)\frac{\Y_t^r}{a_t}+\frac{\varepsilon_t}{n-r}(\I-\V^r{\V^r}^{\top})$
        \ENDFOR
	\end{algorithmic}
\end{algorithm}

\subsection{Convergence of approximated sequences and computing certificates}\label{section:localConvergence}
As a starting point for our convergence analysis of Algorithm \ref{alg:LRMD}, we first state and prove the most general convergence results which do not rely on any additional assumption (such as strict complementarity or ``warm-start'' initialization). These results, which will also serve as the basis for all of our following theoretical derivations, naturally depend on how well the low-rank updates applied in Algorithm \ref{alg:LRMD} approximate their exact counterparts in Algorithm \ref{alg:exactMEG}. While without further assumptions it is not possible to a-priori guarantee much for these approximations, we conclude this section by providing computable certificates that can ensure during runtime that these approximations are indeed properly bounded, and hence, can be used to certify the correct convergence of the method in practice.

We note that the following two convergence theorems are quite generic and apply to any inexact update sequence $\lbrace \X_t\rbrace_{t\ge 1}$ which approximates the steps of the exact MEG method (Algorithm \ref{alg:exactMEG}), and do not directly rely on the updates applied in Algorithm \ref{alg:LRMD}. 
In particular, the first term in the RHS of the bound in Theorem \ref{thm:converganceSmooth} is simply the standard convergence rate of the mirror-decent method when $f(\cdot)$ is smooth (see Section 10.9 in \cite{beckOptimizationBook}), and the first two terms in the RHS of the bound in Theorem \ref{thm:StochasticConvergence} correspond to the standard rate of the stochastic mirror-decent method (see \cite{sebastienMD}). The additional terms in both bounds  are the error terms which arise from the inexact updates.

\begin{theorem}\label{thm:converganceSmooth}[Convergence with deterministic updates] 
Let $\lbrace \X_t\rbrace_{t\geq 1}$ be a sequence of points in $\mS_n\cap\mbS^n_{++}$. Let $\{\Z_t\}_{t\geq 1}$ be a sequence such that $\Z_1=\X_1$, and for all $t\geq 1$, $\Z_{t+1}$ is the accurate MEG update to $\X_t$ as defined in \eqref{eq:MDvonNeumannUpdate} with step-size $\eta_t=\eta\le\frac{1}{\beta}$. Then, for any $T\geq 1$ we have that 
\begin{align*}
\min_{1\le t\le {T}}f(\X_t) -f(\X^*) & \le \frac{1}{\eta T}\breg(\X^*,\X_1)+\frac{1}{\eta T}\sum_{t=1}^{T-1}\left(\breg(\X^*,\X_{t+1})-\breg(\X^*,\Z_{t+1})\right)
\\ & \ \ \ \ +\frac{2}{T}\sum_{t=1}^{T-1}\Vert\nabla f(\X_{t+1})\Vert_{2}\sqrt{\breg(\Z_{t+1},\X_{t+1})}.
\end{align*}
\end{theorem}

\begin{proof}
By the definition of $\Z_{t+1}$ in \eqref{eq:MDGeneralUpdate},
\[ \Z_{t+1}=\argmin_{\X\in\mathcal{S}_n}\{\langle\eta_t\nabla f(\X_t)-\nabla\omega(\X_t),\X\rangle+\omega(\X)\}. \]

Therefore, by the optimality condition for $\Z_{t+1}$, $\forall \X\in\mathcal{S}_n$, 
\[ \langle \eta_t\nabla f(\X_t)-\nabla\omega(\X_t)+\nabla\omega(\Z_{t+1}),\X-\Z_{t+1}\rangle\ge0. \]

Rearranging and using the three point lemma in \eqref{lemma:threePointLemma}, we get that $\forall \X\in\mathcal{S}_n$, 
\begin{align} \label{ineq:smoothOptZ(t+1)}
\eta_t\langle\nabla f(\X_t),\Z_{t+1}-\X\rangle & \le \langle\nabla\omega(\X_t)-\nabla\omega(\Z_{t+1}),\Z_{t+1}-\X\rangle \nonumber 
\\ & = \breg(\X,\X_t)-\breg(\X,\Z_{t+1})-\breg(\Z_{t+1},\X_t).
\end{align}

From the gradient inequality we have that
\begin{align*}
f(\Z_{t+1}) & \ge f(\X_{t+1})-\langle\nabla f(\X_{t+1}),\X_{t+1}-\Z_{t+1}\rangle.
\end{align*} 

In addition, using the $\beta$-smoothness of $f$, the gradient inequality and the inequality in \eqref{ineq:strongConvexityOfBregmanDistance}, we obtain
\begin{align*}
f(\Z_{t+1}) & \le f(\X_{t})+\langle\nabla f(\X_t),\Z_{t+1}-\X_t\rangle+\frac{\beta}{2}\Vert\Z_{t+1}-\X_t\Vert_*^2
\\ & \le f(\X^*)+\langle\nabla f(\X_t),\Z_{t+1}-\X^*\rangle+\frac{\beta}{2}\Vert\Z_{t+1}-\X_t\Vert_*^2
\\ & \le f(\X^*)+\langle\nabla f(\X_t),\Z_{t+1}-\X^*\rangle+\beta\breg(\Z_{t+1},\X_t).
\end{align*} 

Combining the last two inequalities gives us
\begin{align*}
f(\X_{t+1}) & \le f(\X^*)+\langle\nabla f(\X_t),\Z_{t+1}-\X^*\rangle+\beta\breg(\Z_{t+1},\X_t)+\langle\nabla f(\X_{t+1}),\X_{t+1}-\Z_{t+1}\rangle.
\end{align*} 

Plugging in \eqref{ineq:smoothOptZ(t+1)} with $\X=\X^*$ we get
\begin{align*}
f(\X_{t+1}) & \le f(\X^*)+\frac{1}{\eta_t}\left(\breg(\X^*,\X_t)-\breg(\X^*,\Z_{t+1})-\breg(\Z_{t+1},\X_t)\right)
\\ & \ \ \ \ +\beta\breg(\Z_{t+1},\X_t)+\langle\nabla f(\X_{t+1}),\X_{t+1}-\Z_{t+1}\rangle
\\ & \le f(\X^*)+\frac{1}{\eta_t}\left(\breg(\X^*,\X_t)-\breg(\X^*,\Z_{t+1})\right)+\left(\beta-\frac{1}{\eta_t}\right)\breg(\Z_{t+1},\X_t)
\\ & \ \ \ \ +\Vert\nabla f(\X_{t+1})\Vert_{2}\cdot\Vert\X_{t+1}-\Z_{t+1}\Vert_*
\\ & = f(\X^*)+\frac{1}{\eta_t}\left(\breg(\X^*,\X_t)-\breg(\X^*,\X_{t+1})\right)+\left(\beta-\frac{1}{\eta_t}\right)\breg(\Z_{t+1},\X_t)
\\ & \ \ \ \ +\frac{1}{\eta_t}\left(\breg(\X^*,\X_{t+1})-\breg(\X^*,\Z_{t+1})\right)+\Vert\nabla f(\X_{t+1})\Vert_{2}\cdot\Vert\X_{t+1}-\Z_{t+1}\Vert_*
\\ & \le f(\X^*)+\frac{1}{\eta_t}\left(\breg(\X^*,\X_t)-\breg(\X^*,\X_{t+1})\right)+\left(\beta-\frac{1}{\eta_t}\right)\breg(\Z_{t+1},\X_t)
\\ & \ \ \ \ +\frac{1}{\eta_t}\left(\breg(\X^*,\X_{t+1})-\breg(\X^*,\Z_{t+1})\right)+2\Vert\nabla f(\X_{t+1})\Vert_{2}\sqrt{\breg(\Z_{t+1},\X_{t+1})},
\end{align*}
where the last inequality follows from \eqref{ineq:strongConvexityOfBregmanDistance}.

Averaging over $t=0,\ldots,T-1$ and taking $\eta_t=\eta\le\frac{1}{\beta}$ we reach,
\begin{align*}
\min_{1\le t\le {T}}f(\X_t) -f(\X^*) & \le  \frac{1}{T}\sum_{t=0}^{T-1} f(\X_{t+1}) -f(\X^*) 
\\ & \le \frac{1}{\eta T}\breg(\X^*,\X_1)+\frac{1}{\eta T}\sum_{t=0}^{T-1}\left(\breg(\X^*,\X_{t+1})-\breg(\X^*,\Z_{t+1})\right)
\\ & \ \ \ \ +\frac{2}{T\alpha}\sum_{t=0}^{T-1}\Vert\nabla f(\X_{t+1})\Vert_{2}\sqrt{\breg(\Z_{t+1},\X_{t+1})}
\\ & = \frac{1}{\eta T}\breg(\X^*,\X_1)+\frac{1}{\eta T}\sum_{t=1}^{T-1}\left(\breg(\X^*,\X_{t+1})-\breg(\X^*,\Z_{t+1})\right)
\\ & \ \ \ \ +\frac{2}{T\alpha}\sum_{t=1}^{T-1}\Vert\nabla f(\X_{t+1})\Vert_{2}\sqrt{\breg(\Z_{t+1},\X_{t+1})},
\end{align*}
where the last equality follows since $\Z_1=\X_1$.

\end{proof}

\begin{remark}Note that with the standard choice of step-size $\eta = 1/\beta$, whenever the sum $\sum_{t=1}^{T-1}\max\{B(\X^*,\X_{t+1})-B(\X^*,\Z_{t+1}), \sqrt{B(\Z_{t+1},\X_{t+1})}\}$ grows sublinearly in $T$, Theorem \ref{thm:converganceSmooth} indeed implies the convergence of Algorithm \ref{alg:LRMD} with deterministic gradients. In particular, when this sum is bounded by  a constant the rate is $O(1/T)$.
\end{remark}

\begin{theorem} \label{thm:StochasticConvergence}[Convergence with stochastic updates]
Let $\lbrace \X_t\rbrace_{t\geq 1}$ be a sequence of points in $\mS_n\cap\mbS^n_{++}$.
Let $\{\Z_t\}_{t\geq 1}$ be a sequence such that $\Z_1 = \X_1$, and for all $t\geq 1$, $\Z_{t+1}$ is the accurate stochastic MEG update to $\X_t$ as defined in \eqref{eq:MDvonNeumannStochasticUpdate} with fixed step-size $\eta>0$. Then, after $T$ iterations, letting $\bar{\X} \sim \textrm{Uni}\{\X_1,\dots,\X_T\}$, it holds that 
\begin{align*}
\E\left[f(\bar{\X})\right]-f(\X^*) & \le \frac{\breg(\X^*,\X_1)+\frac{G^2}{2}T\eta^2+\sum_{t=1}^T \E\left[\breg(\X^*,\X_{t+1})-\breg(\X^*,\Z_{t+1})\right]}{T\eta},
\end{align*}
where $G \geq \max_{t}\Vert{\widehat{\nabla}_t}\Vert_2$.
\end{theorem}

\begin{proof}

%By the definition of $\Z_{t+1}$ in \eqref{eq:MDGeneralUpdate}, where the exact gradient is replaced with a stochastic gradient $\widehat{\nabla}_t$, 
%\[ \Z_{t+1}=\argmin_{\X\in\mathcal{S}_n}\{\langle\eta_t\widehat{\nabla}_t-\nabla\omega(\X_t),\X\rangle+\omega(\X)\}. \]
%
%Therefore, by the optimality condition for $\Z_{t+1}$, $\forall \X\in\mathcal{S}_n$, 
%\[ \langle \eta_t\widehat{\nabla}_t-\nabla\omega(\X_t)+\nabla\omega(\Z_{t+1}),\X-\Z_{t+1}\rangle\ge0. \]
%
%Rearranging and using the three point lemma \eqref{lemma:threePointLemma}, we get that $\forall \X\in\mathcal{S}_n$, 
%\begin{align} \label{ineq:StochasticOptZ(t+1)}
%\eta_t\langle\widehat{\nabla}_t,\Z_{t+1}-\X\rangle & \le \langle\nabla\omega(\X_t)-\nabla\omega(\Z_{t+1})),\Z_{t+1}-\X\rangle \nonumber 
%\\ & = \breg(\X,\X_t)-\breg(\X,\Z_{t+1})-\breg(\Z_{t+1},\X_t).
%\end{align}

Following the same arguments used to derive Eq. \eqref{ineq:smoothOptZ(t+1)}, but replacing the exact gradient with the stochastic gradient $\widehat{\nabla}_t$, we have that
\begin{align} \label{ineq:StochasticOptZ(t+1)}
\eta_t\langle\widehat{\nabla}_t,\Z_{t+1}-\X\rangle & \le \breg(\X,\X_t)-\breg(\X,\Z_{t+1})-\breg(\Z_{t+1},\X_t).
\end{align}

Adding $\eta_t\langle \widehat{\nabla}_t,\X_t-\Z_{t+1}\rangle$ to both sides of \eqref{ineq:StochasticOptZ(t+1)} with $\X=\X^*$, we obtain
\begin{align} \label{ineq:StochasticOptZ(t+1)_2}
\eta_t\langle\widehat{\nabla}_t,\X_{t}-\X^*\rangle & \le \breg(\X^*,\X_t)-\breg(\X^*,\Z_{t+1})-\breg(\Z_{t+1},\X_t)+\eta_t\langle \widehat{\nabla}_t,\X_t-\Z_{t+1}\rangle.
\end{align}

Using H\"{o}lder's inequality, and \eqref{ineq:strongConvexityOfBregmanDistance} since $\breg(\cdot,\cdot)$ is $1$-strongly convex, we have that
\begin{align*}
-\breg(\Z_{t+1},\X_t)+\eta_t\langle \widehat{\nabla}_t,\X_t-\Z_{t+1}\rangle & \le -\frac{1}{2}\Vert \Z_{t+1}-\X_t\Vert_*^2 + \eta_t\Vert\widehat{\nabla}_t\Vert_{2}\cdot\Vert\Z_{t+1}-\X_t\Vert_*
\\ & \le \max_{a\in\reals}\left\{-\frac{1}{2}a^2+a\eta_t\Vert\widehat{\nabla}_t\Vert_{2}\right\}=\frac{1}{2}\eta_t^2\Vert\widehat{\nabla}_t\Vert_{2}^2.
\end{align*}

Plugging this into \eqref{ineq:StochasticOptZ(t+1)_2} we get
\begin{align*}
\eta_t\langle\widehat{\nabla}_t,\X_{t}-\X^*\rangle & \le \breg(\X^*,\X_t)-\breg(\X^*,\Z_{t+1})+\frac{1}{2}\eta_t^2\Vert\widehat{\nabla}_t\Vert_{2}^2
\\ & = \breg(\X^*,\X_t)-\breg(\X^*,\X_{t+1})+\frac{1}{2}\eta_t^2\Vert\widehat{\nabla}_t\Vert_{2}^2 +\breg(\X^*,\X_{t+1})-\breg(\X^*,\Z_{t+1}).
\end{align*}

Summing over $t=1,...,T$ and dividing by $T$ we get
\begin{align*}
\frac{1}{T}\sum_{t=1}^T \eta_t \langle\widehat{\nabla}_t,\X_{t}-\X^*\rangle & \le \frac{1}{T}[\breg(\X^*,\X_1)-\breg(\X^*,\X_{T+1})]+\frac{1}{2T}\sum_{t=1}^T \eta_t^2\Vert\widehat{\nabla}_t\Vert_{2}^2
\\ & \ \ \ \ +\frac{1}{T}\sum_{t=1}^T \left[\breg(\X^*,\X_{t+1})-\breg(\X^*,\Z_{t+1})\right]
\\ & \le \frac{1}{T}\breg(\X^*,\X_1)+\frac{1}{2T}\sum_{t=1}^T \eta_t^2\Vert\widehat{\nabla}_t\Vert_{2}^2+\frac{1}{T}\sum_{t=1}^T \left[\breg(\X^*,\X_{t+1})-\breg(\X^*,\Z_{t+1})\right]. 
\end{align*}

Taking expectation on both sides, we obtain
\begin{align*}
\E\left[\frac{1}{T}\sum_{t=1}^T \eta_t \langle\widehat{\nabla}_t,\X_{t}-\X^*\rangle\right] \le \frac{1}{T}\breg(\X^*,\X_1)+\frac{G^2}{2T}\sum_{t=1}^T \eta_t^2+\frac{1}{T}\sum_{t=1}^T \E\left[\breg(\X^*,\X_{t+1})-\breg(\X^*,\Z_{t+1})\right]. 
\end{align*}

It addition, using the law of total expectation, the gradient inequality, and taking $\eta_t=\eta$,
\begin{align*}
\E\left[\frac{1}{T}\sum_{t=1}^T \eta \langle\widehat{\nabla}_t,\X_{t}-\X^*\rangle\right] & = \E\left[\frac{1}{T}\sum_{t=1}^T \eta \langle\E_{t}[\widehat{\nabla}_t\vert\X_t],\X_{t}-\X^*\rangle\right] = \E\left[\frac{1}{T}\sum_{t=1}^T \eta \langle\nabla f(\X_t),\X_{t}-\X^*\rangle\right]
\\ & \ge \E\left[\frac{1}{T}\sum_{t=1}^T \eta (f(\X_t)-f(\X^*))\right] = \eta \left(\E\left[\frac{1}{T}\sum_{t=1}^T f(\X_t)\right]-f(\X^*)\right). 
\end{align*}

Noticing that $\E\left[\frac{1}{T}\sum_{t=1}^T f(\X_t)\right]=\E\left[f(\bar{\X})\right]$ and combining the last two inequalities we obtain
\begin{align*}
\E\left[f(\bar{\X})\right]-f(\X^*) & \le \frac{\breg(\X^*,\X_1)+\frac{G^2}{2}T\eta^2+\sum_{t=1}^T \E\left[\breg(\X^*,\X_{t+1})-\breg(\X^*,\Z_{t+1})\right]}{T\eta}. 
\end{align*}
\end{proof}

\begin{remark}
Note that with a standard step-size of the form $\eta = C/\sqrt{T}$, for a suitable constant $C>0$, whenever the sum $\sum_{t=1}^{T}\E[B(\X^*,\X_{t+1})-B(\X^*,\Z_{t+1})]$ grows sublinearly in $\sqrt{T}$, Theorem \ref{thm:StochasticConvergence} indeed implies the convergence of Algorithm \ref{alg:LRMD} with stochastic gradients. In particular, when this sum is bounded by a constant the rate is $O(1/\sqrt{T})$.
\end{remark}

\subsubsection{Computing certificates for convergence}\label{sec:certificate}

As can be seen in Theorems \ref{thm:converganceSmooth} and \ref{thm:StochasticConvergence}, the convergence of our Algorithm \ref{alg:LRMD} naturally depends on the approximation errors $\left(\breg(\X^*,\X_{t+1})-\breg(\X^*,\Z_{t+1})\right)$ and $\breg(\Z_{t+1},\X_{t+1})$. Thus, naturally, a significant portion of our analysis will be devoted to bounding these errors. Our first step is the following lemma.

\begin{lemma} \label{lemma:sufficientConvegenceCondition}
Let $\X\in\mathcal{S}_n$, $\X\succ 0$. Let $\Z_{+}$ be the accurate MEG update to $\X$ as defined in \eqref{eq:MDvonNeumannUpdate} or \eqref{eq:MDvonNeumannStochasticUpdate}, and $\X_{+}$ be the corresponding low-rank MEG update as defined in \eqref{eq:LowRankMDvonNeumannUpdate}.
Then, for any $\varepsilon\in(0,3/4]$,
\begin{align*}
\max\lbrace\breg(\X^*,\X_{+})-\breg(\X^*,\Z_{+}),\breg(\Z_{+},\X_{+})\rbrace \le \max\left\lbrace 2\varepsilon, \log\left(\frac{(n-r)\lambda_{r+1}(\Y)}{\varepsilon \sum_{i=1}^n \lambda_i(\Y)}\right)\right\rbrace,
\end{align*}
where $\Y =\exp(\log (\X)-\eta\nabla{}f(\X))$ for deterministic gradients or $\Y =\exp(\log (\X)-\eta\widehat{\nabla})$ for stochastic gradients with $\widehat{\nabla}$ being an unbiased estimator for $\nabla{}f(\X)$.
\end{lemma}

\begin{proof}

The von Neumann inequality claims that for any $\X,\Y\in\mathbb{S}^n$, it holds that $\langle \X,\Y\rangle \le \sum_{i=1}^n \lambda_i(\X)\lambda_i(\Y)$. Therefore,
\begin{align} \label{bound1ErrorTerm}
\breg(\X^*,\X_{+})-\breg(\X^*,\Z_{+})  & = \trace(\X^*\log(\X^*)-\X^*\log(\X_{+}))-\trace(\X^*\log(\X^*)-\X^*\log(\Z_{+})) \nonumber
\\ & = \trace\left(\X^*\left[\log(\Z_{+})-\log(\X_{+})\right]\right) \nonumber
\\ & \le \sum_{i=1}^n \lambda_i(\X^*)\lambda_i(\log(\Z_{+})-\log(\X_{+}) \nonumber
\\ & \le \lambda_1(\log(\Z_{+})-\log(\X_{+})).
\end{align}

Using similar arguments, we get the bound
\begin{align} \label{bound2ErrorTerm}
\breg(\Z_{+},\X_{+}) & = \trace(\Z_{+}\log(\Z_{+})-\Z_{+}\log(\X_{+})) \nonumber
\\ & \le \sum_{i=1}^n \lambda_i(\Z_{+})\lambda_i(\log(\Z_{+})-\log(\X_{+})
 \le \lambda_1(\log(\Z_{+})-\log(\X_{+})),
\end{align}
where the last inequality holds since $\Z_{t+1}\in\mathcal{S}_n$.

Combining \eqref{bound1ErrorTerm} and \eqref{bound2ErrorTerm}, we obtain the following bound:
\begin{align*} %\label{boundBothErrorTerm}
\max\lbrace\breg(\X^*,\X_{+})-\breg(\X^*,\Z_{+}),\breg(\Z_{+},\X_{+})\rbrace
 \le \lambda_1(\log(\Z_{+})-\log(\X_{+})).
\end{align*}

Let $\Y=\V\Lambda\V^{\top}$ denote the eigen-decomposition of $\Y$. Therefore, its rank-$r$ approximation is $\Y^r=\V^r\Lambda^r{\V^r}^{\top}$. $\Z_{+}$ and $\X_{+}$ have the same eigen-vectors, and therefore $\log(\Z_{+})-\log(\X_{+})=\V\left[\log\left(\frac{\Lambda}{b}\right)-\log\left(\begin{array}{cc}(1-\varepsilon)\frac{\Lambda^r}{a} & 0 \\ 0 & \frac{\varepsilon}{n-r}\I\end{array}\right)\right]\V^{\top}$, where $a:=\sum_{i=1}^{r} \lambda_i(\Y)$ and $b:=\sum_{i=1}^{n} \lambda_i(\Y)$.
Denote $\D=\log\left(\frac{\Lambda}{b}\right)-\log\left(\begin{array}{cc}(1-\varepsilon)\frac{\Lambda^r}{a} & 0 \\ 0 & \frac{\varepsilon}{n-r}\I\end{array}\right)$.
It holds that
\begin{align}
\forall j\in[r] &:\ \lambda_j(\D) = \log\left(\frac{\frac{\lambda_j(\Y)}{b}}{(1-\varepsilon)\frac{\lambda_j(\Y)}{a}}\right); \label{StochasticBigEig} 
\\ \forall j>r & :\ \lambda_j(\D) = \log\left(\frac{\frac{\lambda_j(\Y)}{b}}{\frac{\varepsilon}{n-r}}\right). \label{StochasticLittleEig}
\end{align}
It is important to note that $\lambda_1(\D),\ldots,\lambda_n(\D)$ are not necessarily ordered in a non-increasing order.

From \eqref{StochasticBigEig}, for all $j\le r$ it holds that
\begin{align*}
\lambda_j(\D) & = \log\left(\frac{\lambda_j(\Y)\sum_{i=1}^{r}\lambda_i(\Y)}{(1-\varepsilon)\lambda_j(\Y)\sum_{i=1}^n\lambda_i(\Y)}\right)
= \log\left(\frac{\sum_{i=1}^{r}\lambda_i(\Y)}{\sum_{i=1}^n\lambda_i(\Y)}\right)+ \log\left(\frac{1}{1-\varepsilon}\right)\\
&\le \log\left(\frac{1}{1-\varepsilon}\right) \underset{(a)}{\le} 2\varepsilon,
\end{align*}
where (a) holds for any $\varepsilon\le\frac{3}{4}$.

For $j\ge r+1$, note that using the definition of the eigenvalues of $\D$ in \eqref{StochasticLittleEig}, it holds that 
\begin{align*} 
\lambda_j(\D) & = \log\left(\frac{(n-r)\lambda_j(\Y)}{\varepsilon \sum_{i=1}^n \lambda_i(\Y)}\right)
\le \log\left(\frac{(n-r)\lambda_{r+1}(\Y)}{\varepsilon \sum_{i=1}^n \lambda_i(\Y)}\right) = \lambda_{r+1}(\D).
\end{align*}

The last two inequalities put together give us the desired result.  

\end{proof}

Lemma \ref{lemma:sufficientConvegenceCondition} gives us a sufficient condition so that the errors due to the low rank MEG updates in Algorithm \ref{alg:LRMD} do not become too significant.
As seen, a sufficient condition to guarantee that $\max\lbrace\breg(\X^*,\X_{+})-\breg(\X^*,\Z_{+}),\breg(\Z_{+},\X_{+})\rbrace\le 2\varepsilon$ is that 
\begin{equation} \label{convergenceCondition}
\log\left(\frac{(n-r)\lambda_{r+1}(\Y)}{\varepsilon \sum_{i=1}^n \lambda_i(\Y)}\right)\le 2\varepsilon.
\end{equation}
 
Unfortunately, the condition in \eqref{convergenceCondition} is useless as it involves computing the full matrix of $\Y$ in order to calculate the scalar $b=\sum_{i=1}^n {\lambda_i(\Y)}$, which is exactly what we are striving to avoid. Instead, $b$ can be approximated with some $b_k:=\sum_{i=1}^k {\lambda_i(\Y)}$ for some $k\le n$.
Thus, instead of the full eigen-decomposition of $\Y$, in order to check if our Algorithm \ref{alg:LRMD} is converging with a tolerable error, we can check the weaker condition of $\log\left(\frac{(n-r)\lambda_{r+1}(\Y)}{\varepsilon b_k}\right)\le 2\varepsilon$, which is likewise a sufficient condition since $b_k \leq b$, and can be computed much more efficiently for $k\ll n$. In particular, computing $b_{r+1}$, where $r$ is the SVD rank parameter in Algorithm \ref{alg:LRMD} requires to increase the rank of the SVD computations in Algorithm \ref{alg:LRMD} by only one.\footnote{Since $\lambda_{r+1}(\Y)$ is also required in \eqref{convergenceCondition} there is no point in computing $b_k$ for $k< r+1$ which will result in a worse approximation of the scalar $b$.}\footnote{We note that an alternative approach which we do not pursue here, could be to approximate $b=\trace(\Y)$ using the expectation $\trace(\Y)=\E_{\u\sim\mathcal{N}(0,I)}\u^{\top}\Y\u$. This expectation could be in principle estimated using random sampling and  fast approximated matrix exponential-vector products using the Lanczos algorithm \cite{musco2018stability}.} In Section \ref{section:experiments} we present numerical evidence which demonstrate that in practice the accurate certificate computed using $b=b_n$ does not give any additional benefit over the use of $b_{r+1}$. Moreover, in the experiments we demonstrate that this computationally-cheap certificate indeed holds true from the very early stages of the run (already from the first iteration in most cases), and hence, may indeed be appealing for practical uses.

\section{Provable Bounds on Low-Rank Approximations under Strict Complementarity and Warm-Start Initialization}\label{sec:main_tools}
In this section we provide two of the technical foundations towards obtaining provable convergence results for our Algorithm \ref{alg:LRMD}  from a ``warm-start" initialization under the strict complementarity condition (Assumption \ref{ass:strictcomp}).

In the following lemma we present our main technical result: for any point sufficiently close to an optimal solution $\X^*$ for which Assumption \ref{ass:strictcomp} holds, the error in using low rank MEG updates is indeed guaranteed to be properly bounded.

\begin{lemma} \label{lemma:StochasticRadius}
Let $\X^*\in\mX^*$ be an optimal solution for which Assumption \ref{ass:strictcomp} holds with $\delta >0$. Let $\rank(\X^*):=r^*$ and let $n\not=r\ge r^*$ be the SVD rank parameter. Let $\X\in\mathcal{S}_n,\X\succ 0$ be a matrix such that $\lambda_{r^*+1}(\X)\le\frac{\varepsilon_{-}}{n-r}$ for some $\varepsilon_{-}\in[0,1]$ and let $\widetilde{\nabla}\in\mbS^n$ be such that $\Vert\widetilde{\nabla}-\nabla f(\X)\Vert\le\xi$ for some $\xi\ge0$ (for deterministic updates $\xi=0$). Denote $G=\sup_{\X\in\mathcal{S}_n}\Vert \nabla f(\X)\Vert_2$. \ If 
\begin{equation} \label{eq:StochasticRadiusOfconvergence}
\sqrt{\breg(\X^*,\X)} \le \frac{1}{\sqrt{2}}\left[2\beta+\left(1+\frac{2\sqrt{2r^*}}{\lambda_{r^*}(\X^*)}\right)G\right]^{-1}\left(\delta-\frac{1}{\eta}\log\left(\frac{\varepsilon_{-}}{\varepsilon}\right)+\frac{2\varepsilon}{\eta}-2\xi\right), 
\end{equation}
then, for any $\eta>0$ the matrix $\Y = \exp\left({\log(\X)-\eta\widetilde{\nabla}}\right)$ and scalar $b=\sum_{i=1}^n\lambda_i(\Y)$ satisfy:
\begin{align}
\log\left(\frac{(n-r)\lambda_{r+1}(\Y)}{\varepsilon b}\right) \le 2\varepsilon.
\end{align}
As a result, for $\Z_+ := \Y/b$ and $\X_+$ as defined in Eq. \eqref{eq:LowRankMDvonNeumannUpdate}, it holds that 
\begin{align*}
\max\lbrace\breg(\X^*,\X_{+})-\breg(\X^*,\Z_{+}),\breg(\Z_{+},\X_{+})\rbrace \le 2\varepsilon.
\end{align*}
\end{lemma}

\begin{proof} 

It holds that
\begin{align} \label{ineq:StochasticLittleEigBound}
\log\left(\frac{(n-r)\lambda_{r^*+1}(\Y)}{\varepsilon \sum_{i=1}^n \lambda_i(\Y)}\right) & = \log\left(\frac{n-r}{\varepsilon}\right) + \log(\lambda_{r+1}(\Y))-\log\left( \sum_{i=1}^n \lambda_i(\Y)\right) \nonumber
\\ & \le \log\left(\frac{n-r}{\varepsilon}\right) + \log(\lambda_{r^*+1}(\Y))-\log\left( \sum_{i=1}^n \lambda_i(\Y)\right).
\end{align}
We will now separately bound the last two terms.

\begin{align}  \label{ineq:StochasticInProofLambdaRPlus1}
\log(\lambda_{r^*+1}(\Y)) & = \lambda_{r^*+1}(\log(\Y)) \underset{(a)}{\le} \lambda_{r^*+1}(\log(\X)-\eta\nabla f(\X^*)) + \eta\lambda_1(\nabla f(\X^*)-\widetilde{\nabla}) \nonumber
\\ & \le \sum_{i=1}^{r^*+1}\lambda_{i}(\log(\X)-\eta\nabla f(\X^*))-\sum_{i=1}^{r^*}\lambda_{i}(\log(\X)-\eta\nabla f(\X^*))+\eta\Vert\widetilde{\nabla}-\nabla f(\X^*)\Vert_2 \nonumber
\\ & \underset{(b)}{\le} \sum_{i=1}^{r^*+1}\lambda_{i}(\log(\X))-\eta\sum_{i=1}^{r^*+1}\lambda_{n-i+1}(\nabla f(\X^*))-\sum_{i=1}^{r^*}\lambda_{i}(\log(\X)-\eta\nabla f(\X^*)) \nonumber \\ & \ \ \ +\eta\Vert\widetilde{\nabla}-\nabla f(\X^*)\Vert_2 \nonumber
\\ & \underset{(c)}{=} \sum_{i=1}^{r^*}\lambda_{i}(\log(\X))+\log\left(\lambda_{r^*+1}(\X)\right)-\eta r^*\lambda_n(\nabla f(\X^*))-\eta \lambda_{n-r^*}(\nabla f(\X^*))\nonumber \\ & \ \ \ -\sum_{i=1}^{r^*}\lambda_{i}(\log(\X)-\eta\nabla f(\X^*))+\eta\Vert\widetilde{\nabla}-\nabla f(\X^*)\Vert_2,
\end{align}
where (a) follows from Weyl's inequality (see for instance\cite{Weyl}), (b) follows from Ky Fan's inequality for eigenvalues (see for instance \cite{kyFan}) and (c) follows from Lemma \ref{lemma:opt-grad}.

Since $f$ is $\beta$-smooth,
\begin{align} \label{ineq:StochasticNorm-nablat-gradopt}
\Vert\widetilde{\nabla}-\nabla f(\X^*)\Vert_2 \le \Vert\widetilde{\nabla}-\nabla f(\X)\Vert_2+\Vert\nabla f(\X)-\nabla f(\X^*)\Vert_2 \le \xi+\beta\Vert\X-\X^*\Vert_*.
\end{align}

Denote $\log(\X)=\W\D\W^{\top}$ to be the eigen-decomposition of $\log(\X)$ and denote $\W_{r^*}$ to be the matrix with the $r^*$ first columns of $\W$.
\begin{align} \label{ineq:inProofSumReigs}
-\sum_{i=1}^{r^*}\lambda_{i}(\log(\X)-\eta\nabla f(\X^*)) & \le -\W_{r^*}\W_{r^*}^{\top}\bullet(\log(\X)-\eta\nabla f(\X^*)) \nonumber
\\ & = -\sum_{i=1}^{r^*}\lambda_i(\log(\X))+\eta\W_{r^*}\W_{r^*}^{\top}\bullet\nabla f(\X^*).
\end{align}

Let $\X^*=\V^*\Lambda^*{\V^*}^{\top}$ denote the eigen-decomposition of $\X^*$. From Lemma \ref{lemma:opt-grad} this implies that $\V^*{\V^*}^{\top}\bullet\nabla f(\X^*)=r^*\lambda_{n}(\nabla f(\X^*))$. Therefore,
\begin{align} \label{ineq:inProofInnerProd}
\W_{r^*}\W_{r^*}^{\top}\bullet\nabla f(\X^*) & = \V^*{\V^*}^{\top}\bullet\nabla f(\X^*)+\W_{r^*}\W_{r^*}^{\top}\bullet\nabla f(\X^*)-\V^*{\V^*}^{\top}\bullet\nabla f(\X^*) \nonumber
\\ & \le \V^*{\V^*}^{\top}\bullet\nabla f(\X^*)+\Vert\W_{r^*}\W_{r^*}^{\top}-\V^*{\V^*}^{\top}\Vert_{*}\Vert\nabla f(\X^*)\Vert_{2} \nonumber
\\ & = r^*\lambda_{n}(\nabla f(\X^*))+\Vert\W_{r^*}\W_{r^*}^{\top}-\V^*{\V^*}^{\top}\Vert_{*}\Vert\nabla f(\X^*)\Vert_{2} \nonumber
\\ & \le r^*\lambda_{n}(\nabla f(\X^*))+\sqrt{2r^*}\Vert\W_{r^*}\W_{r^*}^{\top}-\V^*{\V^*}^{\top}\Vert_{F}\Vert\nabla f(\X^*)\Vert_{2} \nonumber
\\ & \underset{(a)}{\le} r^*\lambda_{n}(\nabla f(\X^*))+\frac{2\sqrt{2r^*}\Vert\X-\X^*\Vert_{F}}{\lambda_{r^*}(\X^*)}\Vert\nabla f(\X^*)\Vert_{2} \nonumber
\\ & \le r^*\lambda_{n}(\nabla f(\X^*))+\frac{2\sqrt{2r^*}\Vert\X-\X^*\Vert_{*}}{\lambda_{r^*}(\X^*)}\Vert\nabla f(\X^*)\Vert_{2}.
\end{align}
Here, (a)  follows using the Davis-Kahan $\sin(\theta)$ theorem --- see Lemma \ref{lemma:DavisKahan} in the appendix.

Plugging \eqref{ineq:StochasticNorm-nablat-gradopt} and the bounds \eqref{ineq:inProofSumReigs} and \eqref{ineq:inProofInnerProd} into \eqref{ineq:StochasticInProofLambdaRPlus1}, we get
\begin{align} 
\log(\lambda_{r^*+1}(\Y)) & \le \log\left(\lambda_{r^*+1}(\X)\right)-\eta \lambda_{n-r^*}(\nabla f(\X^*))+\eta\frac{2\sqrt{2r^*}\Vert\X-\X^*\Vert_{*}}{\lambda_{r^*}(\X^*)}\Vert\nabla f(\X^*)\Vert_{2} \nonumber
\\ & \ \ \ +\eta\xi+\eta\beta\Vert\X-\X^*\Vert_* \label{ineq:StochasticInProofSecondTermWithoutAssumption}
\\ & \le \log\left(\frac{\varepsilon_{-}}{n-r}\right)-\eta \lambda_{n-r^*}(\nabla f(\X^*))+\eta\frac{2\sqrt{2r^*}\Vert\X-\X^*\Vert_{*}}{\lambda_{r^*}(\X^*)}\Vert\nabla f(\X^*)\Vert_{2} \nonumber
\\ & \ \ \ +\eta\xi+\eta\beta\Vert\X-\X^*\Vert_*. \label{ineq:StochasticInProofSecondTerm}
\end{align}
where the second inequality holds from the assumption that $\lambda_{r^*+1}(\X)\le\frac{\varepsilon_{-}}{n-r}$.

For the last term of \eqref{ineq:StochasticLittleEigBound}, 
\begin{align}  
-\log\left(\sum_{i=1}^n\lambda_i(\Y)\right) & = -\log\Big(\sum_{i=1}^n\lambda_i\left(\exp\left(\log(\X)-\eta_t\widetilde{\nabla}\right)\right)\Big) \nonumber
\\ & \underset{(a)}{\le} -\sum_{i=1}^n\lambda_i(\X)\cdot\log\left(\frac{\lambda_i\left(\exp\left(\log(\X)-\eta\widetilde{\nabla}\right)\right)}{\lambda_i(\X)}\right) \nonumber
\\ & = -\sum_{i=1}^n\lambda_i(\X)\cdot\left[\log\left(\lambda_i\left(\exp\left(\log(\X)-\eta\widetilde{\nabla}\right)\right)\right)-\log(\lambda_i(\X)\right] \nonumber
\\ & = -\sum_{i=1}^n\lambda_i(\X)\cdot\left[\lambda_i\left(\log(\X)-\eta\widetilde{\nabla}\right)-\log(\lambda_i(\X)\right] \nonumber
\\ & \underset{(b)}{\le} -\langle\X,\log(\X)-\eta\widetilde{\nabla}\rangle +\sum_{i=1}^n\lambda_i(\X)\lambda_i(\log(\X)) \nonumber
\\ & \underset{(c)}{=} -\langle\X,\log(\X)-\eta\widetilde{\nabla}\rangle +\sum_{i=1}^n\lambda_i(\X\log(\X)) \nonumber
\\ & = -\trace(\X\log(\X))+\trace(\X\log(\X))+\eta\langle\X,\widetilde{\nabla}\rangle \label{ineq:middleOfProof1}
\\ & = \eta\langle\X^*,\nabla f(\X^*)\rangle + \eta\left(\langle\X-\X^*,\nabla f(\X^*)\rangle+\langle\X,\widetilde{\nabla}-\nabla f(\X^*)\rangle\right) \nonumber
\\ & \underset{(d)}{=} \eta\lambda_{n}(\nabla f(\X^*)) + \eta\left(\langle\X-\X^*,\nabla f(\X^*)\rangle+\langle\X,\widetilde{\nabla}-\nabla f(\X^*)\rangle\right) \nonumber
\\ & \le \eta\lambda_{n}(\nabla f(\X^*)) + \eta\Vert\X-\X^*\Vert_*\Vert\nabla f(\X^*)\Vert_2+\eta\Vert\X\Vert_*\Vert\widetilde{\nabla}-\nabla f(\X^*)\Vert_2 \nonumber
\\ & \underset{(e)}{\le} \eta\lambda_{n}(\nabla f(\X^*)) + \eta\Vert\X-\X^*\Vert_*\Vert\nabla f(\X^*)\Vert_2+\eta(\xi+\beta\Vert\X-\X^*\Vert_*) \label{ineq:StochasticInProofThirdTerm}
\end{align} 
where (a) follows since $-\log(\cdot)$ is convex, (b) follows from the von Neumann inequality, (c) follows since $\X$ and $\log(\X)$ have the same eigen-vectors, (d) follows from Lemma \ref{lemma:opt-grad}, and (e) follows from \eqref{ineq:StochasticNorm-nablat-gradopt} and since $\Vert\X\Vert_*=1$.

Plugging \eqref{ineq:StochasticInProofSecondTerm} and \eqref{ineq:StochasticInProofThirdTerm} into \eqref{ineq:StochasticLittleEigBound}, we get
\begin{align*}
\log\left(\frac{(n-r)\lambda_{r+1}(\Y)}{\varepsilon \sum_{i=1}^n \lambda_i(\Y)}\right)
& \le -\eta\delta +\log\left(\frac{\varepsilon_{-}}{\varepsilon}\right)+ 2\eta\xi \\
&~+ \eta\left(2\beta+\left(1+\frac{2\sqrt{2r^*}}{\lambda_{r^*}(\X^*)}\right)\Vert\nabla f(\X^*)\Vert_{2}\right)\Vert\X-\X^*\Vert_*
\\ & \le -\eta\delta +\log\left(\frac{\varepsilon_{-}}{\varepsilon}\right)+ 2\eta\xi + \sqrt{2}\eta\left(2\beta+\left(1+\frac{2\sqrt{2r^*}}{\lambda_{r^*}(\X^*)}\right)G\right)\sqrt{\breg(\X^*,\X)},
\end{align*}
where the last inequality follows from \eqref{ineq:strongConvexityOfBregmanDistance} due to the $1$-strong convexity of $\breg(\cdot,\cdot)$ and the definition of $G$.

Therefore, if $\X$ is a matrix such that
\begin{equation*} %\label{eq:StochasticRadiusInProof}
\sqrt{\breg(\X^*,\X)} \le \frac{1}{\sqrt{2}}\left[2\beta+\left(1+\frac{2\sqrt{2r^*}}{\lambda_{r^*}(\X^*)}\right)G\right]^{-1}\left(\delta-\frac{1}{\eta}\log\left(\frac{\varepsilon_{-}}{\varepsilon}\right)+\frac{2\varepsilon}{\eta}-2\xi\right),
\end{equation*}
then
\begin{align*}
\log\left(\frac{(n-r)\lambda_{r+1}(\Y)}{\varepsilon \sum_{i=1}^n \lambda_i(\Y)}\right)\le 2\varepsilon.
\end{align*}

\end{proof}

Since our provable convergence results only apply from a ``warm-start" initialization, we present the following lemma which provides certain conditions on the initialization parameters $(\X_0,\varepsilon_0)$ in Algorithm \ref{alg:LRMD}, which are used to produce the first iterate $\X_1$, so that it is close enough to an optimal solution of interest. Unfortunately, without further assumptions on the objective function $f(\cdot)$, we do not have a simple procedure that can provably generate initialization parameters that satisfy these conditions.

\begin{lemma}[Warm-start Initialization]\label{lem:warmstart}
Let $\X^*\in\mX^*$ be an optimal solution to Problem \eqref{eq:Model}. Let $\rank(\X^*):=r^*$, and let $\X\in\mS_n$ be such that $\rank(\X) \geq \rank(\X^*)$. Let $\varepsilon\in(0,3/4]$ and fix some $R>0$. Suppose
\begin{align}\label{eq:init:highRank}
\trace(\X^*\log(\X^*) - \X^*\V_{\X}\log(\Lambda_{\X})\V_{\X}^{\top}) +  \lambda_1(\X^*)\log\left({\frac{n}{\varepsilon}}\right)(r^*-\Vert{\V_{\X^*}^{\top}\V_{\X}}\Vert_F^2)  + 4\varepsilon \leq R^2
\end{align}
holds, where $\V_{\X}\Lambda_{\X}\V_{\X}^{\top}$ denotes the compact-form eigen-decomposition\footnote{The compact-form eigen-decomposition of a rank-$r$ matrix $\M$ is written as $\M=\V_r\Lambda_r{\V_r}^{\top}$, where $\V_r\in\reals^{n\times{}r}$ stores the eigenvectors of $\M$ associated with nonzero eigenvalues, and $\Lambda_r\in\reals^{r\times{}r}$ is a diagonal matrix whose diagonal entries are exactly the nonzero eigenvalues of $\M$.}  of $\X$ (i.e., $\Lambda_{\X}\succ\mathbf{0}$), and similarly $\V_{\X^*}\Lambda_{\X^*}\V_{\X^*}^{\top}$ denotes the compact-form eigen-decomposition of $\X^*$. Then, the matrix $\W = (1-\varepsilon)\X + \frac{\varepsilon}{n }\I$ satisfies $\breg(\X^*,\W) \leq R^2$. 

Moreover, if $\rank(\X) = \rank(\X^*)$ Then, the condition in \eqref{eq:init:highRank} can be replaced with
\begin{align}\label{eq:init:equalRank}
\trace(\X^*\log(\X^*) - \X^*\V_{\X}\log(\Lambda_{\X})\V_{\X}^{\top}) +  \frac{2\lambda_1(\X^*)}{\lambda_{r^*}(\X^*)^2}\log\left({\frac{n}{\varepsilon}}\right)\Vert{\X^*-\X}\Vert_F^2  + 4\varepsilon \leq R^2.
\end{align}
\end{lemma}

\begin{proof}
From simple calculations we have
\begin{align*}
-\log(\W) &= -\log\left({(1-\varepsilon)\X + \frac{\varepsilon}{n}\I}\right)  \\
&= -\V_{\X}\log\left({(1-\varepsilon)\Lambda_{\X}+\frac{\varepsilon}{n}\I}\right)\V_{\X}^{\top} - \log\left({\frac{\varepsilon}{n}}\right)(\I-\V_{\X}\V_{\X}^{\top}) \\
&\preceq -\V_{\X}\log\left({(1-\varepsilon)\Lambda_{\X}}\right)\V_{\X}^{\top} - \log\left({\frac{\varepsilon}{n}}\right)(\I-\V_{\X}\V_{\X}^{\top})\\
&=-\V_{\X}\left({\log(1-\varepsilon)\I+\log(\Lambda_{\X})}\right)\V_{\X}^{\top} + \log\left({\frac{n}{\varepsilon}}\right)(\I-\V_{\X}\V_{\X}^{\top}) \\
&\preceq -\log(1-\varepsilon)\I - \V_{\X}\log(\Lambda_{\X})\V_{\X}^{\top} + \log\left({\frac{n}{\varepsilon}}\right)(\I-\V_{\X}\V_{\X}^{\top}).
\end{align*}

Thus,
\begin{align*}
B(\X^*,\W) &= \trace(\X^*\log(\X^*) - \X^*\log(\W)) \\
&\leq \trace\left({\X^*\log(\X^*)  -\X^*\V_{\X}\log(\Lambda_{\X})\V_{\X}^{\top}}\right)\\
&~ - \trace\left({\X^*\left({\log(1-\varepsilon)\I-\log\left({\frac{n}{\varepsilon}}\right)(\I-\V_{\X}\V_{\X}^{\top})}\right)}\right)\\
&=\trace(\X^*\log(\X^*) - \X^*\V_{\X}\log(\Lambda_{\X})\V_{\X}^{\top})+\log\left({\frac{n}{\varepsilon}}\right)\trace\left({\X^*(\I-\V_{\X}\V_{\X}^{\top})}\right)-\log(1-\varepsilon).
\end{align*}

We continue to bound the last two terms. 
\begin{align*}
\trace\left({\X^*(\I-\V_{\X}\V_{\X}^{\top})}\right) &= \trace\left({\V_{\X^*}\Lambda_{\X^*}\V_{\X^*}^{\top}(\I-\V_{\X}\V_{\X}^{\top})}\right)\\
&= \trace\left({\Lambda_{\X^*}(\I-\V_{\X^*}^{\top}\V_{\X}\V_{\X}^{\top}\V_{\X^*})}\right)\\
&\leq \lambda_1(\X^*)\trace\left({\I-\V_{\X^*}^{\top}\V_{\X}\V_{\X}^{\top}\V_{\X^*}}\right)\\
&=\lambda_1(\X^*)\left({r^*-\trace\left({\V_{\X^*}^{\top}\V_{\X}\V_{\X}^{\top}\V_{\X^*}}\right)}\right)\\
&= \lambda_1(\X^*)\left({r^*-\Vert{\V_{\X^*}^{\top}\V_{\X}}\Vert_F^2}\right).
\end{align*}

Also, for all $\varepsilon\in(0,3/4]$ we have $- \log(1-\varepsilon) \leq \frac{\varepsilon}{1-\varepsilon} \leq 4\varepsilon$.

Thus, we conclude that
\begin{align*}
B(\X^*,\W) &\leq\trace(\X^*\log(\X^*) - \X^*\V_{\X}\log(\Lambda_{\X})\V_{\X}^{\top}) +   \log\left({\frac{n}{\varepsilon}}\right)\lambda_1(\X^*)\left({r^*-\Vert{\V_{\X^*}^{\top}\V_{\X}}\Vert_F^2}\right)  + 4\varepsilon.
\end{align*}

In case $\rank(\X)=\rank(\X^*)$ we can replace the upper-bound on $\trace\left({\X^*(\I-\V_{\X}\V_{\X}^{\top})}\right)$ with the following:
\begin{align*}
\trace\left({\X^*(\I-\V_{\X}\V_{\X}^{\top})}\right) &= \trace\left({\V_{\X^*}\Lambda_{\X^*}\V_{\X^*}^{\top}(\I-\V_{\X}\V_{\X}^{\top})}\right)\\
&= \trace\left({\Lambda_{\X^*}(\I-\V_{\X^*}^{\top}\V_{\X}\V_{\X}^{\top}\V_{\X^*})}\right)\\
&\leq \lambda_1(\X^*)\trace\left({\I-\V_{\X^*}^{\top}\V_{\X}\V_{\X}^{\top}\V_{\X^*}}\right)\\
&= \lambda_1(\X^*)\left({r^*-\trace\left({\V_{\X^*}\V_{\X^*}^{\top}\V_{\X}\V_{\X}^{\top}}\right)}\right)\\
&= \frac{\lambda_1(\X^*)}{2}\Vert{\V_{\X^*}\V_{\X^*}^{\top}-\V_{\X}\V_{\X}^{\top}}\Vert_F^2\\
&\leq \frac{2\lambda_1(\X^*)}{\lambda_{r^*}(\X^*)^2}\Vert{\X^*-\X}\Vert_F^2,
\end{align*}
where the last inequality follows from Lemma \ref{lemma:DavisKahan}.

\end{proof}

\section{Provable Local Convergence with Deterministic Updates under Strict Complementarity} \label{section:smoothCase}

In this section we prove a convergence rate of $O(1/t)$ for our  Algorithm \ref{alg:LRMD} from ``warm-start" initialization when exact gradients are used and the strict complementarity condition holds.

In the following lemma we show that there is a decrease in the Bregman distance between the iterates of the algorithm and the optimal solution up to a small error. This step is important since as discussed in Section \ref{section:localConvergence} we can only provably control the errors due to the low rank updates in a certain ball around an optimal solution, and hence, we must show that when the algorithm is initialized inside this ball, the iterates do not leave it. 

\begin{lemma} \label{lemma:bregmanDecrease}
Let $\X^*\in\mX^*$ be an optimal solution for which Assumption \ref{ass:strictcomp} holds with $\delta >0$. Let $\{\X_t\}_{t\ge1}$ be the sequence generated by the low rank MEG algorithm, Algorithm \ref{alg:LRMD}, on Problem \eqref{eq:Model} with $n\not=r\ge r^*=\rank(\X^*)$, $\eta_t\le\frac{1}{\beta}$ and $\varepsilon_t\in(0,3/4]$ such that $\lambda_{r^*+1}(\X_t)\le\frac{\varepsilon_{t-1}}{n-r}$ for all $t\ge1$. Let $\Z_{t+1}$ be the standard MEG step update to $\X_t$, as defined in \eqref{eq:MDvonNeumannUpdate}. Then for any $t\ge1$, if $\X_t$ satisfies the condition in \eqref{eq:StochasticRadiusOfconvergence} with $\xi=0$, it holds that
\[ \breg(\X^*,\X_{t+1}) \le \breg(\X^*,\X_{t})+2\varepsilon_t. \]
\end{lemma}

\begin{proof}
From the optimality of $\X^*$, $\beta$-smoothness of $f$ and plugging in \eqref{ineq:strongConvexityOfBregmanDistance} due to the $1$-strong convexity of $\breg(\cdot,\cdot)$, it holds that
\begin{align*}
f(\X^*) & \le f(\Z_{t+1}) \le f(\X_t)+\langle \nabla f(\X_t),\Z_{t+1}-\X_t\rangle+\frac{\beta}{2}\Vert\Z_{t+1}-\X_t\Vert_F^2
\\ & \le f(\X_t)+\langle \nabla f(\X_t),\Z_{t+1}-\X_t\rangle+\frac{\beta}{2}\Vert\Z_{t+1}-\X_t\Vert_*^2
\\ & \le f(\X_t)+\langle \nabla f(\X_t),\Z_{t+1}-\X_t\rangle+\beta\breg(\Z_{t+1},\X_{t}).
\end{align*}

Therefore,
\begin{align} \label{ineq:partOfFejerProof}
\langle \nabla f(\X_t),\Z_{t+1}-\X_t\rangle \ge f(\X^*)-f(\X_t)-\beta\breg(\Z_{t+1},\X_{t}).
\end{align}

By the definition of $\Z_{t+1}$ in \eqref{eq:MDGeneralUpdate}, 
\[ \Z_{t+1}=\argmin_{\Z\in\mathcal{S}_n}\{\langle\eta_t\nabla f(\X_t)-\nabla\omega(\X_t),\Z\rangle+\omega(\Z)\}. \]

Therefore, by the optimality condition for $\Z_{t+1}$ it holds that
\begin{align*}
0 & \ge \langle -\eta_t\nabla f(\X_t)+\nabla\omega(\X_t)-\nabla\omega(\Z_{t+1}),\X^*-\Z_{t+1}\rangle
\\ & = \langle \nabla\omega(\X_t)-\nabla\omega(\Z_{t+1}),\X^*-\Z_{t+1}\rangle + \langle \eta_t\nabla f(\X_t),\Z_{t+1}-\X_t\rangle + \langle \eta_t\nabla f(\X_t),\X_t-\X^*\rangle
\\ & \underset{(a)}{=} \breg(\X^*,\Z_{t+1})+\breg(\Z_{t+1},\X_{t})-\breg(\X^*,\X_{t}) + \langle \eta_t\nabla f(\X_t),\Z_{t+1}-\X_t\rangle + \langle \eta_t\nabla f(\X_t),\X_t-\X^*\rangle 
\\ & \underset{(b)}{\ge} \breg(\X^*,\Z_{t+1})+\breg(\Z_{t+1},\X_{t})-\breg(\X^*,\X_{t}) + \eta_t\left[f(\X^*)-f(\X_t)-\beta\breg(\Z_{t+1},\X_{t})\right]
\\ & \ \ \ + \langle \eta_t\nabla f(\X_t),\X_t-\X^*\rangle 
\\ & \underset{(c)}{\ge} \breg(\X^*,\Z_{t+1})-\breg(\X^*,\X_{t}) + \eta_t\left(f(\X^*)-f(\X_t) + \langle\nabla f(\X_t),\X_t-\X^*\rangle\right).
\end{align*}
Here (a) follows from the three point lemma in \eqref{lemma:threePointLemma}, (b) follows from plugging in \eqref{ineq:partOfFejerProof}, and (c) holds since $\eta_t\le\frac{1}{\beta}$.

From the gradient inequality $f(\X^*)-f(\X_t) + \langle\nabla f(\X_t),\X_t-\X^*\rangle\ge0$, and therefore, 
\[ \breg(\X^*,\Z_{t+1}) \le \breg(\X^*,\X_{t}). \]

If $\X_t$ satisfies the condition in \eqref{eq:StochasticRadiusOfconvergence} with $\xi=0$ and an exact gradient, then from Lemma \ref{lemma:StochasticRadius}
\begin{align*}
\breg(\X^*,\X_{t+1}) = \breg(\X^*,\Z_{t+1})+\breg(\X^*,\X_{t+1})-\breg(\X^*,\Z_{t+1}) \le \breg(\X^*,\X_{t})+2\varepsilon_t.
\end{align*}

\end{proof}

We can now finally derive a concrete convergence rate for Algorithm \ref{alg:LRMD} with exact gradients. For simplicity and ease of presentation, in the following theorem we consider the case in which the SVD rank parameter in Algorithm \ref{alg:LRMD} is set exactly to that of a corresponding optimal solution, i.e., $r=\rank(\X^*)$. In the sequel (see Theorem \ref{thm:smoothCaseGen}) we consider the more general case in which $r\geq \rank(\X^*)$.

\begin{theorem}\label{thm:smoothCaseDeminish}
Fix an optimal solution $\X^*$ to Problem \eqref{eq:Model} for which Assumption \ref{ass:strictcomp} holds with parameter $\delta>0$ and denote $r^*=\rank(\X^*)$.
Let $\{\X_t\}_{t\geq 1}$ be the sequence of iterates generated by Algorithm \ref{alg:LRMD} with deterministic updates and with a fixed step-size $\eta_t=\eta=1/\beta$, and with SVD rank parameter $r=r^*$. Suppose that for all $t\geq 0$: $\varepsilon_t=\frac{\tilde{\varepsilon}_0}{2\max\{G^2,1\}}\frac{1}{(t+1+c)^3}$ for some $\tilde{\varepsilon}_0\leq R_0^2$,
where
\begin{align*}
R_0 := \frac{1}{4}\left[2\beta+\left(1+\frac{2\sqrt{2r^*}}{\lambda_{r^*}(\X^*)}\right)G\right]^{-1}\delta,
\end{align*}
$c\ge 6\beta/\delta$, and $G\geq\sup_{\X\in\mathcal{S}_n}\Vert \nabla f(\X)\Vert_2$. Finally, assume the initialization matrix $\X_0$ satisfies $\rank(\X_0) = r^*$ and the parameters $\X_0,\X^*,\tilde{\varepsilon}_0,R_0$ satisfy the initialization condition in Eq. \eqref{eq:init:equalRank} (substituting $\X=\X_0$, $\varepsilon=\tilde{\varepsilon}_0$, $R=R_0$).
Then, for any $T\geq 1$ it holds that
\begin{align*}
\min_{1\le t\le {T}}f(\X_t) -f(\X^*) & \le \left(\left(1+\frac{1}{2\max\{G^2,1\}}\right)\beta{}R_0^2+4R_0\right)\frac{1}{T}.
\end{align*}
\end{theorem}

\begin{proof}
We first observe that by our choice of the sequence $\{\varepsilon_t\}_{t\geq 0}$ we have that the sequence $\left\{-\frac{1}{\eta}\log\left(\frac{\varepsilon_{t-1}}{\varepsilon_{t}}\right)\right\}_{t\geq 1}$ is monotone non-decreasing and thus, for all $t\geq 1$:
\begin{align*}
 -\frac{1}{\eta}\log\left(\frac{\varepsilon_{t-1}}{\varepsilon_{t}}\right) &\ge -\frac{1}{\eta}\log\left(\frac{\varepsilon_{0}}{\varepsilon_{1}}\right)=-3\beta\log\left(\frac{c+2}{c+1}\right) \geq -3\beta\log\left(1 + \frac{1}{c}\right) \geq -\frac{3\beta}{c} \geq -\frac{\delta}{2},
\end{align*}
where the last inequality follows from plugging our choice for $c$.

Thus, we have that for all $t\geq 1$:
\begin{align}\label{eq:thm:smoothConv:1}
&\frac{1}{2\sqrt{2}}\left[2\beta+\left(1+\frac{2\sqrt{2r^*}}{\lambda_{r^*}(\X^*)}\right)G\right]^{-1}\delta \le \frac{1}{\sqrt{2}}\left[2\beta+\left(1+\frac{2\sqrt{2r^*}}{\lambda_{r^*}(\X^*)}\right)G\right]^{-1}\left(\delta-\frac{1}{\eta}\log\left(\frac{\varepsilon_{t-1}}{\varepsilon_{t}}\right)\right) .
\end{align}

Thus, in order to invoke Lemma \ref{lemma:StochasticRadius} for all $t\geq 1$ it suffices to prove that for all $t\geq 1$: i. $\lambda_{r^*+1}(\X_t) \leq \frac{\varepsilon_{t-1}}{n-r}$ and ii. $\sqrt{B(\X^*,\X_t)} \leq \textrm{LHS of } \eqref{eq:thm:smoothConv:1}$.

The requirement $\lambda_{r^*+1}(\X_t) \leq \frac{\varepsilon_{t-1}}{n-r}$ holds trivially by the design of the algorithm and since the SVD parameter satisfies $r=r^*$. 

We now prove by induction that indeed for all $t\geq 1$, $\sqrt{B(\X^*,\X_t)} \leq \textrm{LHS of } \eqref{eq:thm:smoothConv:1}$. The base case $t=1$ clearly holds due to the choice of initialization. Now, if the assumption holds for all $i\in\{1,...t\}$, then invoking Lemma \ref{lemma:bregmanDecrease} for all $i=1,...,t$ we have that
\begin{align} \label{ineq:recursionOfRadious}
\breg(\X^*,\X_{t+1}) & \le \breg(\X^*,\X_{t})+2\varepsilon_{t} \le \ldots \le \breg(\X^*,\X_{1})+2\sum_{i=1}^{t}\varepsilon_i \nonumber
\\ &  \leq \breg(\X^*,\X_{1})+\frac{\tilde{\varepsilon}_0}{\max\{G^2,1\}}\sum_{i=1}^{\infty}\frac{1}{(i+1)^3}  \le \breg(\X^*,\X_{1})+\tilde{\varepsilon}_0 \leq 2R_0^2,
\end{align}
where the last inequality follows from our initialization assumption and Lemma \ref{lem:warmstart}.

Thus, we obtain that 
\begin{align*}
\sqrt{\breg(\X^*,\X_{t+1})} & \le \sqrt{2}R_0  = \frac{1}{2\sqrt{2}}\left[2\beta+\left(1+\frac{2\sqrt{2r}}{\lambda_r(\X^*)}\right)G\right]^{-1}\delta,
\end{align*}
and the induction holds.

Invoking Lemma \ref{lemma:StochasticRadius} for all $t\geq 1$ guarantees that for all $t\ge 1$, $\breg(\Z_{t+1},\X_{t+1})\le 2\varepsilon_t$ and $\breg(\X^*,\X_{t+1})-\breg(\X^*,\Z_{t+1})\le 2\varepsilon_t$. Plugging-in $\eta=1/\beta$ and our initialization assumption, it holds from Theorem \ref{thm:converganceSmooth} that for all $T\geq 1$,
\begin{align*}
\min_{1\le t\le {T}}f(\X_t) -f(\X^*) & \le \frac{\beta{}R_0^2}{T}+\frac{2\beta}{T}\sum_{t=1}^{T-1}\varepsilon_t+\frac{2\sqrt{2}G}{T}\sum_{t=1}^{T-1}\sqrt{\varepsilon_t}.
\end{align*}

In order to bound the RHS of this inequality, we note that the following inequalities hold:
\begin{align*}
\sum_{t=1}^{T-1}\sqrt{\varepsilon_t} &\leq \sum_{t=1}^{T-1}\frac{R_0}{\sqrt{2}\max\{G,1\}}\frac{1}{(t+1+c)^{3/2}} < \frac{R_0}{\sqrt{2}\max\{G,1\}}\int_0^{\infty}\frac{1}{(t+1)^{3/2}}dt = \frac{\sqrt{2}R_0}{\max\{G,1\}}.\\
\sum_{t=1}^{T-1}\varepsilon_t &\leq \sum_{t=1}^{T-1}\frac{R_0^2}{2\max\{G^2,1\}}\frac{1}{(t+1+c)^{3}} < \frac{R_0^2}{2\max\{G^2,1\}}\int_0^{\infty}\frac{1}{(t+1)^{3}}dt = \frac{R_0^2}{4\max\{G^2,1\}}.
\end{align*}

Plugging-in these  bounds we obtain the rate in the theorem.
\end{proof}

One caveat of Theorem \ref{thm:smoothCaseDeminish} is that it requires exact knowledge of the rank of the optimal solution since it requires to set $r=\rank(\X^*)$ in Algorithm \ref{alg:LRMD}. We will now present an analysis that removes this limitation. 
  
Towards this end, we observe that from Lemma \ref{lemma:StochasticRadius}, in order for any $n\not=r>r^*$ to be suitable, it needs to satisfy the condition $\lambda_{r^*+1}(\X_t)\le\frac{\varepsilon_{t-1}}{n-r}$ (in the case of $r=r^*$ we have shown that it holds trivially). Thus,  we begin by upper-bounding $\lambda_{r^*+1}(\X_{t+1})$.
\begin{lemma}
Let $\X^*\in\mX^*$ be an optimal solution which satisfies Assumption \ref{ass:strictcomp} with some $\delta>0$ and denote $r^*=\rank(\X^*)$. Suppose the rank parameter in Algorithm \ref{alg:LRMD} satisfies $n\not=r>r^*$. Denote $G=\sup_{\X\in\mathcal{S}_n}\Vert \nabla f(\X)\Vert_2$. Then for any $t\ge1$ it holds that 
\begin{align} \label{ineq:boundOnLambda_rstarplus1}
\lambda_{r^*+1}(\X_{t+1}) \le \max\left\lbrace(1-\varepsilon_t)\lambda_{r^*+1}(\X_t)\exp\left(-\eta_t\delta+\eta_t \gamma \sqrt{\breg(\X^*,\X_t)}+\breg(\X^*,\X_t)\right),\frac{\varepsilon_t}{n-r}\right\rbrace,
\end{align}
where $\gamma=\sqrt{2}\left(2\beta+\frac{2\sqrt{2{r^*}}}{\lambda_{r^*}(\X^*)}G\right)$.

\end{lemma}

\begin{proof}

We have seen in \eqref{ineq:StochasticInProofSecondTermWithoutAssumption} with $\xi=0$ that 
\begin{align*}
\log(\lambda_{r^*+1}(\Y_t)) \le \log\left(\lambda_{r^*+1}(\X_t)\right)-\eta_t \lambda_{n-r^*}(\nabla f(\X^*))+\eta_t\frac{2\sqrt{2r^*}G\Vert\X_t-\X^*\Vert_{*}}{\lambda_{r^*}(\X^*)} \nonumber
+\eta_t\beta\Vert\X_t-\X^*\Vert_*.
\end{align*}

In addition,
\begin{align*}
-\log\left(\sum_{i=1}^r\lambda_i(\Y_t)\right) & \le -\log\left(\sum_{i=1}^{r^*}\lambda_i(\Y_t)\right) = -\log\left(\sum_{i=1}^{r^*}\lambda_i\left(\exp\left(\log(\X_t)-\eta_t\nabla f(\X_t)\right)\right)\right)
\\ & \underset{(a)}{\le} -\sum_{i=1}^{r^*}\lambda_i(\X^*)\cdot\log\left(\frac{\lambda_i\left(\exp\left(\log(\X_t)-\eta_t\nabla f(\X_t)\right)\right)}{\lambda_i(\X^*)}\right)
\\ & = -\sum_{i=1}^{r^*}\lambda_i(\X^*)\cdot\left[\log\left(\lambda_i\left(\exp\left(\log(\X_t)-\eta_t\nabla f(\X_t)\right)\right)-\log(\lambda_i(\X^*)\right)\right]
\\ & = -\sum_{i=1}^{r^*}\lambda_i(\X^*)\lambda_i\left(\log(\X_t)-\eta_t\nabla f(\X_t)\right)+\sum_{i=1}^{r^*}\lambda_i(\X^*)\log(\lambda_i(\X^*))
\\ & = -\sum_{i=1}^{n}\lambda_i(\X^*)\lambda_i\left(\log(\X_t)-\eta_t\nabla f(\X_t)\right)+\sum_{i=1}^{r^*}\lambda_i(\X^*)\log(\lambda_i(\X^*))
\\ & \underset{(b)}{\le} -\langle\X^*,\log(\X_t)-\eta_t\nabla f(\X_t)\rangle +\sum_{i=1}^{r^*}\lambda_i(\X^*)\lambda_i(\log(\X^*))
\\ & \underset{(c)}{=} -\langle\X^*,\log(\X_t)-\eta_t\nabla f(\X_t)\rangle +\sum_{i=1}^{r^*}\lambda_i(\X^*\log(\X^*))
\\ & = -\trace(\X^*\log(\X_t))+\trace(\X^*\log(\X^*))+\eta_t\langle\X^*,\nabla f(\X_t)\rangle 
\\ & = \breg(\X^*,\X_t)+\eta_t\langle\X^*,\nabla f(\X^*)\rangle + \eta_t\langle\X^*,\nabla f(\X_t)-\nabla f(\X^*)\rangle
\\ & = \breg(\X^*,\X_t)+\eta_t\lambda_{n}(\nabla f(\X^*)) + \eta_t\langle\X^*,\nabla f(\X_t)-\nabla f(\X^*)\rangle
\\ & \le \breg(\X^*,\X_t)+\eta_t\lambda_{n}(\nabla f(\X^*)) +\eta_t\Vert\X^*\Vert_*\cdot\Vert\nabla f(\X_t)-\nabla f(\X^*)\Vert_2
\\ & \underset{(d)}{\le} \breg(\X^*,\X_t)+\eta_t\lambda_{n}(\nabla f(\X^*)) +\eta_t\Vert\nabla f(\X_t)-\nabla f(\X^*)\Vert_2
\\ & \le \breg(\X^*,\X_t)+\eta_t\lambda_{n}(\nabla f(\X^*))+\eta_t\beta\Vert\X_t-\X^*\Vert_*
\end{align*}
where (a) follows since $-\log(\cdot)$ is convex, (b) follows from the von Neumann inequality, (c) follows since $\X_t$ and $\log(\X_t)$ have the same eigen-vectors, and (d) follows since $\Vert\X^*\Vert_*=1$.

Then
\begin{align*}
\log\left(\frac{\lambda_{r^*+1}(\Y_t)}{a_t}\right) & = \log(\lambda_{r^*+1}(\Y_t))-\log\left(\sum_{i=1}^{r^*}\lambda_i(\Y_t)\right) 
\\ & \le \log\left(\lambda_{r^*+1}(\X_t)\right)-\eta_t \delta +2\eta_t\left(\beta+\frac{\sqrt{2r^*}G}{\lambda_{r^*}(\X^*)}\right)\Vert\X-\X^*\Vert_*+\breg(\X^*,\X_t)
\\ & \le \log\left(\lambda_{r^*+1}(\X_t)\right)-\eta_t\delta+\eta_t \gamma \sqrt{\breg(\X^*,\X_t)}+\breg(\X^*,\X_t),
\end{align*}
where the last inequality holds from \eqref{ineq:strongConvexityOfBregmanDistance} and $\gamma=\sqrt{2}\left(2\beta+\frac{2\sqrt{2{r^*}}}{\lambda_{r^*}(\X^*)}G\right)$.

Therefore,
\begin{align*}
\frac{\lambda_{r^*+1}(\Y_t)}{a_t} & \le \exp\left(\log\left(\lambda_{r^*+1}(\X_t)\right)-\eta_t\delta+\eta_t \gamma \sqrt{\breg(\X^*,\X_t)}+\breg(\X^*,\X_t)\right) 
\\ & = \lambda_{r^*+1}(\X_t)\exp\left(-\eta_t\delta+\eta_t \gamma \sqrt{\breg(\X^*,\X_t)}+\breg(\X^*,\X_t)\right).
\end{align*}

Using the definition of the update of $\X_{t+1}$,
\begin{align*}
\lambda_{r^*+1}(\X_{t+1}) & = \max\left\lbrace(1-\varepsilon_t)\frac{\lambda_{r^*+1}(\Y_t)}{a_t},\frac{\varepsilon_t}{n-r}\right\rbrace
\\ & \le \max\left\lbrace(1-\varepsilon_t)\lambda_{r^*+1}(\X_t)\exp\left(-\eta_t\delta+\eta_t \gamma \sqrt{\breg(\X^*,\X_t)}+\breg(\X^*,\X_t)\right),\frac{\varepsilon_t}{n-r}\right\rbrace.
\end{align*}

\end{proof}

\begin{theorem}\label{thm:smoothCaseGen}Fix an optimal solution $\X^*$ to Problem \eqref{eq:Model} for which Assumption \ref{ass:strictcomp} holds with parameter $\delta>0$ and denote $r^*=\rank(\X^*)$. Let $\{\X_t\}_{t\geq 1}$ be the sequence of iterates generated by Algorithm \ref{alg:LRMD} with deterministic updates and with a fixed step-size $\eta_t=\eta=1/\beta$, and with SVD rank parameter $n\not=r\geq{}r^*$. Suppose that for all $t\geq 0$: $\varepsilon_t=\frac{3\tilde{\varepsilon}_0}{2\max\{G^2,1\}}\frac{1}{(t+c+1)^3}$ for some $\tilde{\varepsilon}_0\leq R_0^2$,
where
\begin{align*}
R_0 &:=\min\left\lbrace\frac{1}{4}\left[2\beta+\left(1+\frac{2\sqrt{2r^*}}{\lambda_{r^*}(\X^*)}\right)G\right]^{-1}\delta,~\frac{G}{\beta}\right\rbrace,
\end{align*}
$c\ge\frac{6\max\{\beta,1\}}{\delta}$, and $G\geq\sup_{\X\in\mathcal{S}_n}\Vert \nabla f(\X)\Vert_2$. Finally, assume the initialization matrix $\X_0$ satisfies $\lambda_{r^*+1}(\X_0) \leq r\varepsilon_0/(n(n-r))$ and the parameters $\X_0,\X^*,\tilde{\varepsilon}_0,R_0$ satisfy the initialization condition in Eq. \eqref{eq:init:highRank} (substituting $\X=\X_0$, $\varepsilon=\tilde{\varepsilon}_0$, $R=R_0$).
Then, for any $T\geq 1$ it holds that
\begin{align*}
\min_{1\le t\le {T}}f(\X_t) -f(\X^*) & \le \left(\left(1+\frac{1}{\max\{G^2,1\}}\right)\beta R_0^2+4\sqrt{3}R_0\right)\frac{1}{T}.
\end{align*}
\end{theorem}

\begin{proof}

Similarly to \eqref{eq:thm:smoothConv:1}, it can be seen that
\begin{align}\label{eq:thm:smoothConv:4}
& \min\left\lbrace\frac{1}{\sqrt{2}}\left[2\beta+\left(1+\frac{2\sqrt{2r^*}}{\lambda_{r^*}(\X^*)}\right)G\right]^{-1}\left(\delta-\frac{1}{\eta}\log\left(\frac{\varepsilon_{t-1}}{\varepsilon_{t}}\right)\right),\sqrt{2}\eta G\right\rbrace \nonumber
\\ & \ge \min\left\lbrace\frac{1}{\sqrt{2}}\left[2\beta+\left(1+\frac{2\sqrt{2r^*}}{\lambda_{r^*}(\X^*)}\right)G\right]^{-1}\frac{\delta}{2},\sqrt{2}\eta G\right\rbrace. 
\end{align}

Thus, in order to invoke Lemma \ref{lemma:StochasticRadius} for all $t\geq 1$ it suffices to prove that for all $t\geq 1$: i. $\lambda_{r^*+1}(\X_t) \leq \frac{\varepsilon_{t-1}}{n-r}$ and ii. $\sqrt{B(\X^*,\X_t)} \leq \textrm{RHS of } \eqref{eq:thm:smoothConv:4}$.

We now prove by induction that indeed for all $t\geq 1$, $\sqrt{B(\X^*,\X_t)} \leq \textrm{RHS of } \eqref{eq:thm:smoothConv:4}$ and $\lambda_{r^*+1}(\X_{t})\le\frac{\varepsilon_{t-1}}{n-r}$. The base case holds due to our initialization, and by noticing that from the definition of $\X_1$, the condition $\lambda_{r^*+1}(\X_0) \leq r\varepsilon_0/(n(n-r))$ implies that $\lambda_{r^*+1}(\X_1)\le \frac{r\varepsilon_0}{n(n-r)}+\frac{\varepsilon_0}{n}=\frac{\varepsilon_0}{n-r}$. Now, if the assumptions holds for all $i\in\lbrace 1,\ldots,t-1\rbrace$, then from \eqref{ineq:boundOnLambda_rstarplus1} with $\gamma=\sqrt{2}\left(2\beta+\frac{2\sqrt{2{r^*}}}{\lambda_{r^*}(\X^*)}G\right)$, it holds that
\begin{align*} 
\lambda_{r^*+1}(\X_{t}) 
& \le \max\left\lbrace(1-\varepsilon_{t-1})\lambda_{r^*+1}(\X_{t-1})\exp\left(-\eta\delta+\eta \gamma \sqrt{\breg(\X^*,\X_{t-1})}+\breg(\X^*,\X_{t-1})\right),\frac{\varepsilon_{t-1}}{n-r}\right\rbrace
\\ & \le \max\left\lbrace\frac{\varepsilon_{t-2}}{n-r}\exp\left(-\eta\delta+\eta \gamma \sqrt{\breg(\X^*,\X_{t-1})}+\breg(\X^*,\X_{t-1})\right),\frac{\varepsilon_{t-1}}{n-r}\right\rbrace
\\ & \underset{(a)}{\le} \max\left\lbrace\frac{\varepsilon_{t-2}}{n-r}\exp\left(-\eta\delta+\eta \gamma \sqrt{\breg(\X^*,\X_{t-1})}+\sqrt{2}\eta G\sqrt{\breg(\X^*,\X_{t-1})}\right),\frac{\varepsilon_{t-1}}{n-r}\right\rbrace
\\ & \underset{(b)}{\le} \max\left\lbrace\frac{\varepsilon_{t-2}}{n-r}\exp\left(-\log\left(\frac{\varepsilon_{t-2}}{\varepsilon_{t-1}}\right)\right),\frac{\varepsilon_{t-1}}{n-r}\right\rbrace
 = \frac{\varepsilon_{t-1}}{n-r},
\end{align*}
where (a) holds due to the induction hypothesis which implies that $\sqrt{\breg(\X^*,\X_{t-1})}\le\sqrt{2}\eta G$, and (b) also holds due to the induction hypothesis  which implies that $$\sqrt{\breg(\X^*,\X_{t-1})}\le\frac{1}{\sqrt{2}}\left[2\beta+\left(1+\frac{2\sqrt{2r^*}}{\lambda_{r^*}(\X^*)}\right)G\right]^{-1}\left(\delta-\frac{1}{\eta}\log\left(\frac{\varepsilon_{t-2}}{\varepsilon_{t-1}}\right)\right).$$  

Invoking Lemma \ref{lemma:StochasticRadius}, it holds that $\breg(\X^*,\X_{i}) \le \breg(\X^*,\Z_{i})+2\varepsilon_{i-1}$ for all $i \in \lbrace1,\ldots,t\rbrace$. 
Invoking Lemma \ref{lemma:bregmanDecrease} and using recursion, we obtain
\begin{align*} 
\breg(\X^*,\X_{t}) & \le \breg(\X^*,\X_{t-1})+2\varepsilon_{t-1} \le \ldots \le \breg(\X^*,\X_{1})+2\sum_{i=1}^{t-1}\varepsilon_i \nonumber
\\ & \le \breg(\X^*,\X_{1})+\frac{3\tilde{\varepsilon}_0}{\max\{G^2,1\}}\sum_{i=1}^{T}\frac{1}{(t+c+1)^3} 
\\ & \le \breg(\X^*,\X_{1})+3\tilde{\varepsilon}_0\sum_{i=1}^{\infty}\frac{1}{(t+1)^3} \le \breg(\X^*,\X_{1})+\tilde{\varepsilon}_0\le 2R_0^2,
\end{align*} 
where the last inequality holds from our initialization assumption and Lemma \ref{lem:warmstart}.

Thus, we obtain that for all $t\ge1$:
\begin{align*}
& \sqrt{\breg(\X^*,\X_t)}\le \sqrt{2}R_0 = \min\left\lbrace\frac{1}{\sqrt{2}}\left[2\beta+\left(1+\frac{2\sqrt{2r^*}}{\lambda_{r^*}(\X^*)}\right)G\right]^{-1}\frac{\delta}{2},\sqrt{2}\eta G\right\rbrace,
\end{align*}
and the induction holds.

Invoking Lemma \ref{lemma:StochasticRadius} for all $t\ge1$ guarantees that for all $t\ge1$,
$\breg(\X^*,\X_{t+1}) - \breg(\X^*,\Z_{t+1}) \le 2\varepsilon_t$ and $\breg(\Z_{t+1},\X_{t+1})\le 2\varepsilon_t$.

In addition, we note that the following inequalities hold for any $T\ge 1$:
\[ \sum_{t=1}^{T-1}\sqrt{\varepsilon_t}=\frac{\sqrt{3\tilde{\varepsilon}_0}}{\sqrt{2}\max\{G,1\}}\sum_{t=1}^{T-1} \frac{1}{(t+c+1)^{1.5}} \le \frac{\sqrt{3\tilde{\varepsilon}_0}}{\sqrt{2}\max\{G,1\}}\sum_{t=1}^{\infty} \frac{1}{(t+1)^{1.5}} < \frac{\sqrt{6}R_0}{\max\{G,1\}}, \]
\[ \sum_{t=1}^{T-1}\varepsilon_t=\frac{3\tilde{\varepsilon}_0}{2\max\{G^2,1\}}\sum_{t=1}^{T-1} \frac{1}{(t+c+1)^3}\le \frac{3\tilde{\varepsilon}_0}{2\max\{G^2,1\}}\sum_{t=1}^{\infty} \frac{1}{(t+1)^3} < \frac{R_0^2}{2\max\{G^2,1\}}.\]

Plugging in the last two bounds and our choice $\eta=\frac{1}{\beta}$ into Theorem \ref{thm:converganceSmooth}, we obtain that for all $T\ge1$,
\begin{align*}
\min_{1\le t\le {T}}f(\X_t) -f(\X^*) & \le \frac{\beta\breg(\X^*,\X_{1})}{T}+\frac{2\beta}{T}\sum_{t=1}^{T-1}\varepsilon_t+\frac{2\sqrt{2}G}{T}\sum_{t=1}^{T-1}\sqrt{\varepsilon_t}
\\ & \le \frac{\beta R_0^2}{T}+\frac{\beta R_0^2}{\max\{G^2,1\}T}+\frac{4\sqrt{3}R_0}{T}.
\end{align*}

\end{proof}

\section{Provable Local Convergence with Stochastic Updates under Strict Complementarity} \label{section:stochasticCase}

In this section we turn to consider the stochastic setting in which $f(\cdot)$ is given by a first-order stochastic oracle that for any $\X\in\mathbb{S}^n$ it returns a random matrix $\widehat{\nabla}\in\mathbb{S}^n$ such that $\E[\widehat{\nabla}\vert\X]=\nabla f(\X)$, $\Vert\widehat{\nabla}\Vert\le G$ and $\var[\widehat{\nabla}\vert\X] = \E[\Vert{\widehat{\nabla}-\nabla{}f(\X)}\Vert^2|\X]\le\sigma^2$, for some $G>0$ and $\sigma>0$. 

Our main result for this section is the proof that when initialized with a ``warm-start" point and suitable choice of parameters and assuming strict complementarity holds, Algorithm \ref{alg:LRMD} with mini-batches of stochastic gradients converges in expectation  to an optimal solution with rate $O(1/\sqrt{t})$. 

Using Algorithm \ref{alg:LRMD} with mini-batches is crucial to the application of Lemma \ref{lemma:StochasticRadius}, in order to bound the gradient error parameter $\xi$ (at least with some positive probability). Thus, throughout this section we consider Algorithm \ref{alg:LRMD} with stochastic updates and with a fixed mini-batch size $L$, where the stochastic gradient on iteration $t$ is given by $\widehat{\nabla}_t = \frac{1}{L}\sum_{i=1}^L\widehat{\nabla}_t^{(i)}$, such that $\widehat{\nabla}_t^{(1)},\dots,\widehat{\nabla}_t^{(L)}$ denote $L$ i.i.d. queries to the stochastic oracle with the current iterate $\X_t$. Naturally, this is important since the variance of the mini-batched gradient satisfies $\var[\widehat{\nabla}_t|\X_t] = \sigma^2/L$.

In the following lemma we bound the decrease in the Bregman distance between an iterate and the expectation of its standard stochastic MEG update step computed upon it.

\begin{lemma} \label{lemma:StochasticBregmanDecrease}
Let $\{\X_t\}_{t\ge1}$ be the sequence generated by Algorithm \ref{alg:LRMD} with stochastic updates, and let $\X^*$ be an optimal solution. For all $t\geq 1$, let $\Z_{t+1}$ be the standard stochastic MEG update to $\X_t$, as defined in \eqref{eq:MDvonNeumannStochasticUpdate}. Then for any $t\ge1$,  
it holds that
\[ \E[\breg(\X^*,\Z_{t+1})\vert\X_t] \le \breg(\X^*,\X_{t})+\frac{1}{2}\eta_t^2G^2. \]
\end{lemma}

\begin{proof}

By the definition of $\Z_{t+1}$ in \eqref{eq:MDGeneralUpdate}, where the exact gradient is replaced with a stochastic gradient $\widehat{\nabla}_t$,
\[ \Z_{t+1}=\argmin_{\Z\in\mathcal{S}_n}\{\langle\eta_t\widehat{\nabla}_t-\nabla\omega(\X_t),\Z\rangle+\omega(\Z)\}. \]

Therefore, by the optimality condition for $\Z_{t+1}$, it holds that
\begin{align*}
0 & \ge \langle -\eta_t\widehat{\nabla}_t+\nabla\omega(\X_t)-\nabla\omega(\Z_{t+1}),\X^*-\Z_{t+1}\rangle
\\ & = \langle \nabla\omega(\X_t)-\nabla\omega(\Z_{t+1}),\X^*-\Z_{t+1}\rangle + \langle \eta_t\widehat{\nabla}_t,\Z_{t+1}-\X_t\rangle + \langle \eta_t\widehat{\nabla}_t,\X_t-\X^*\rangle
\\ & \underset{(a)}{=} \breg(\X^*,\Z_{t+1})+\breg(\Z_{t+1},\X_{t})-\breg(\X^*,\X_{t}) + \langle \eta_t\widehat{\nabla}_t,\Z_{t+1}-\X_t\rangle + \langle \eta_t\widehat{\nabla}_t,\X_t-\X^*\rangle 
\\ & \underset{(b)}{\ge} \breg(\X^*,\Z_{t+1})-\breg(\X^*,\X_{t}) +\langle \eta_t\widehat{\nabla}_t,\X_t-\X^*\rangle +\frac{1}{2}\Vert\Z_{t+1}-\X_t\Vert^2_* - \eta_t\Vert\widehat{\nabla}_t\Vert_2\Vert\Z_{t+1}-\X_t\Vert_*
\\ & \ge \breg(\X^*,\Z_{t+1})-\breg(\X^*,\X_{t}) +\langle \eta_t\widehat{\nabla}_t,\X_t-\X^*\rangle + \min_{a\in\reals}\left\{\frac{1}{2}a^2-a\eta_t\Vert\widehat{\nabla}_t\Vert_{2}\right\}
\\ & = \breg(\X^*,\Z_{t+1})-\breg(\X^*,\X_{t}) +\langle \eta_t\widehat{\nabla}_t,\X_t-\X^*\rangle -\frac{1}{2}\eta_t^2\Vert\widehat{\nabla}_t\Vert_{2}^2
\\ & \underset{(c)}{\ge} \breg(\X^*,\Z_{t+1})-\breg(\X^*,\X_{t}) + \eta_t\left(f(\X^*)-f(\X_t) + \langle\widehat{\nabla}_t,\X_t-\X^*\rangle\right)-\frac{1}{2}\eta_t^2\Vert\widehat{\nabla}_t\Vert_{2}^2.
\end{align*}
where (a) holds from the three point lemma in \eqref{lemma:threePointLemma}, (b) follows from $1$-strong convexity of the Bregman distance as in \eqref{ineq:strongConvexityOfBregmanDistance} and H\"{o}lder's inequality, and (c) follows from the optimality of $\X^*$.

Taking expectation with respect to $\X_t$, we get
\begin{align*}
\E[\breg(\X^*,\Z_{t+1})\vert\X_t] \le \breg(\X^*,\X_{t})-\eta_t\left(f(\X^*)-f(\X_t) + \langle\nabla f(\X_t),\X_t-\X^*\rangle\right)+\frac{1}{2}\eta_t^2\E[\Vert\widehat{\nabla}_t\Vert_{2}^2\vert\X_t]. 
\end{align*}

Using the gradient inequality, this implies that 
\[ \E[\breg(\X^*,\Z_{t+1})\vert\X_t] \le \breg(\X^*,\X_{t})+\frac{1}{2}\eta_t^2G^2. \]

\end{proof}

In our analysis of the deterministic case we used Lemma \ref{lemma:bregmanDecrease} to show that the Bregman distance between the iterates and the optimal solution does not increase over time, up to some small easily controlled error resulting from the inexact low-rank computations. This property does not hold anymore  when using stochastic updates. Therefore, towards obtaining a convergence rate for the stochastic setting, we  introduce a martingale argument to prove that with high probability the iterates of Algorithm \ref{alg:LRMD} remain inside the ball around the optimal solution inside-which the convergence of the algorithm could be guaranteed.

\begin{lemma} \label{lemma:concentrationBound}
Fix $p\in(0,1)$. Let $\lbrace \X_t\rbrace_{t\in [T]}$ be the sequence generated by $T$ iterations of Algorithm \ref{alg:LRMD} with stochastic updates and with fixed mini-batch size $L>0$, and suppose that $\lbrace\eta_t\rbrace_{t\in [T-1]}$ is a non-increasing sequence of non-negative scalars. For any $t\ge1$ let $\Z_{t+1}$ be the standard stochastic MEG update to $\X_t$ as defined in \eqref{eq:MDvonNeumannStochasticUpdate} and $\Z_1=\X_1$. Then for any $\X^*\in\mathcal{X}^*$, if the inequality
\begin{align} \label{ineq:conditionOfConcentrationBoundLemma}
\sqrt{\sum_{i=1}^{T-1}\eta_{i}^2\sigma^2/L}\ge \frac{1}{3}\eta_1G\sqrt{2\log\left(\frac{T}{p}\right)}
\end{align}
holds, then with probability at least $1-p$, it holds  for all $t\in[T]$ that
\begin{align*}
\breg(\X^*,\X_t) & \le \breg(\X^*,\X_1)+\frac{1}{2}G^2\sum_{i=1}^{t-1}\eta_i^2+\sum_{i=1}^{t}(\breg(\X^*,\X_{i})-\breg(\X^*,\Z_{i})) \\
&~+\sqrt{16\sigma^2/L\sum_{i=1}^{T-1}\eta_{i}^2}\sqrt{\log\left(\frac{T}{p}\right)}.
\end{align*}
\end{lemma}

\begin{proof}

For all $t\in[T]$ define the random variable 
\[ W_t:=\breg(\X^*,\Z_t)-\frac{1}{2}G^2\sum_{i=1}^{t-1}\eta_i^2-\sum_{i=1}^{t-1}(\breg(\X^*,\X_{i})-\breg(\X^*,\Z_{i})). \] 
%where $\gamma_i=\frac{1}{2}\eta_i^2G^2+\varepsilon_i$.

$W_1,\ldots,W_T$ form a sub-martingale sequence with respect to the sequence $\X_1,\ldots,\X_{T-1}$, since for all $t\in[T-1]$ it holds that
\begin{align*}
\E[W_{t+1}\vert \X_1,\ldots,\X_{t}] & = \E[\breg(\X^*,\Z_{t+1})\vert\X_{t}]-\frac{1}{2}G^2\sum_{i=1}^{t}\eta_i^2-\sum_{i=1}^{t}\E[\breg(\X^*,\X_{i})-\breg(\X^*,\Z_{i})\vert\X_i]
\\ & = \E[\breg(\X^*,\Z_{t+1})\vert\X_{t}]-\frac{1}{2}G^2\sum_{i=1}^{t}\eta_i^2-\sum_{i=1}^{t}(\breg(\X^*,\X_{i})-\breg(\X^*,\Z_{i}))
\\ & \underset{(a)}{\le} \breg(\X^*,\X_{t})+\frac{1}{2}\eta_t^2G^2-\frac{1}{2}G^2\sum_{i=1}^{t}\eta_i^2-\sum_{i=1}^{t}(\breg(\X^*,\X_{i})-\breg(\X^*,\Z_{i}))
\\ & = \breg(\X^*,\Z_{t})-\frac{1}{2}G^2\sum_{i=1}^{t-1}\eta_i^2-\sum_{i=1}^{t-1}(\breg(\X^*,\X_{i})-\breg(\X^*,\Z_{i})) = W_t,
\end{align*}
where (a) follows from Lemma \ref{lemma:StochasticBregmanDecrease}.

We will show that the sub-martingale has bounded-differences to upper-bound its variance.
It holds for all $2\le t\le T$ that
\begin{align} \label{eq:boundedDifferences}
W_t-\E[W_{t}\vert W_1,\ldots,W_{t-1}] & = \breg(\X^*,\Z_{t})-\E[\breg(\X^*,\Z_{t})\vert\X_{t-1}].
\end{align}

First, note that for any $i\in[n]$ it holds that $$\lambda_i\left(\log\left(\frac{\Y_{t-1}}{b_{t-1}}\right)\right)=\log\left(\lambda_i\left(\frac{\Y_{t-1}}{b_{t-1}}\right)\right)=\log\left(\frac{1}{b_{t-1}}\lambda_i\left(\Y_{t-1}\right)\right)=\log\left(\lambda_i\left(\Y_{t-1}\right)\right)-\log(b_{t-1}).$$
Therefore, $\log\left(\frac{\Y_{t-1}}{b_{t-1}}\right)=\log\left(\Y_{t-1}\right)-\log(b_{t-1}\I)$. This implies that
\begin{align} \label{eq:traceYt/bt}
\trace\left(\X^*\log\left(\Z_t\right)\right) & =\trace\left(\X^*\log\left(\frac{\Y_{t-1}}{b_{t-1}}\right)\right) = \trace\left(\X^*\log\left(\Y_{t-1}\right)\right)-\trace\left(\X^*\log(b_{t-1}\I)\right) \nonumber
\\ & = \trace\left(\X^*\log\left(\Y_{t-1}\right)\right)-\log(b_{t-1}).
\end{align}

Now, by the definition of $\Z_t$ in \eqref{eq:MDvonNeumannStochasticUpdate}, we obtain
\begin{align} \label{ineq:boundedDifferences1}
& \breg(\X^*,\Z_{t})-\E[\breg(\X^*,\Z_{t})\vert\X_{t-1}] \nonumber
\\ & = \E[\trace(\X^*\log(\Z_t))\vert\X_{t-1}]-\trace(\X^*\log(\Z_t)) \nonumber
\\ & \underset{(a)}{=} \E[\trace(\X^*(\log(\X_{t-1})-\eta_{t-1}\widehat{\nabla}_{t-1}))\vert\X_{t-1}]-\trace(\X^*(\log(\X_{t-1})-\eta_{t-1}\widehat{\nabla}_{t-1})) \nonumber \\ & \ \ \ -\E[\log(b_{t-1})\vert\X_{t-1}]+\log(b_{t-1}) \nonumber
\\ & = \eta_{t-1}\trace(\X^*(\widehat{\nabla}_{t-1}-\nabla f(\X_{t-1})))-\E[\log(b_{t-1})\vert\X_{t-1}]+\log(b_{t-1}) \nonumber
\\ & \underset{(b)}{\le} \eta_{t-1}\Vert\X^*\Vert_*\Vert\widehat{\nabla}_{t-1}-\nabla f(\X_{t-1})\Vert_2-\E[\log(b_{t-1})\vert\X_{t-1}]+\log(b_{t-1}) \nonumber
\\ & = \eta_{t-1}\Vert\widehat{\nabla}_{t-1}-\nabla f(\X_{t-1})\Vert_2-\E[\log(b_{t-1})\vert\X_{t-1}]+\log(b_{t-1}),
\end{align}
where (a) follows from \eqref{eq:traceYt/bt}, and (b) follows from H\"{o}lder's inequality.

In addition,
\begin{align} \label{ineq:boundedDifferences2}
\log(b_{t-1}) & = \log\left(\trace(\exp(\log(\X_{t-1})-\eta_{t-1}\widehat{\nabla}_{t-1}))\right) \nonumber
\\ & \underset{(a)}{\le} \log\left(\trace(\exp(\log(\X_{t-1}))\exp(-\eta_{t-1}\widehat{\nabla}_{t-1}))\right) \nonumber
\\ & =\log\left(\trace(\X_{t-1}\exp(-\eta_{t-1}\widehat{\nabla}_{t-1}))\right) \nonumber
\\ & \underset{(b)}{\le} \log\left(\sum_{i=1}^{n}\lambda_i(\X_{t-1})\lambda_i(\exp(-\eta_{t-1}\widehat{\nabla}_{t-1}))\right) \nonumber
\\ & = \log\left(\sum_{i=1}^{n}\lambda_i(\X_{t-1})\exp(\lambda_i(-\eta_{t-1}\widehat{\nabla}_{t-1}))\right) \nonumber
\\ & \le \log\left(\sum_{i=1}^{n}\lambda_i(\X_{t-1})\exp(\eta_{t-1}\Vert\widehat{\nabla}_{t-1}\Vert_2)\right)= \eta_{t-1}\Vert\widehat{\nabla}_{t-1}\Vert_2 \le \eta_{t-1}G,
\end{align}
where (a) follows from the Golden-Thompson inequality\footnote{The Golden-Thompson inequality states that for any two matrices $\A,\B\in\mbS^n$ it holds that $\trace(\exp(\A+\B)) \leq \trace(\exp(\A)\exp(\B))$ \cite{GoldenThompson}.}, and (b) follows from the von Neumann inequality.

Similarly to the proof of \eqref{ineq:middleOfProof1} and using H\"{o}lder's inequality, we get
\begin{align} \label{ineq:boundedDifferences3}
\E[-\log(b_{t-1})\vert\X_{t-1}] & \le \E[\eta_{t-1}\langle \X_{t-1},\widehat{\nabla}_{t-1}\rangle\vert\X_{t-1}] \nonumber
\\ & \le \E[\eta_{t-1}\Vert\X_{t-1}\Vert_*\Vert\widehat{\nabla}_{t-1}\Vert_2\vert\X_{t-1}] \le \eta_{t-1}G.
\end{align}

Plugging \eqref{ineq:boundedDifferences1},\eqref{ineq:boundedDifferences2} and \eqref{ineq:boundedDifferences3} into \eqref{eq:boundedDifferences} we have
\begin{align*}
W_t-\E[W_{t}\vert W_1,\ldots,W_{t-1}] & \le \eta_{t-1}\Vert\widehat{\nabla}_{t-1}-\nabla f(\X_{t-1})\Vert_2+2\eta_{t-1}G
\le 4\eta_{t-1}G.
\end{align*}

We will now upper-bound the conditional variance. For any $2\le t\le T$ it holds that
\begin{align} \label{ineq:varBound}
\var[W_{t}\vert W_1,\ldots,W_{t-1}] & = \var[\breg(\X^*,\Z_{t})\vert \X_{t-1}] \nonumber
\\ & \underset{(a)}{=} \var[\trace(\X^*(\log(\X_{t-1})-\eta_{t-1}\widehat{\nabla}_{t-1}))-\log(b_{t-1})\vert \X_{t-1}] \nonumber
\\ & \underset{(b)}{\le} 2\var[\eta_{t-1}\trace(\X^*\widehat{\nabla}_{t-1})\vert \X_{t-1}]+2\var[\log(b_{t-1})\vert \X_{t-1}],
\end{align}
where (a) follows from \eqref{eq:traceYt/bt}, and (b) follows since $\var(X+Y)\le2\var(X)+2\var(Y)$.

\begin{align} \label{ineq:varBound1}
\var[\eta_{t-1}\trace(\X^*\widehat{\nabla}_{t-1})\vert \X_{t-1}] & = \var[\eta_{t-1}\trace(\X^*\widehat{\nabla}_{t-1})-\eta_{t-1}\trace(\X^*\nabla f(\X_{t-1}))\vert \X_{t-1}] \nonumber
\\ & \le \eta_{t-1}^2\E\left[\left(\trace(\X^*(\widehat{\nabla}_{t-1}-\nabla f(\X_{t-1}))\right)^2\Big\vert \X_{t-1}\right] \nonumber
\\ & \underset{(a)}{\le} \eta_{t-1}^2\E\left[\left(\Vert\X^*\Vert_*\Vert\widehat{\nabla}_{t-1}-\nabla f(\X_{t-1})\Vert_2\right)^2\Big\vert \X_{t-1}\right] \nonumber \\
&\le \eta_{t-1}^2\sigma^2/L,
\end{align}
where (a) follows from H\"{o}lder's inequality.

\begin{align} \label{ineq:varBound2}
& \var[\log(b_{t-1})\vert\X_{t-1}] = \var\left[\log\left(\trace(\exp(\log(\X_{t-1})-\eta_{t-1}\widehat{\nabla}_{t-1}))\right)\Big\vert\X_{t-1}\right] \nonumber
\\ & = \var\left[\log\left(\trace(\exp(\log(\X_{t-1})-\eta_{t-1}\widehat{\nabla}_{t-1}))\right)+\log\left(\exp(\eta_{t-1}\lambda_n(\nabla f(\X_{t-1})))\right)\Big\vert\X_{t-1}\right] \nonumber
\\ & = \var\left[\log\left(\trace(\exp(\log(\X_{t-1})-\eta_{t-1}\widehat{\nabla}_{t-1}))\exp(\eta_{t-1}\lambda_n(\nabla f(\X_{t-1})))\right)\Big\vert\X_{t-1}\right] \nonumber
\\ & \le \E\left[\left(\log\left(\trace(\exp(\log(\X_{t-1})-\eta_{t-1}\widehat{\nabla}_{t-1}))\exp(\eta_{t-1}\lambda_n(\nabla f(\X_{t-1})))\right)\right)^2\Big\vert\X_{t-1}\right] \nonumber
\\ & \underset{(a)}{\le} \E\left[\left(\log\left(\trace(\exp(\log(\X_{t-1}))\exp(-\eta_{t-1}\widehat{\nabla}_{t-1}))\exp(\eta_{t-1}\lambda_n(\nabla f(\X_{t-1})))\right)\right)^2\Big\vert\X_{t-1}\right] \nonumber
\\ & =  \E\left[\left(\log\left(\trace(\X_{t-1}\exp(-\eta_{t-1}\widehat{\nabla}_{t-1}))\exp(\eta_{t-1}\lambda_n(\nabla f(\X_{t-1})))\right)\right)^2\Big\vert\X_{t-1}\right] \nonumber
\\ & \underset{(b)}{\le} \E\left[\left(\log\left(\sum_{i=1}^{n}\lambda_i(\X_{t-1})\lambda_i(\exp(-\eta_{t-1}\widehat{\nabla}_{t-1}))\exp(\eta_{t-1}\lambda_n(\nabla f(\X_{t-1})))\right)\right)^2\Bigg\vert\X_{t-1}\right] \nonumber
\\ & \underset{(c)}{\le} \E\left[\left(\log\left(\sum_{i=1}^{n}\lambda_i(\X_{t-1})\exp(\eta_{t-1}\lambda_i(\nabla f(\X_{t-1})-\widehat{\nabla}_{t-1}))\right)\right)^2\Bigg\vert\X_{t-1}\right] \nonumber
\\ & \le \E\left[\left(\log\left(\sum_{i=1}^{n}\lambda_i(\X_{t-1})\exp(\eta_{t-1}\Vert\widehat{\nabla}_{t-1}-\nabla f(\X_{t-1})\Vert_2)\right)\right)^2\Bigg\vert\X_{t-1}\right] \nonumber
\\ & = \eta_{t-1}^2\E\left[\Vert\widehat{\nabla}_{t-1}-\nabla f(\X_{t-1})\Vert_2^2\big\vert\X_{t-1}\right]
\le \eta_{t-1}^2\sigma^2/L,
\end{align}
where (a) follows from the Golden–Thompson inequality, (b) follows from the von Neumann inequality, and (c) follows from Weyl's inequality.

Plugging \eqref{ineq:varBound1} and \eqref{ineq:varBound2} into \eqref{ineq:varBound}, we get
\begin{align*}
\var[W_{t}\vert W_1,\ldots,W_{t-1}] & \le 4\eta_{t-1}^2\sigma^2/L.
\end{align*}

Now, using a standard concentration argument for sub-martingales (see Theorem 7.3 in \cite{concentrationBound}, which we apply with parameters $\sigma_i=4\eta_{i-1}^2\sigma^2/L$, $\phi_i=0$, $a_i=0$, and $M=\max_{t}\lbrace4\eta_{t-1}G\rbrace=4\eta_{1}G$), we have that for any $\Delta>0$ and $t\in[T]$,
\begin{align*}
\Pr(W_t\ge W_1+\Delta) & \le \exp\left(\frac{-\Delta^2}{\sum_{i=1}^{t-1}8\eta_{i}^2\sigma^2/L+8/3\eta_1G\Delta}\right)
\le \exp\left(\frac{-\Delta^2}{\sum_{i=1}^{T-1}8\eta_{i}^2\sigma^2/L+8/3\eta_1G\Delta}\right).
\end{align*}

Let $\Delta=\sqrt{2\sum_{i=1}^{T-1}8\eta_{i}^2\sigma^2/L}\sqrt{\log\left(\frac{1}{p'}\right)}$, and suppose that it satisfies $\sum_{i=1}^{T-1}8\eta_{i}^2\sigma^2/L \ge 8/3\eta_1G\Delta$, which in turn is equivalent to the condition
\begin{align*}
\sqrt{\sum_{i=1}^{T-1}\eta_{i}^2\sigma^2/L}\ge \frac{1}{3}\eta_1G\sqrt{2\log\left(\frac{1}{p'}\right)}.
\end{align*}

Then, it holds that $\Pr(W_t\ge W_1+\Delta)\le p'$, which equivalently implies with probability at least $1-p'$
\begin{align*}
\breg(\X^*,\Z_t) \le \breg(\X^*,\Z_1)+\frac{1}{2}G^2\sum_{i=1}^{t-1}\eta_i^2+\sum_{i=1}^{t-1}(\breg(\X^*,\X_{i})-\breg(\X^*,\Z_{i}))+\Delta.
\end{align*}

Plugging in $\breg(\X^*,\Z_t)=\breg(\X^*,\X_t)-(\breg(\X^*,\X_t)-\breg(\X^*,\Z_t))$ and recalling that $\Z_1=\X_1$, we have that with probability at least $1-p'$ it holds that
\begin{align*}
\breg(\X^*,\X_t) \le \breg(\X^*,\X_1)+\frac{1}{2}G^2\sum_{i=1}^{t-1}\eta_i^2+\sum_{i=1}^{t}(\breg(\X^*,\X_{i})-\breg(\X^*,\Z_{i}))+\Delta.
\end{align*}

The lemma now follows from setting $p' = p/T$ and using the union-bound.

\end{proof}

We can now finally state and prove our local convergence result for Algorithm \ref{alg:LRMD} with stochastic updates. We prove the local convergence under the assumption that the SVD rank parameter $r$ in Algorithm \ref{alg:LRMD} is set to exactly the rank of the corresponding optimal solution $r^*$ (similarly to our first theorem in the deterministic setting --- Theorem \ref{thm:smoothCaseDeminish}).
 
\begin{theorem} \label{thm:StochasticPutTogether}
Let $\X^*\in\mathcal{X}^*$ be an optimal solution such that $\rank(\X^*)=r^*$ and suppose $\X^*$ satisfies Assumption \ref{ass:strictcomp} with some parameter $\delta >0$. Consider running Algorithm \ref{alg:LRMD} for $T$ iterations with SVD rank parameter $r=r^*$, a fixed minibatch-size $L\ge\max\left\lbrace\frac{16\sigma^2}{G^2R_0^2}\log\left(8T\right),\frac{128G^2}{\delta^2}\log(nT)\right\rbrace$ and a fixed step-size  $\eta=\frac{R_0}{2G\sqrt{T}}$, where
\begin{align*}
R_0 :=  \frac{1}{8}\left[2\beta+\left(1+\frac{2\sqrt{2r^*}}{\lambda_{r^*}(\X^*)}\right)G\right]^{-1}\delta.
\end{align*}
Suppose further that for all $t\geq 0$: $\varepsilon_t=\frac{9}{32}\frac{\tilde{\varepsilon}_0}{(t+c+1)^2}$ and $c\ge\frac{\delta\sqrt{\tilde{\varepsilon}_0}}{16G\sqrt{T}}$ for some  $\tilde{\varepsilon}_0 \leq R_0^2$. Finally, assume the initialization matrix $\X_0$ satisfies $\rank(\X_0) = r^*$, and the parameters $\X_0,\X^*,\tilde{\varepsilon}_0,R_0$ satisfy the initialization condition in Eq. \eqref{eq:init:equalRank} (substituting $\X=\X_0$, $\varepsilon=\tilde{\varepsilon}_0$, $R=R_0$).
Then, for any $T$ sufficiently large and $\bar{\X}\sim\textrm{Uni}\{1,\dots,T\}$ it holds with probability at least $1/2$ that
\[f(\bar{\X})-f(\X^*)=\frac{12GR_0}{\sqrt{T}}.\]
\end{theorem}

\begin{proof}

For all $t\in[T-1]$ it holds that
\[ \Vert  (\nabla f(\X_t)-\widehat{\nabla}_t)^2\Vert \le 2\Vert  \nabla_t\Vert^2+2\Vert\widehat{\nabla}_t\Vert^2 \le 4G^2.\]
Thus, using a standard Hoeffding concentration argument (see for instance \cite{HoeffdingBound}), we have that with a batch-size of $L\ge\frac{128G^2}{\delta^2}\log(nT)$, 
\begin{align*}
\Pr\left(\Vert\nabla f(\X_t)-\widehat{\nabla}_t\Vert\ge\frac{\delta}{4}\right) & \le 2n\cdot\exp\left({-\frac{\delta^2}{16}}\Bigg/{\frac{8G^2}{\frac{128G^2}{\delta^2}\log(nT)}}\right) 
\\ & = 2n\cdot\exp\left(-\log(nT)\right)
= \frac{2}{T}.
\end{align*}
Therefore, for a large enough $T$, with probability at least $9/10$, it holds for all $t\in[T-1]$ that 
\begin{align} \label{xsiBound}
\Vert \widehat{\nabla}_t-\nabla f(\X_t)\Vert_2\le\frac{\delta}{4}.
\end{align}

With a constant step-size $\eta=\frac{R_0}{2G\sqrt{T}}$, and constant batch-size $L\ge\frac{16\sigma^2}{G^2 R_0^2}\log\left(8T\right)$, condition \eqref{ineq:conditionOfConcentrationBoundLemma} holds with $p=1/8$ for any sufficiently large $T$. Thus, invoking Lemma \ref{lemma:concentrationBound}, we have with probability at least $7/8$ it holds that for all $t\in[T]$,
\begin{align} \label{ineq:StochasticRecursionBound}
\breg(\X^*,\X_t) & \le \breg(\X^*,\X_1)+\frac{1}{2}G^2\sum_{i=1}^{t-1}\eta_i^2+\sum_{i=1}^{t}(\breg(\X^*,\X_{i})-\breg(\X^*,\Z_{i})) \nonumber \\
&~+\sqrt{16\sigma^2/L\sum_{i=1}^{T-1}\eta_{i}^2}\sqrt{\log\left(\frac{T}{p}\right)} \nonumber
\\ & \le \frac{13}{8}\breg(\X^*,\X_1)+\sum_{i=1}^{t}(\breg(\X^*,\X_{i})-\breg(\X^*,\Z_{i})).
\end{align}

In addition, note that by our choice of the sequence $\lbrace \varepsilon_t\rbrace_{t\ge0}$, we have that the sequence $\left\lbrace-\frac{1}{\eta}\log\left(\frac{\varepsilon_{t-1}}{\varepsilon_{t}}\right)\right\rbrace_{t\ge1}$ is monotone non-decreasing, and thus, for all $t\ge1$:
\begin{align*}
-\frac{1}{\eta}\log\left(\frac{\varepsilon_{t-1}}{\varepsilon_{t}}\right) & \ge -\frac{1}{\eta}\log\left(\frac{\varepsilon_{0}}{\varepsilon_{1}}\right) = -\frac{4G\sqrt{T}}{R_0}\log\left(\frac{c+2}{c+1}\right)
\ge -\frac{4G\sqrt{T}}{R_0}\log\left(1+\frac{1}{c}\right) \\
&\ge -\frac{4G\sqrt{T}}{R_0}\frac{1}{c} \ge -\frac{\delta}{4},
\end{align*}
where the last inequality follows from plugging in our choice for $c$. 

Thus, we have that for all $t\geq 1$:
\begin{align}\label{eq:thm:smoothConv:8}
&\frac{1}{\sqrt{2}}\left[2\beta+\left(1+\frac{2\sqrt{2r^*}}{\lambda_{r^*}(\X^*)}\right)G\right]^{-1}\frac{\delta}{4} \le \frac{1}{\sqrt{2}}\left[2\beta+\left(1+\frac{2\sqrt{2r^*}}{\lambda_{r^*}(\X^*)}\right)G\right]^{-1}\left(\frac{\delta}{2}-\frac{1}{\eta_t}\log\left(\frac{\varepsilon_{t-1}}{\varepsilon_{t}}\right)\right).
\end{align}

Thus, in order to invoke Lemma \ref{lemma:StochasticRadius} for all $t\geq 1$ it suffices to prove that for all $t\geq 1$: i. $\lambda_{r^*+1}(\X_t) \leq \frac{\varepsilon_{t-1}}{n-r}$ and ii. $\sqrt{B(\X^*,\X_t)} \leq \textrm{LHS of } \eqref{eq:thm:smoothConv:8}$.

The requirement $\lambda_{r^*+1}(\X_t) \leq \frac{\varepsilon_{t-1}}{n-r}$ holds trivially by the design of the algorithm and since the SVD parameter satisfies $r=r^*$. 

We now prove by induction that indeed for all $t\geq 1$, it holds with probability at least $7/8$ that $\sqrt{B(\X^*,\X_t)} \leq \textrm{LHS of } \eqref{eq:thm:smoothConv:8}$.
The base case holds due to the initialization choice.
Now, if the assumption holds for all $i\in[t-1]$ then invoking Lemma \ref{lemma:StochasticRadius} it holds that $\breg(\X^*,\X_{i})-\breg(\X^*,\Z_{i})\le 2\varepsilon_i$ for all $i\in[t]$. From \eqref{ineq:StochasticRecursionBound}, this implies that with probability at least $7/8$,
\begin{align*} %\label{ineq:inStochasticProofRecursion}
\breg(\X^*,\X_t) & \le \frac{13}{8}\breg(\X^*,\X_1)+\sum_{i=1}^{t}2\varepsilon_i \nonumber
=\frac{13}{8}\breg(\X^*,\X_1)+\frac{9\tilde{\varepsilon}_0}{16}\sum_{i=1}^{t}\frac{1}{(i+c+1)^2} 
\\ & \le \frac{13}{8}\breg(\X^*,\X_1)+\frac{9\tilde{\varepsilon}_0}{16}\sum_{i=1}^{\infty}\frac{1}{(i+1)^2} \le 2R_0^2.
\end{align*}

Thus, we obtain that with probability at least $7/8$ for all $t\ge1$: 
\begin{align} \label{BtBound}
\sqrt{\breg(\X^*,\X_t)} & \le \sqrt{2}R_0 = \frac{1}{\sqrt{2}}\left[2\beta+\left(1+\frac{2\sqrt{2r^*}}{\lambda_{r^*}(\X^*)}\right)G\right]^{-1}\frac{\delta}{4}, 
\end{align}
and the induction holds.

By combining \eqref{xsiBound}, \eqref{BtBound}, and invoking Lemma \ref{lemma:StochasticRadius} we obtain with probability at least $3/4$ that for all $t\ge1$,
\begin{align} \label{ineq:boundOnStochasticEpsInProof}
\breg(\X^*,\X_{t})-\breg(\X^*,\Z_{t})\le2\varepsilon_t.
\end{align}

From Theorem \ref{thm:StochasticConvergence}, it holds that for any $\eta>0$, after $T$ iterations,
\begin{align*}
\E\left[f(\bar{\X})\right]-f(\X^*) & \le \frac{\breg(\X^*,\X_1)+\frac{G^2}{2}T\eta^2+\sum_{t=1}^T \E[\breg(\X^*,\X_{t})-\breg(\X^*,\Z_{t})]}{T\eta}.
\end{align*}
Then, from Markov's inequality, it holds with probability at least $3/4$ that
\begin{align} \label{ineq:stochasticBound}
f(\bar{\X})-f(\X^*) & \le 4\frac{\breg(\X^*,\X_1)+\frac{G^2}{2}T\eta^2+\sum_{t=1}^T \E[\breg(\X^*,\X_{t})-\breg(\X^*,\Z_{t})]}{T\eta}.
\end{align}

In addition, note that 
\begin{align} \label{ineq:boundEpsStochastic}
\sum_{t=1}^T 2\varepsilon_{t} & = \frac{9\tilde{\varepsilon}_0}{16}\sum_{t=1}^T \frac{1}{(t+c+1)^2} \le \frac{9\tilde{\varepsilon}_0}{16}\sum_{t=1}^{\infty} \frac{1}{(t+1)^2} \le \frac{3}{8}R_0^2.
\end{align}

Thus, by plugging \eqref{ineq:boundOnStochasticEpsInProof}, \eqref{ineq:boundEpsStochastic}, and our choice for $\eta$ into \eqref{ineq:stochasticBound}, we have with probability at least $1/2$ that
\begin{align*}
f(\bar{\X})-f(\X^*) & \le \frac{12GR_0}{\sqrt{T}}.
\end{align*}
\end{proof}

The following corollary states the overall sample complexity of our method for obtaining $\varepsilon$ approximation error from ``warm-start'' initialization, which, up to logarithmic factors, is optimal in $\varepsilon$.

\begin{corollary}
Under the assumptions of Theorem \ref{thm:StochasticPutTogether}, the overall sample complexity to achieve $f(\bar{\X})-f(\X^*)\le\varepsilon$ with probability at least $1/2$, when initializing from a ``warm-start", is upper-bounded by
\[ \tilde{O}\left(\frac{1}{\varepsilon^2}\max\left\lbrace \sigma^2,G^2\lambda^2_{r^*}(\X^*)\frac{1}{r^*}\right\rbrace\right). \]
\end{corollary}

\begin{proof}
We have from Theorem \ref{thm:StochasticPutTogether}, that with probability $1/2$ in order to achieve $f(\bar{\X})-f(\X^*)\le\varepsilon$, we need to run 
\begin{align*}
\tilde{O}\left(\frac{G^2R_0^2}{\varepsilon^2}\right)
\end{align*}
iterations.

The sample complexity is given by the mini-batch size times the number of iterations it takes to reach an $\varepsilon$ error.
Therefore, it is bounded by 
\begin{align*}
& \tilde{O}\left(\frac{G^2R_0^2}{\varepsilon^2}\cdot 
\max\left\lbrace\frac{\sigma^2}{G^2 R_0^2},\frac{G^2}{\delta^2}\right\rbrace\right)
= \tilde{O}\left(\frac{1}{\varepsilon^2}\max\left\lbrace \sigma^2,\frac{G^4R_0^2}{\delta^2}\right\rbrace\right)
\\ & = \tilde{O}\left(\frac{1}{\varepsilon^2}\max\left\lbrace \sigma^2,G^4\left(\frac{\lambda_{r^*}(\X^*)}{(\beta+G)\lambda_{r^*}(\X^*)+\sqrt{r^*}G}\right)^2\right\rbrace\right)
\\ & = \tilde{O}\left(\frac{1}{\varepsilon^2}\max\left\lbrace \sigma^2,G^4\left(\frac{\lambda_{r^*}(\X^*)}{\sqrt{r^*}G}\right)^2\right\rbrace\right)
 = \tilde{O}\left(\frac{1}{\varepsilon^2}\max\left\lbrace \sigma^2,G^2\lambda^2_{r^*}(\X^*)\frac{1}{r^*}\right\rbrace\right),
\end{align*}
where the second to last equation follows since $\lambda_{r^*}(\X^*)\le\frac{1}{r^*}\le\sqrt{r^*}$.

\end{proof}

\section{Experiments} \label{section:experiments}
In this section we present preliminary empirical evidence in support of our theoretical findings.

We consider the problem of recovering a low-rank matrix from quadratic measurements. Concretely, we consider the following optimization problem:
\begin{align} \label{eq:recoveryProb}
\min_{\X\succeq0\ \trace(\X)=\tau}\left\lbrace f(\X):= \frac{1}{2}\sum_{i=1}^m \left(\trace(\mathbf{a}_i^{\top}\X \mathbf{b}_i)-\y_i\right)^2\right\rbrace.
\end{align}

Throughout this section we focus on the deterministic setting in which full gradients of the objective  function in \eqref{eq:recoveryProb} are available.

We let $\M=(\sqrt{n}\V)(\sqrt{n}\V)^{\top}$ be the ground-truth low-rank matrix,  where $\V\in\reals^{n\times r}$ is generated by taking a random matrix with standard Gaussian entries and then normalizing it to have unit Frobenius norm, i.e., $\Vert\V\Vert_F=1$.
We generate $m$ pairs of random uniformly-distributed unit vectors $\lbrace(\mathbf{a}_i,\mathbf{b}_i)\rbrace_{i=1}^m\subset\reals^n\times\reals^n$. $\y_0\in\reals^m$, the vector of quadratic measurements of $\M$, is given by $\y_0(i)=\mathbf{a}_i^{\top}\M \mathbf{b}_i$. 
We then add noise to produce the noisy vector $\y=\y_0+\n$, where $\n=\kappa\Vert \y_0\Vert_2 \mathbf{v}$ for a random unit vector $\mathbf{v}\in\reals^m$ and $\kappa\in\reals$. In the following experiments, unless stated otherwise, we set $\kappa=1/2$.

In all of our experiments, we scale our final estimate for the ground-truth matrix $\M$ to be $\frac{\trace(\M)}{\tau}\X^*$, where $\tau$ is the trace bound in \eqref{eq:recoveryProb} and $\X^*$ is obtained by solving \eqref{eq:recoveryProb},  in order for it to have the same trace as $\M$.
We measure the relative recovery error by $\left\Vert\frac{\trace(\M)}{\tau}\X^*-\M\right\Vert_F^2\Big/\left\Vert\M\right\Vert_F^2$. Similarly, the initial relative error of the matrix $\X_0$ used to initialize our algorithm is given by $\left\Vert\frac{\trace(\M)}{\tau}\X_0-\M\right\Vert_F^2\Big/\left\Vert\M\right\Vert_F^2$.
The signal-to-noise ratio is given by $\Vert \y_0\Vert^2 /\Vert\n\Vert^2=4$ (under our fixed choice of $\kappa=1/2$).

Since we cannot solve Problem \eqref{eq:recoveryProb} exactly, to verify the near-optimality of the found solution $\X^*$, we compute for it the dual-gap  which is given by
\begin{align*}
\max_{\Z\succeq0\ \trace(\Z)=\tau}{\langle\X^*-\Z,\nabla f(\X^*)\rangle} = \trace((\X^*-\tau \mathbf{v}_{n} \mathbf{v}_{n}^{\top})\nabla f(\X^*)),
\end{align*}
where $\mathbf{v}_n$ is the eigen-vector corresponding to the smallest eigenvalue of $\nabla f(\X^*)$. Note that due to convexity of the objective, the dual gap is always an upper-bound on the approximation error w.r.t. function value.

For all experiments we set the initialization matrix to
\begin{align}\label{eq:exp:init}
\X_0 = \V_r\diag\left(\Pi_{\Delta_{\tau,r}}[\diag(-\Lambda_r)]\right)\V_r^{\top},
\end{align}
where $\Delta_{\tau,r}=\lbrace \z\in\reals^r~|~ \z\ge0,\ \sum_{i=1}^r \z_i=\tau\rbrace$ is the $\tau$-simplex in $\reals^r$ and $\Pi_{\Delta_{\tau,r}}[\cdot]$ denotes the Euclidean projection over it, $\V_r\Lambda_r\V_r$ is the rank-$r$ eigen-decomposition of $-\nabla f(\tau\U\U^{\top})$, and $\U\in\reals^{n\times r}$ is produced by taking a random matrix with standard Gaussian entries and normalizing it to have a unit Frobenius norm.

In all experiments we set the sequence $\{\varepsilon_t\}_{t\geq 0}$ in Algorithm \ref{alg:LRMD} to $\varepsilon_t=\frac{4}{5}\frac{1}{(t+c+1)^2}$ where $c=10$. We set the number of measurements in the objective \eqref{eq:recoveryProb} to $m=20nr$, the number of iterations to $T=200$, and the smoothness parameter to $\beta=0.4\sqrt{rn}$. When the ground-truth matrix is  rank-$1$ or rank-$5$ we set $\tau=0.5\cdot\trace(\M)$, and for rank-$20$ we set $\tau=0.65\cdot\trace(\M)$. We note that we choose the trace bound $\tau$ to be strictly smaller than $\trace(\M)$ to prevent the optimal solution from fitting some of the noise and resulting in a higher rank matrix. For every set of parameters ($r,n$) we take the averages of $20$ i.i.d runs.  

\subsection{Rank of ground-truth matrix is known}

In our first line of experiments we assume that $\rank(\M)$ is known and we set the SVD rank parameter in Algorithm \ref{alg:LRMD} to $r=\rank(\M)$. This parameter is also used to set the initialization matrix $\X_0$.

We record several quantities of interest in Table \ref{table:results}. Importantly, Table \ref{table:results} indicates that (i.) for our recovery setting the strict complementarity condition (Assumption \ref{ass:strictcomp}) indeed holds and seems dimension-independent, and (ii.) in all cases, the convergence certification condition 
\begin{align} \label{convergenceGuaranteeExperiment}
\log\left(\frac{(n-r)\lambda_{r+1}(\Y_t)}{\varepsilon_t b_t}\right)\le 2\varepsilon_t, \quad b_t =\sum_{i=1}^n\lambda_i(\Y_t)
\end{align}
(see Section \ref{sec:certificate}) holds from the early stages of the run (right from the first iteration in most cases) and throughout  all following iterations, implying the correct convergence of our low-rank MEG method, up to negligible error.

\begin{table}\renewcommand{\arraystretch}{1.3} 
\begin{subtable}{\textwidth}
        {\small
        \begin{center}
 	    \begin{tabular}{| c | c | c | c | c | c |} \hline 
   		dimension & avg. initialization & avg. 	recovery &  avg. gap in & avg. dual & avg. first iter. \\
    ($n$) &  error & error & $\nabla f(\X^*)$ & gap & \eqref{convergenceGuaranteeExperiment} holds \\ \hline
   $100$ & $0.1593$ & $0.0439$ & $5.3665$ & $0.0065$ & $1$ \\ \hline
   $200$ & $0.1674$ & $0.0471$ & $5.2046$ & $0.0129$ & $1$ \\ \hline
   $400$ & $0.1732$ & $0.0487$ & $5.3049$ & $0.0264$ & $1$ \\ \hline
   $600$ & $0.1702$ & $0.0489$ & $5.2400$ & $0.0395$ & $1$ \\ \hline
  \end{tabular}
  \caption*{$\rank(\M)=1$}\label{table:numericalResults1}
  \end{center}
  }  
  \vskip -0.2in
\end{subtable} 

\begin{subtable}{\textwidth}
        {\small
        \begin{center}
  \begin{tabular}{| c | c | c | c | c | c |} \hline 
    dimension & avg. initialization & avg. recovery &  avg. gap in & avg. dual & avg. first iter. \\
    ($n$) &  error & error & $\nabla f(\X^*)$ & gap & \eqref{convergenceGuaranteeExperiment} holds \\ \hline
   $100$ & $0.5763$ & $0.0780$ & $5.8528$ & $0.0079$ & $1$ \\ \hline
   $200$ & $0.2072$ & $0.0561$ & $5.6697$ & $0.0136$ & $1$ \\ \hline
   $400$ & $0.1307$ & $0.0451$ & $5.5942$ & $0.0271$ & $1$ \\ \hline
   $600$ & $0.1101$ & $0.0381$ & $5.5771$ & $0.0404$ & $1$ \\ \hline
  \end{tabular}
  \caption*{$\rank(\M)=5$}\label{table:numericalResults5}
  \end{center}
}  
  \vskip -0.2in
\end{subtable}

\begin{subtable}{\textwidth}
  {\small
  \begin{center}
  \begin{tabular}{| c | c | c | c | c | c |} \hline 
    dimension & avg. initialization & avg. recovery &  avg. gap in  & avg. dual & avg. first iter. \\
    ($n$) &  error & error & $\nabla f(\X^*)$ & gap & \eqref{convergenceGuaranteeExperiment} holds \\ \hline
   $100$ & $2.2782$ & $0.0638$ & $3.3870$ & $0.0167$ & $5.7$ \\ \hline
   $200$ & $1.2007$ & $0.0437$ & $3.3210$ & $0.0156$ & $8.6$ \\ \hline
   $300$ & $0.6729$ & $0.0365$ & $3.2839$ & $0.0265$ & $11.35$ \\ \hline
  \end{tabular}
  \caption*{$\rank(\M)=20$}\label{table:numericalResults20}
  \end{center}
  }  
  \vskip -0.2in
\end{subtable}
\caption{Results for the low-rank matrix recovery from quadratic measurements problem using the Low-Rank Matrix Exponentiated Gradient Method. The fourth column is the spectral gap $\lambda_{n-r}(\nabla{}f(\X^*)) - \lambda_{n}(\nabla{}f(\X^*))$, and the last column records the iteration of the algorithm at which  the convergence certificate \eqref{convergenceGuaranteeExperiment} begins to take effect.}
\label{table:results}
\end{table}\renewcommand{\arraystretch}{1} 

We move on to compare the empirical convergence of the standard MEG algorithm, Algorithm \ref{alg:exactMEG}, and our low rank variant, Algorithm \ref{alg:LRMD}. For every run we initialize both algorithms with the same matrix (as given in Eq. \eqref{eq:exp:init}). In addition to measuring convergence in function value, we also measure the Bregman distance $\breg(\W_t,\X_t)$ between the iterates $\lbrace \X_t\rbrace_{t\ge1}$ produced by our low-rank algorithm and the iterates $\lbrace \W_t\rbrace_{t\ge1}$ produced by the standard MEG algorithm. For all graphs we set $n=200$ and $\rank(\M)=1,5,20$. The results are shown in Figure \ref{fig:comparison between methods}.

As can be seen, both the standard MEG method and our low rank variant converge very similarly in terms of the function value.  In addition, it can be seen that the Bregman distance $\breg(\W_t,\X_t)$ decays very quickly, providing additional evidence for the correct convergence of our low rank variant.

\begin{figure}
\centering
  \begin{subfigure}[t]{0.49\textwidth}
    \includegraphics[width=\textwidth]{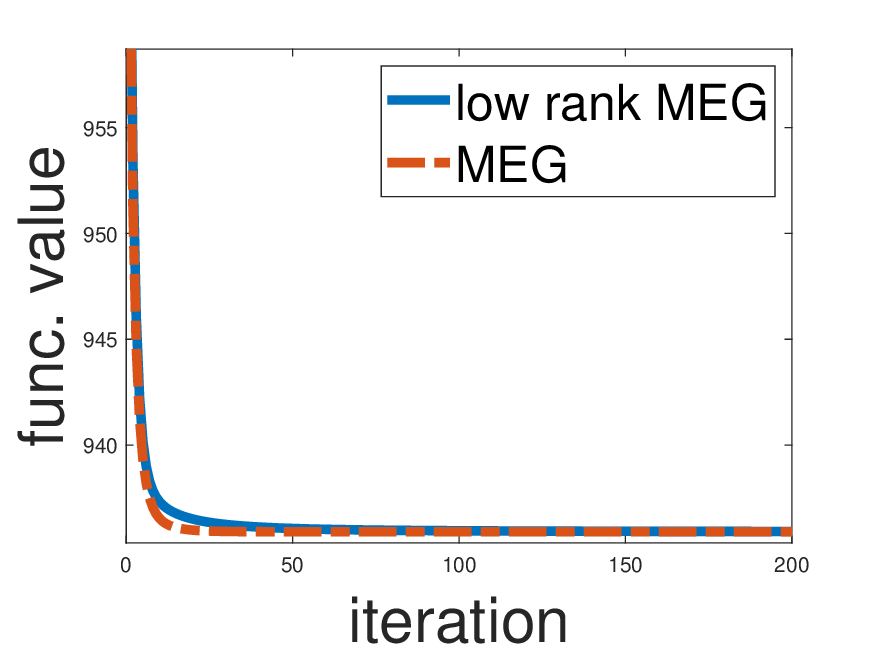}
    %\caption{}
    \label{fig:1}
  \end{subfigure}\hfil
  \begin{subfigure}[t]{0.49\textwidth}
    \includegraphics[width=\textwidth]{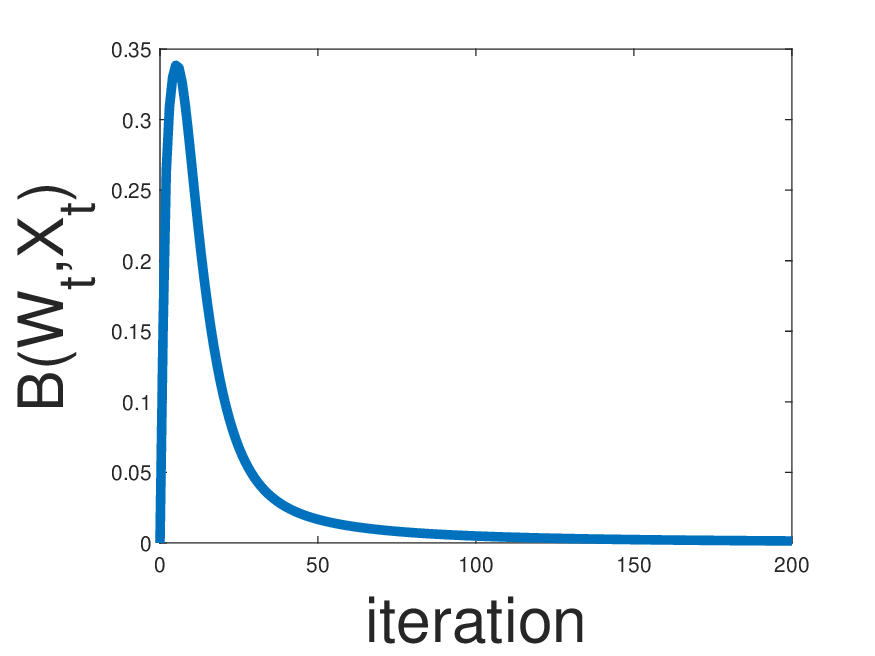}    
    %\caption{}
    \label{fig:2}
  \end{subfigure}\hfil
  \caption*{$\rank(\M)=1$}
  \label{fig:distIteration200_1}
\medskip
  \begin{subfigure}[t]{0.49\textwidth}
    \includegraphics[width=\textwidth]{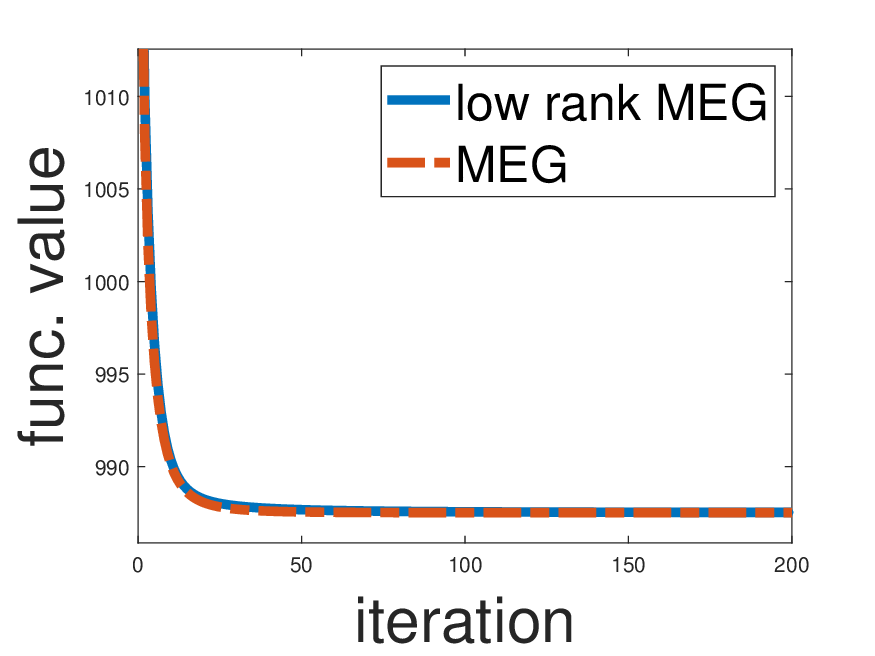}
    %\caption{}
    \label{fig:3}
  \end{subfigure}
  \begin{subfigure}[t]{0.49\textwidth}
    \includegraphics[width=\textwidth]{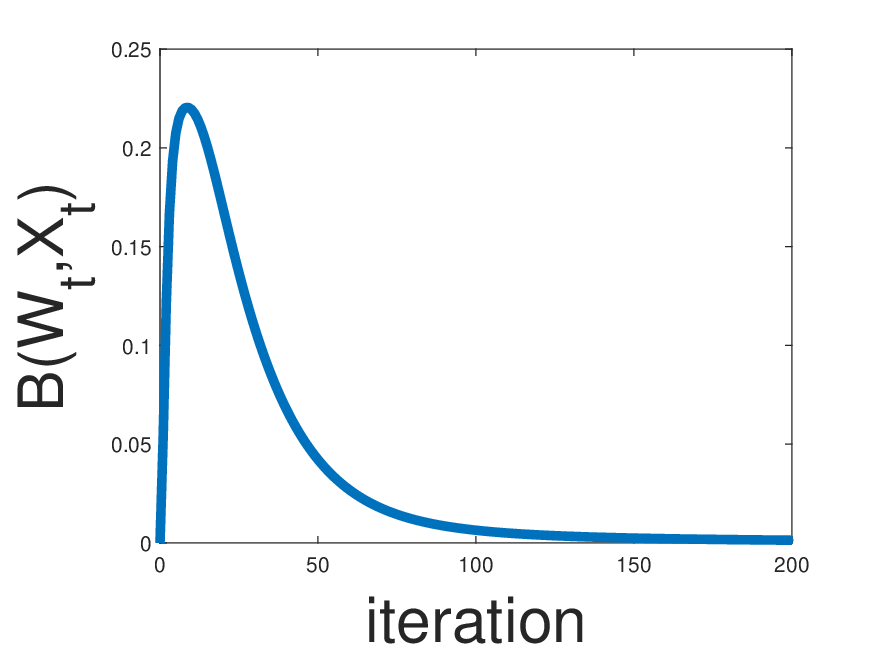}    
    %\caption{}
    \label{fig:4}
  \end{subfigure}
  \caption*{$\rank(\M)=5$}
  \label{fig:distIteration200_5}
\medskip
  \begin{subfigure}[t]{0.49\textwidth}
    \includegraphics[width=\textwidth]{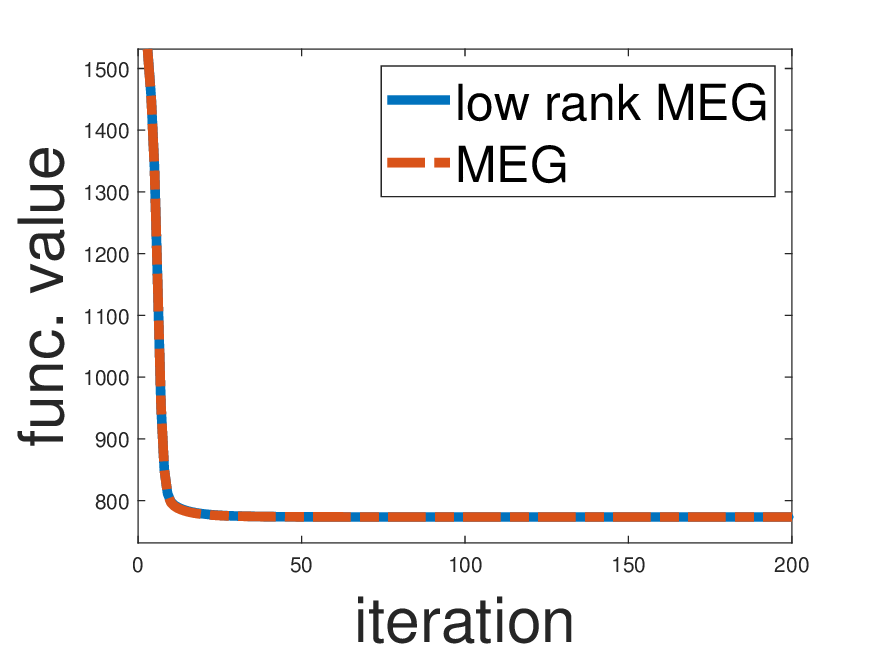}
    %\caption{}
    \label{fig:5}
  \end{subfigure}
  \begin{subfigure}[t]{0.49\textwidth}
    \includegraphics[width=\textwidth]{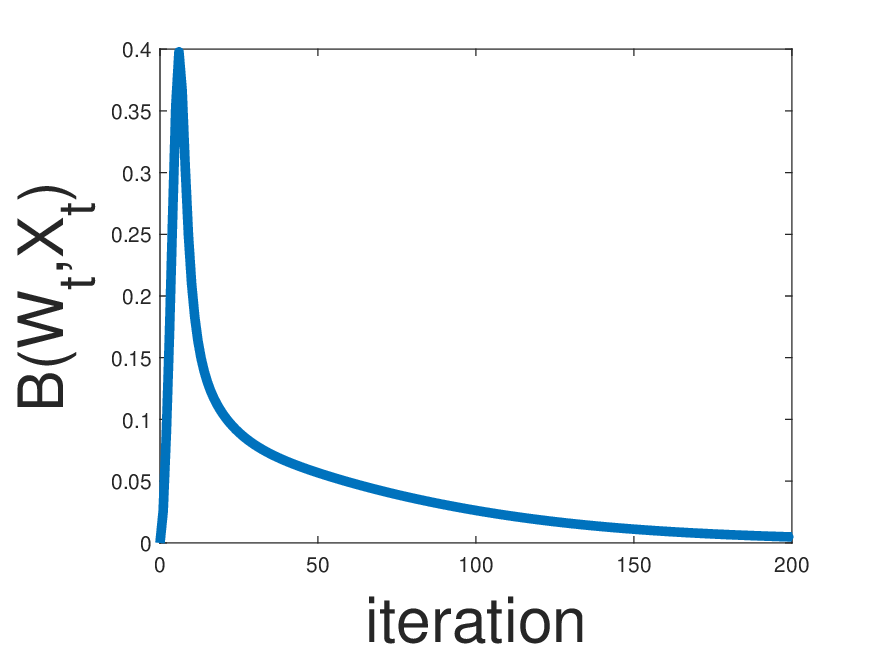}    
    %\caption{}
    \label{fig:6}
  \end{subfigure}
  \caption*{$\rank(\M)=20$}
  \caption{Comparison between the low-rank MEG algorithm and the standard MEG method for the low-rank matrix recovery from quadratic measurements problem with $n=200$. On the left we plot the convergence in function values for both variants and on the right we plot the Bregman distance between their iterates.} \label{fig:comparison between methods}
\end{figure}

Since the certificate \eqref{convergenceGuaranteeExperiment} requires computing all eigenvalues of the matrix $\Y_t$ (due to the scalar $b_t$) which makes no sense in practice when our goal is to use only low rank SVD computations, as discussed in Section \ref{sec:certificate}, we consider the harder-to-satisfy certificate
\begin{align} \label{convergenceGuaranteeExperiment2}
\log\left(\frac{(n-r)\lambda_{r+1}(\Y_t)}{\varepsilon_t b^{r+1}_t}\right)\le 2\varepsilon_t, \quad b^{r+1}_t =\sum_{i=1}^{r+1}\lambda_i(\Y_t).
\end{align}

Recall that when \eqref{convergenceGuaranteeExperiment2} holds, \eqref{convergenceGuaranteeExperiment} also holds, but not the other way around. Also, computing \eqref{convergenceGuaranteeExperiment2} only requires computing the top $r+1$ eigenvalues of $\Y_t$, i.e., increasing the SVD rank used in Algorithm \ref{alg:LRMD} by only 1. In Table \ref{table:at_bt} we provide evidence that \eqref{convergenceGuaranteeExperiment2} is indeed a reliable certificate: in all runs it  holds starting from the same iteration for which the tighter yet  computationally more expensive certificate \eqref{convergenceGuaranteeExperiment} holds.
\begin{table*}\renewcommand{\arraystretch}{1.3}
{\small
\begin{center}
  \begin{tabular}{| c | c | c | c |} \hline 
      & avg. first iter.  & avg. first iter. \\ 
   $\rank(\M)$  & $\log\left(\frac{(n-r)\lambda_{r+1}(\Y_t)}{\varepsilon_tb_t}\right)\le 2\varepsilon_t$ holds & $\log\left(\frac{(n-r)\lambda_{r+1}(\Y_t)}{\varepsilon_tb^{r+1}_t}\right)\le 2\varepsilon_t$ holds\\ \hline
   $1$ & $1$  & $1$\\ \hline
   $5$ & $1$ & $1$\\ \hline
   $20$ & $8.6$ & $8.6$\\ \hline
  \end{tabular}
  \caption{Comparison between the first iteration the convergence guarantee \eqref{convergenceGuaranteeExperiment} holds vs. the first iteration the weaker condition $\log\left(\frac{(n-r)\lambda_{r+1}(\Y_t)}{\varepsilon_tb^{r+1}_t}\right)\le 2\varepsilon_t$ holds for $n=200$. 
  }\label{table:at_bt}
\end{center}
}
\vskip -0.2in
\end{table*}\renewcommand{\arraystretch}{1}

\subsection{Rank of ground-truth matrix is overestimated}
In a second line of experiments we examine the case in which $\rank(\M)$ --- the rank of the ground-truth matrix is not precisely known, but only overestimated. That is, in the following experiments, both for the SVD computations and for the initialization in Algorithm \ref{alg:LRMD}, we use rank parameter $r > \rank(\M)$.

We run all experiments with $n=200$. As can be seen in Table \ref{table:bigR}, the convergence guarantee of \eqref{convergenceGuaranteeExperiment} begins to hold starting from a very early stage.
In Figure \ref{fig:comparison between methods bigR} it can be seen that the results using $r>\rank(\M)$ are very similar to the previously examined case $r = \rank(\M)$, and that the standard MEG method and our low-rank variant perform very similarly also in this case. 
\begin{table*}[!htb]\renewcommand{\arraystretch}{1.3}
{\small
\begin{center}
  \begin{tabular}{| c | c | c | c |} \hline 
   $\rank(\M)$  & SVD \rank & avg. first iteration \eqref{convergenceGuaranteeExperiment} holds\\ \hline
   $1$ & $3$  & $2$\\ \hline
   $5$ & $10$ & $3.85$\\ \hline
   $20$ & $30$ & $9.25$\\ \hline
  \end{tabular}
  \caption{The first iteration the convergence certificate \eqref{convergenceGuaranteeExperiment} holds when using $r>\rank(\M)$ for $n=200$. 
  }\label{table:bigR}
\end{center}
}
\vskip -0.2in
\end{table*}\renewcommand{\arraystretch}{1} 

\begin{figure}
\centering
  \begin{subfigure}[t]{0.45\textwidth}
    \includegraphics[width=\textwidth]{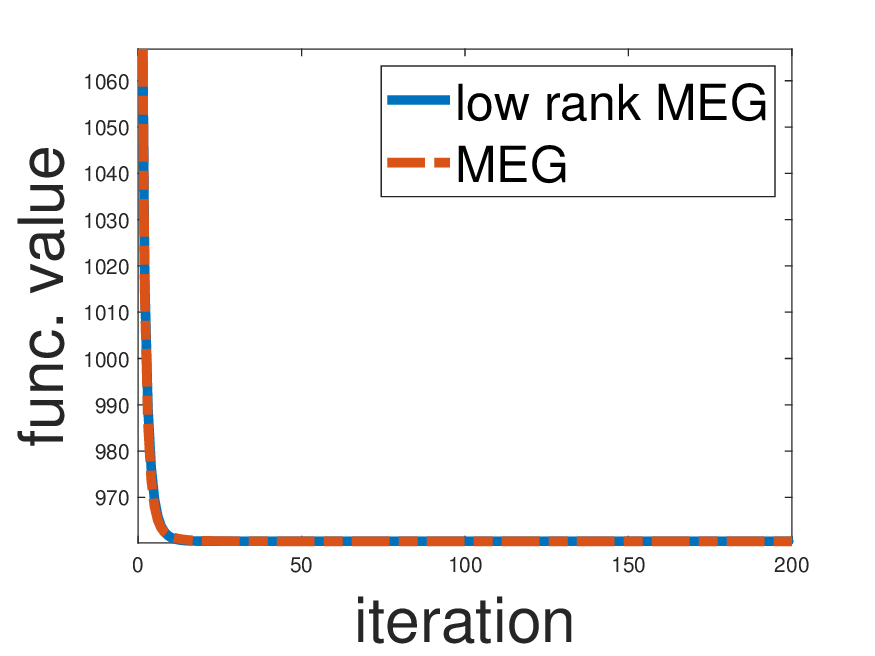}
    %\caption{}
    \label{fig:12}
  \end{subfigure}\hfil
  \begin{subfigure}[t]{0.45\textwidth}
    \includegraphics[width=\textwidth]{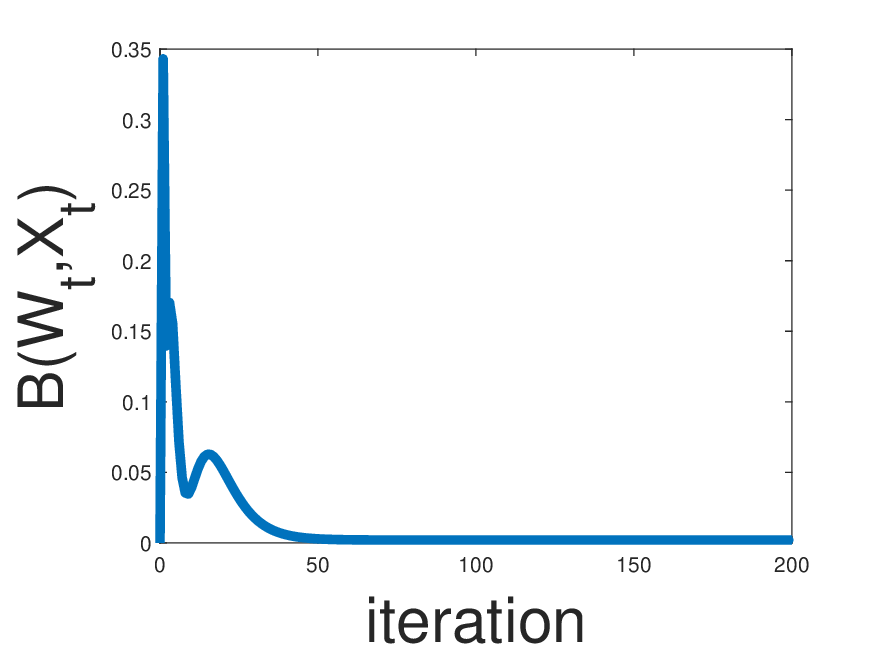}    
    %\caption{}
    \label{fig:22}
  \end{subfigure}\hfil
  \caption*{$\rank(\M)=1$, SVD rank = $3$}
  \label{fig:bigRdistIteration200_1}
\medskip
  \begin{subfigure}[t]{0.45\textwidth}
    \includegraphics[width=\textwidth]{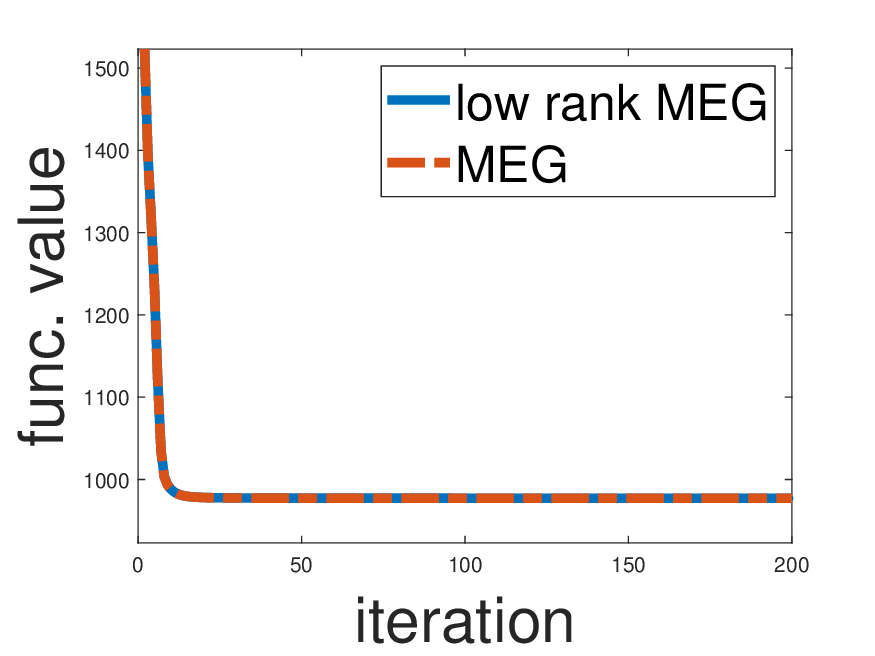}
    %\caption{}
    \label{fig:32}
  \end{subfigure}
  \begin{subfigure}[t]{0.45\textwidth}
    \includegraphics[width=\textwidth]{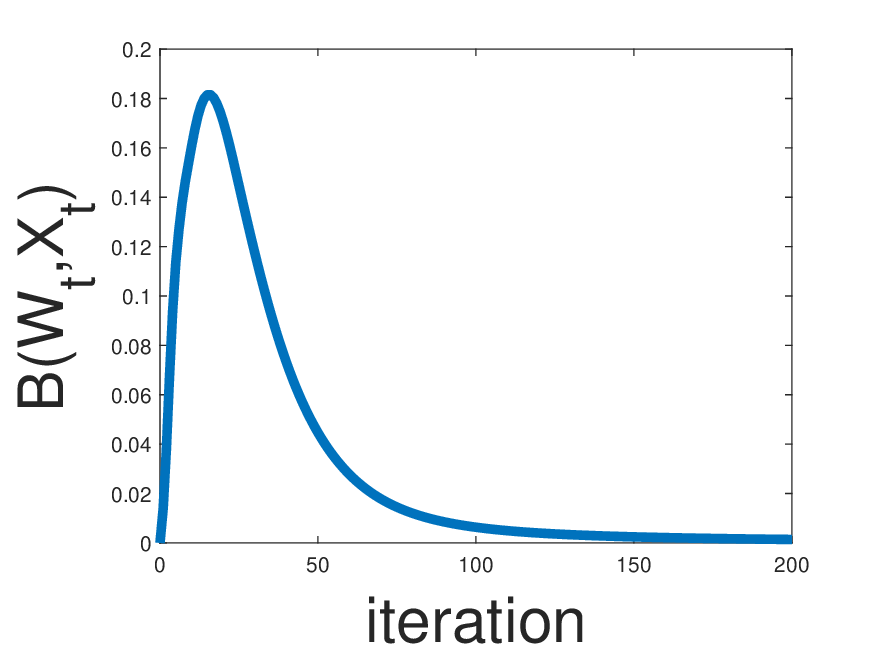}    
    %\caption{}
    \label{fig:42}
  \end{subfigure}
  \caption*{$\rank(\M)=5$, SVD rank = $10$}
  \label{fig:bigRdistIteration200_5}
\medskip
  \begin{subfigure}[t]{0.45\textwidth}
    \includegraphics[width=\textwidth]{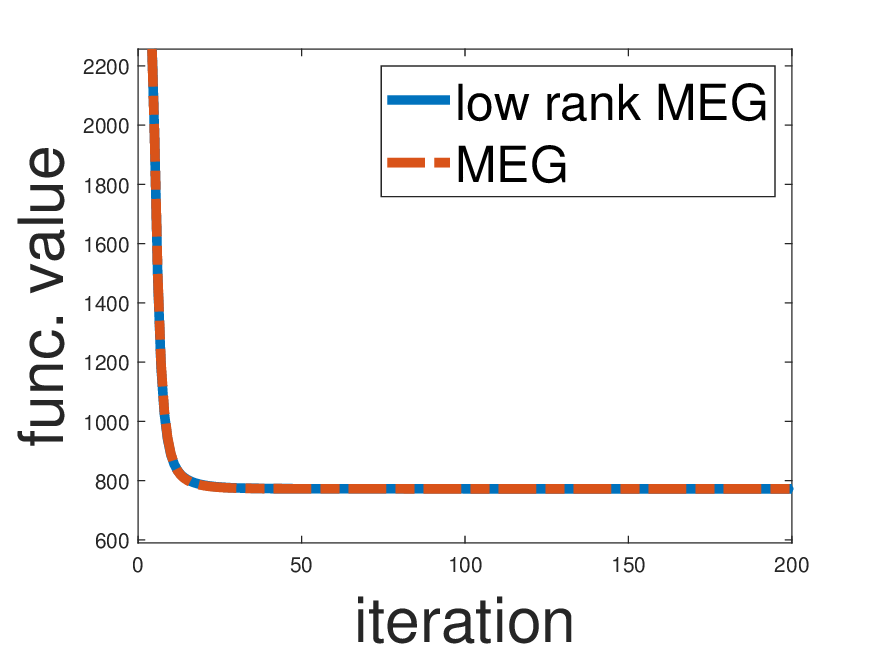}
    %\caption{}
    \label{fig:52}
  \end{subfigure}
  \begin{subfigure}[t]{0.45\textwidth}
    \includegraphics[width=\textwidth]{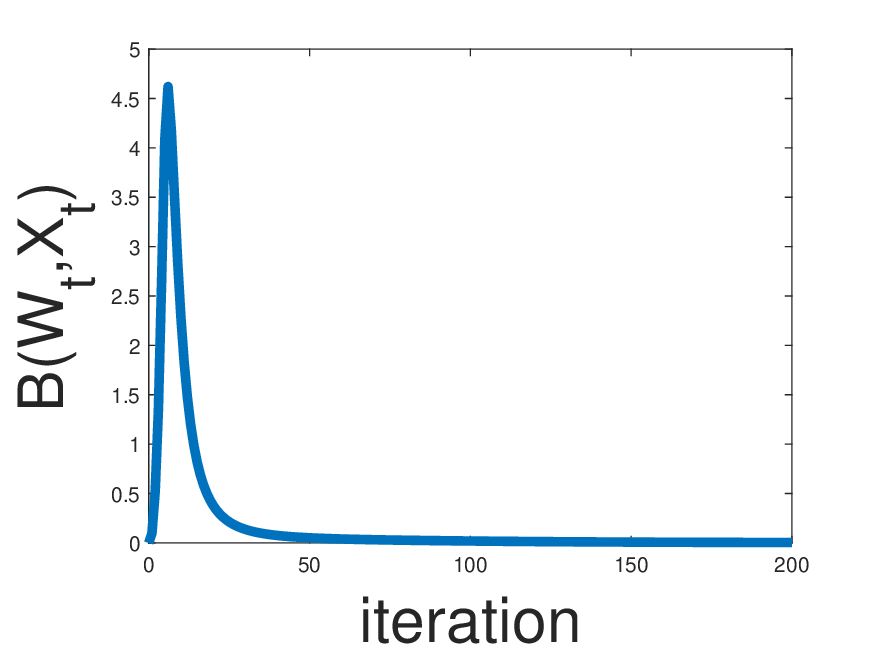}    
    %\caption{}
    \label{fig:62}
  \end{subfigure}
  \caption*{$\rank(\M)=20$, SVD rank = $30$}
  \caption{Comparison between the low-rank MEG and the standard MEG for the low-rank matrix recovery from quadratic measurements problem where $n=200$ and $r>\rank(\M)$. On the left we plot the convergence in function values for both variants and on the right we plot the Bregman distance between their iterates.} \label{fig:comparison between methods bigR}
\end{figure}

\subsection{Higher-rank measurements}
In a third line of experiments we consider a small change to Problem \eqref{eq:recoveryProb} where we use measurement matrices of rank greater than $1$. Formally, we consider the problem
\begin{align*} %\label{eq:recoveryProbHighRank}
\min_{\X\succeq0\ \trace(\X)=\tau}\left\lbrace f(\X):= \frac{1}{2}\sum_{i=1}^m \left(\trace(\A_i^{\top}\X \B_i)-\y_i\right)^2\right\rbrace,
\end{align*}
for matrices $\A_i,\B_i\in\reals^{n\times p}, i=1,\dots,m, p > 1$, which are chosen to have orthonormal columns. 

Note that already in the case of a single measurement, i.e., $m=1$, a simple calculation shows that the smoothness parameter w.r.t. the Frobenius norm (relevant to Euclidean gradient methods) of $f(\cdot)$  is upper-bounded by $\Vert{\A_1\B_1^{\top}}\Vert_F^2 = p$, while  the smoothness parameter w.r.t. the spectral norm (relevant to the MEG method) is upper-bounded only by $\Vert{\A_1\B_1^{\top}}\Vert_2^2 = 1$\footnote{Indeed standard algebraic manipulations yield that for any $\X,\Y\in\mbS^n$, $\Vert{\nabla{}f(\X) - \nabla{}f(\Y)}\Vert_F = \Vert{\langle{\X-\Y,\A_1\B_1^{\top}}\rangle\A_1\B_1^{\top}}\Vert_F \leq \Vert{\Vert{\X-\Y}\Vert_F\Vert{\A_1\B_1^{\top}}\Vert_F\A_1\B_1^{\top}}\Vert_F = \Vert{\A_1\B_1^{\top}}\Vert_F^2\Vert{\X-\Y}\Vert_F$, and similarly, $\Vert{\nabla{}f(\X) - \nabla{}f(\Y)}\Vert_2\leq \Vert{\Vert{\X-\Y}\Vert_*\Vert{\A_1\B_1^{\top}}\Vert_2\A_1\B_1^{\top}}\Vert_2 = \Vert{\A_1\B_1^{\top}}\Vert_2^2\Vert{\X-\Y}\Vert_*$.}, and thus, at least in terms of worst-case gradient complexity and for large values of $p$, MEG can be significantly faster than its Euclidean counterpart, which makes this setting in particular interesting in our context.

We set the rank of the ground truth matrix to $\rank(\M)=5$,  the SVD rank parameter to $r=\rank(\M)$, the smoothness parameter to $\beta=12\sqrt{rn}$, the constant $c=100$, the dimension to $n=200$, and the number of iterations to $T=400$. The rank of the measurements are set to either $p=10$ or $p=50$. All other parameters are as in the previous experiments.

As can be seen in Table \ref{table:resultsHighRankOperators}, strict complementarity holds in this case too and the convergence guarantee condition of \eqref{convergenceGuaranteeExperiment} holds from early on. Also, similarly to the previous experiments, as can be seen in Figure \ref{fig:highRankOperators}, our low-rank MEG method performs very similarly to the standard MEG method.

\begin{table}\renewcommand{\arraystretch}{1.3} 
        {\small
        \begin{center}
 	    \begin{tabular}{| c | c | c | c | c | c |} \hline 
   		$\rank(\A_i)=$ & avg. initialization & avg. 	recovery &  avg. gap in & avg. dual & avg. first iter. \\
    $\rank(\B_i)$ &  error & error & $\nabla f(\X^*)$ & gap & \eqref{convergenceGuaranteeExperiment} holds \\ \hline 
   $10$ & $1.5218$ & $0.0539$ & $56.6864$ & $0.0244$ & $18.4$ \\ \hline
      $50$ & $3.0973$ & $0.0505$ & $284.1826$ & $0.1236$ & $4.5$ \\ \hline
  \end{tabular}
  %\caption*{$\rank(\M)=5$, $\kappa=0.5$, $\beta=30\cdot0.4\sqrt{rn}$, $c=100$, $n=200$, $T=400$.}
  \caption{Results for the low-rank matrix recovery from quadratic measurements problem using the Low-Rank Matrix Exponentiated Gradient Method with operators of rank greater than $1$. The fourth column is the spectral gap $\lambda_{n-r}(\nabla{}f(\X^*)) - \lambda_{n}(\nabla{}f(\X^*))$, and the last column records the iteration of the algorithm at which  the convergence certificate \eqref{convergenceGuaranteeExperiment} begins to take effect.}
\label{table:resultsHighRankOperators}
  \end{center}
  }  
  \vskip -0.2in
\end{table}\renewcommand{\arraystretch}{1}

\begin{figure}
\centering
\caption*{$p=\rank(\A_i)=\rank(\B_i)=10$}
\begin{subfigure}[t]{0.45\textwidth}
    \includegraphics[width=\textwidth]{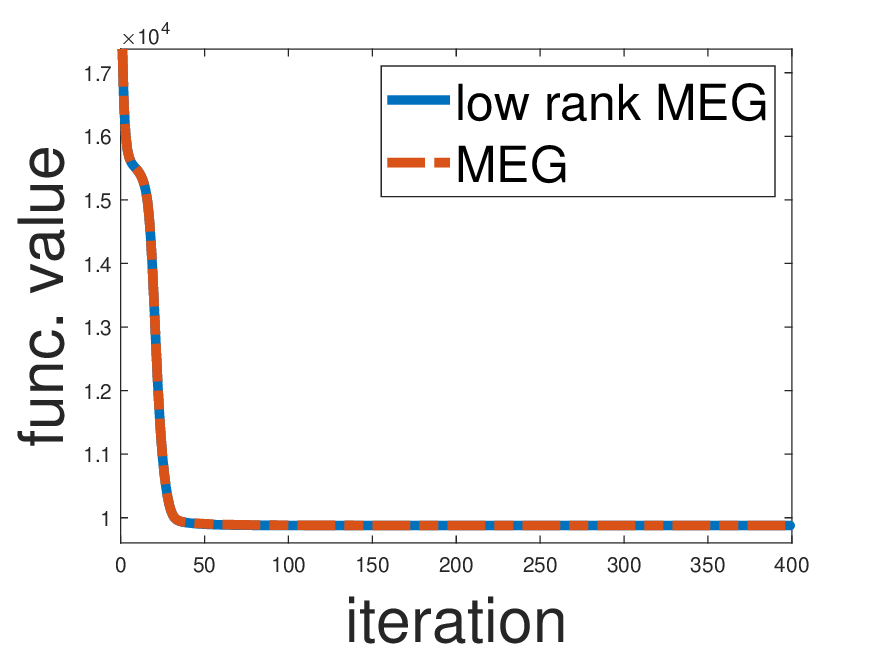}
    %\caption{}
    \label{fig:71}
  \end{subfigure}
  \begin{subfigure}[t]{0.45\textwidth}
    \includegraphics[width=\textwidth]{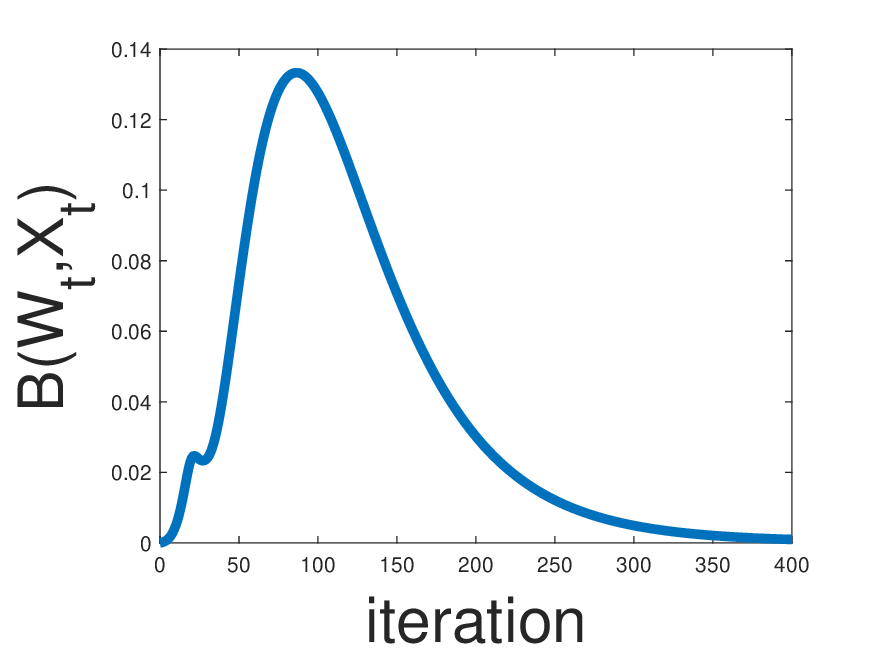}    
    %\caption{}
    \label{fig:72}
  \end{subfigure}
  \medskip
  \caption*{$p=\rank(\A_i)=\rank(\B_i)=50$}
  \begin{subfigure}[t]{0.45\textwidth}
    \includegraphics[width=\textwidth]{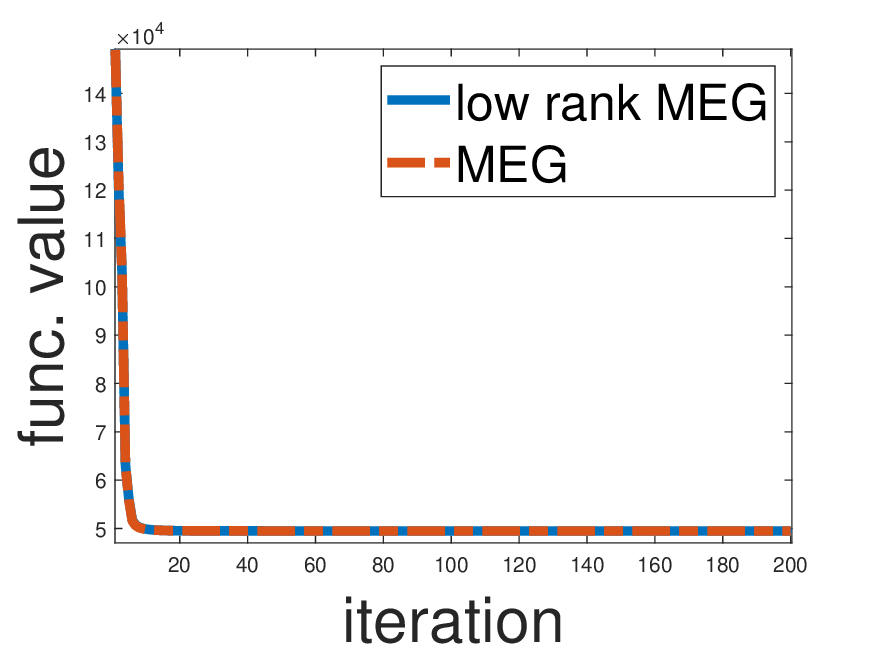}
    %\caption{}
    \label{fig:5}
  \end{subfigure}
  \begin{subfigure}[t]{0.45\textwidth}
    \includegraphics[width=\textwidth]{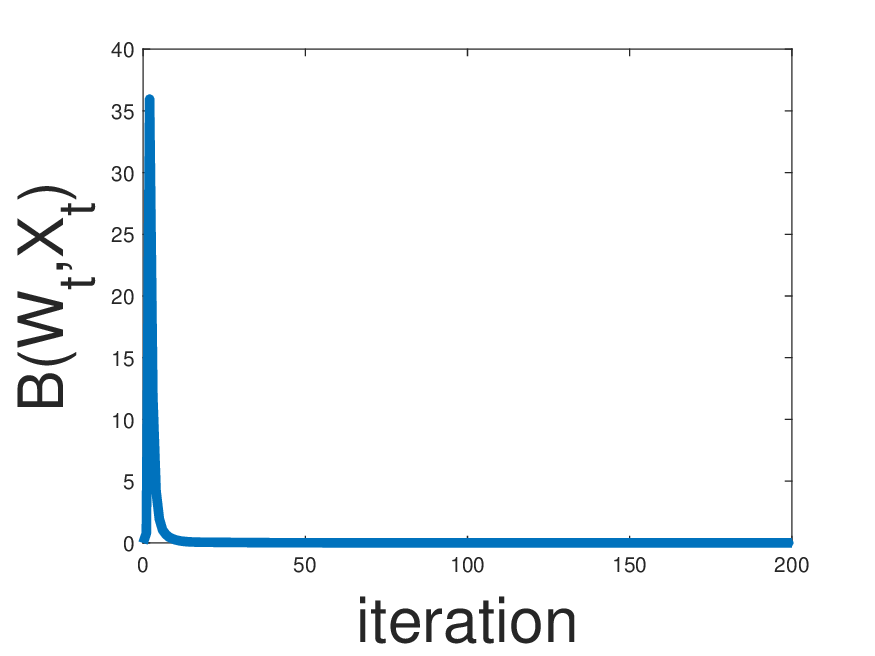}    
    %\caption{}
    \label{fig:6}
  \end{subfigure}
  \caption{Comparison between the low-rank MEG algorithm and the standard MEG method for the low-rank matrix recovery from quadratic measurements problem with $n=200$ and measurements of rank $p>1$. On the left we plot the convergence in function values for both variants and on the right we plot the Bregman distance between their iterates.} \label{fig:highRankOperators}
\end{figure}

\subsection{Performance of low-rank MEG without strict complementarity}
In the fourth and final line of experiments we empirically examine the necessity of the strict complementarity assumption to the convergence of our low-rank MEG method. We consider the same setup of Problem \eqref{eq:recoveryProb}, but this time with noiseless measurements, that is, we set $\kappa=0$. We additionally set the trace parameter to $\tau=\trace(\M)$. Importantly, these choices guarantee that the ground-truth matrix $\M$ is a $\rank(\M)$-optimal solution of Problem \eqref{eq:recoveryProb}  and that $\nabla{}f(\M)=\mathbf{0}$, which in turn implies that strict complementarity does not hold. 

We consider a simplified setup in which we set set the dimension to $n=200$, the number of measurements to $m=1$, the rank of the ground-truth matrix to $r=\rank(\M)=1$, and we set the step-size to $0.03$. The rest of the parameters are the same as in the previous experiments. We initialize both the standard MEG and low-rank MEG methods with the same matrix as in \eqref{eq:exp:init}.  As before, the results are the averages of 20 i.i.d. runs.

As can be seen in the left panel of Figure \ref{fig:necessity of strict complementarity}, while the standard MEG method converges extremely fast, the low-rank MEG method does not converge at all for many iterations and remains ``stuck'' at roughly the same function value until eventually managing to reduce the function value and converge. Moreover, the right panel of Figure \ref{fig:necessity of strict complementarity} shows that, as opposed to the previous experiments (in which strict complementarity hold), the distances between the iterates of the two methods grow over time, which suggests that they converge to very different solutions. Thus, this experiment not only exhibits the potentially poor performance of the low-rank MEG method without strict complementarity, but it also empirically demonstrates the neccessity of this assumption to our theoretical analysis, which is based on establishing that the low-rank MEG updates approximate sufficiently well their exact full-rank counterparts.

\begin{figure}
\centering
\begin{subfigure}[t]{0.45\textwidth}
    \includegraphics[width=\textwidth]{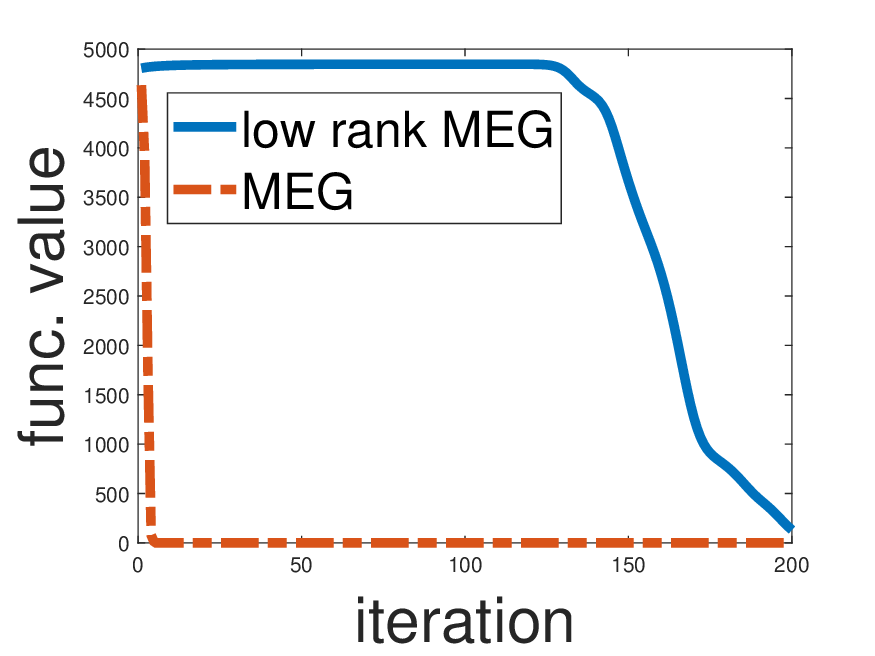}
    %\caption{}
    \label{fig:71}
  \end{subfigure}
  \begin{subfigure}[t]{0.45\textwidth}
    \includegraphics[width=\textwidth]{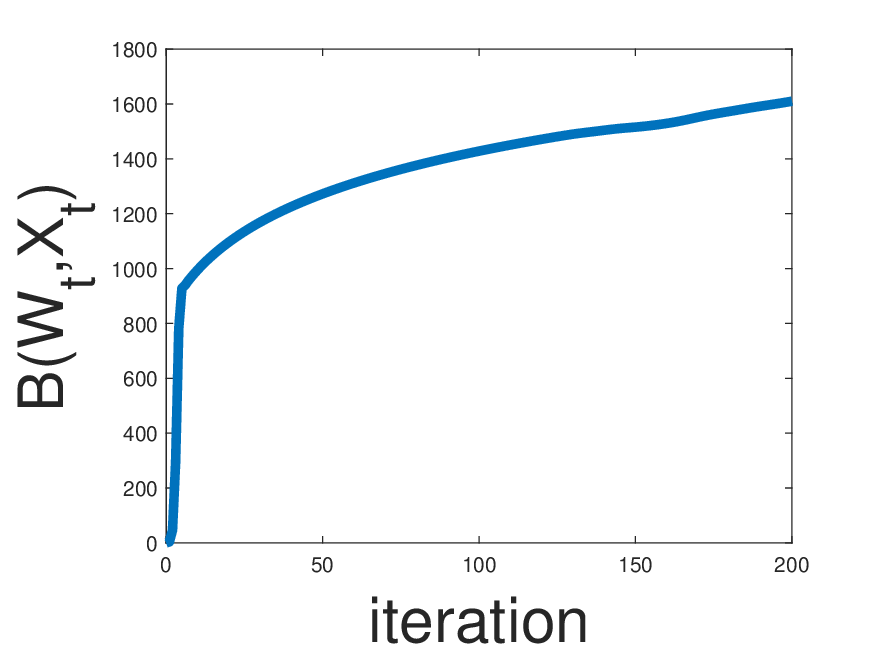}    
    %\caption{}
    \label{fig:72}
  \end{subfigure}
  \caption{Comparison between the low-rank MEG and the standard MEG for the low-rank matrix recovery from quadratic measurements problem where $n=200$ and $r=\rank(\M)=1$ with $\kappa=0$ and $\tau=\trace(\M)$. On the left we plot the convergence in function values for both variants and on the right we plot the Bregman distance between their iterates.} \label{fig:necessity of strict complementarity}
\end{figure}

\bibliographystyle{plain}
\bibliography{bilbli}

\appendix
\section{The Davis-Kahan $\sin\theta$ Theorem.} 
Below we state and prove a variant of the Davis-Kahan $\sin{\theta}$ theorem which is particularly useful for our analysis.
  
    \begin{lemma} \label{lemma:DavisKahan}
  Let $\X\in\mathcal{S}_n$, denote $r=\rank(\X)$, and let us write its eigen-decomposition as $\X=\V_r{\Lambda_r}{\V_r}^{\top}$, where $\V_r\in\reals^{n\times{}r}$, $\Lambda_r\in\reals^{r\times r}$. Let $\Z\in\mathcal{S}_n$ and let $\W_{r}\in\reals^{n\times{}r}$ be a matrix whose columns are the $r$ leading eigenvectors of $\Z$. Then,
  \begin{align*}
\Vert\W_{r}\W_{r}^{\top}-\V_r{\V_r}^{\top}\Vert_F \le \frac{2\Vert\Z-\X\Vert_F}{\lambda_{r}(\X)}.
 \end{align*}
  \end{lemma}
  
  \begin{proof}
 Write the eigen-decomposition of $\X$ as $\X=\V_r{\Lambda}{\V_r}^{\top}=\sum_{i=1}^{r}\lambda_i(\X){\mathbf{v}_i}{\mathbf{v}_i}^{\top}$. Then, we have that
  \begin{align} \label{ineq:DK1}
  \lambda_{r}(\X)[r-\langle\W_{r}\W_{r}^{\top},\V_r{\V_r}^{\top}\rangle] & = \lambda_{r}(\X)\sum_{i=1}^{r}[1-{\mathbf{v}_i}^{\top}(\W_{r}\W_{r}^{\top}){\mathbf{v}_i}] \nonumber
  \\ & \underset{(a)}{\le} \sum_{i=1}^{r}\lambda_i(\X)[1-{\mathbf{v}_i}^{\top}(\W_{r}\W_{r}^{\top}){\mathbf{v}_i}] \nonumber
  \\ & = \sum_{i=1}^{r}\lambda_i(\X) - \sum_{i=1}^{r}\lambda_i(\X){\mathbf{v}_i}^{\top}(\W_{r}\W_{r}^{\top}){\mathbf{v}_i} \nonumber
  \\ & =\langle \V_r{\V_r}^{\top}-\W_{r}\W_{r}^{\top},\X\rangle,
  \end{align}
  where (a) follows since $\W_{r}\W_{r}^{\top}\preceq\I$ which implies that ${\mathbf{v}_i}^{\top}(\W_{r}\W_{r}^{\top}){\mathbf{v}_i}\le {\mathbf{v}_i}^{\top}{\mathbf{v}_i}=1$.

Therefore, it holds that  
   \begin{align} \label{ineq:DK2}
  \Vert\W_{r}\W_{r}^{\top}-\V_r{\V_r}^{\top}\Vert_F^2 & = \Vert\W_{r}\W_{r}^{\top}\Vert_F^2 + \Vert\V_r{\V_r}^{\top}\Vert_F^2 - 2\langle\W_{r}\W_{r}^{\top},\V_r{\V_r}^{\top}\rangle \nonumber
  %\\ & = \trace(\W_{r}\W_{r}^{\top}\W_{r}\W_{r}^{\top}) + \trace(\V_r{\V_r}^{\top}\V_r{\V_r}^{\top}) - 2\trace(\W_{r}\W_{r}^{\top}\V_r{\V_r}^{\top}) \nonumber
 % \\ & = \trace(\W_{r}\W_{r}^{\top}) + \trace(\V_r{\V_r}^{\top}) - 2\trace(\W_{r}\W_{r}^{\top}\V_r{\V_r}^{\top}) \nonumber
   \\ & = 2r - 2\langle\W_{r}\W_{r}^{\top},\V_r{\V_r}^{\top}\rangle \nonumber
   \\ & \le  \frac{2}{\lambda_{r}(\X)}\langle \V_r{\V_r}^{\top}-\W_{r}\W_{r}^{\top},\X\rangle,
 \end{align}
  where the last inequality follows from \eqref{ineq:DK1}.
 In addition, it can be seen that
    \begin{align} \label{ineq:DK3}
\langle \V_r{\V_r}^{\top}-\W_{r}\W_{r}^{\top},\X\rangle & = \langle \V_r{\V_r}^{\top},\X\rangle - \langle\W_{r}\W_{r}^{\top},\Z\rangle +\langle\W_{r}\W_{r}^{\top},\Z-\X\rangle \nonumber
\\ & \le \langle \V_r{\V_r}^{\top},\X\rangle - \langle\V_r{\V_r}^{\top},\Z\rangle +\langle\W_{r}\W_{r}^{\top},\Z-\X\rangle \nonumber
\\ & \le \Vert\W_{r}\W_{r}^{\top}-\V_r{\V_r}^{\top}\Vert_F\Vert\Z-\X\Vert_F,
 \end{align}
 where the last inequality follows from the Cauchy-Schwarz inequality.
 Finally, plugging \eqref{ineq:DK3} into the RHS of \eqref{ineq:DK2}, we indeed obtain that
 \begin{align*}
\Vert\W_{r}\W_{r}^{\top}-\V_r{\V_r}^{\top}\Vert_F \le \frac{2\Vert\Z-\X\Vert_F}{\lambda_{r}(\X)}.
 \end{align*}
 \end{proof}

\end{document}